\documentclass[a4paper,11pt,openany]{amsart}
\usepackage[centertags]{amsmath}
\usepackage{amscd}
\usepackage{amsthm}
\usepackage{amssymb}
\usepackage{array}
\usepackage{enumerate}
\usepackage{hyperref}
\usepackage{cleveref}
\usepackage{multicol}
\usepackage{centernot}
\usepackage{mathtools}
\usepackage{stmaryrd}
\usepackage{tikzsymbols}
\usepackage{textcomp}
\usepackage{parskip}
\usepackage{mathrsfs}
\usepackage{tikz}
\usetikzlibrary{cd}

\usepackage[utf8]{inputenc}
\usepackage[T1]{fontenc}

\usepackage[english]{babel}
\usepackage{color}
\usepackage{tikz}
\usepackage{indentfirst} 
\usepackage[T1]{fontenc}
\usepackage[sc]{mathpazo}
\usepackage[T1]{fontenc}
\usepackage[sc]{mathpazo}

\usepackage{geometry}
\usepackage{subfiles}
\usepackage{csquotes}
\usepackage{fancyhdr}

\newtheorem{defn}{Definition}[section]
\newtheorem{prop}[defn]{Proposition}
\newtheorem{thm}[defn]{Theorem}
\newtheorem{lem}[defn]{Lemma}
\newtheorem{cor}[defn]{Corollary}
\newtheorem{ex}[defn]{Example}
\newtheorem{rem}[defn]{Remark}
\newtheorem{conj}[defn]{Conjecture}
\newtheorem{ques}[defn]{Question}
\newtheorem{nota}[defn]{Notation}

\newcommand{\Z}{\mathbb{Z}}
\newcommand{\Q}{\mathbb{Q}}
\newcommand{\R}{\mathbb{R}}

\renewcommand{\ker}{\operatorname{Ker}}

\newcommand{\Hom}{\operatorname{Hom}}
\newcommand{\Isom}{\operatorname{Isom}}
\newcommand{\Homeo}{\operatorname{Homeo}}

\newcommand{\Cov}{\operatorname{Cov}}

\newcommand{\Out}{\operatorname{Out}}
\newcommand{\rank}{\operatorname{rank}}
\newcommand{\Aut}{\operatorname{Aut}}
\newcommand{\Inn}{\operatorname{Inn}}
\newcommand{\SO}{\operatorname{SO}}

\newcommand{\Gl}{\operatorname{GL}}

\newcommand{\Zc}{\mathcal{Z}}

\newcommand{\D}{\operatorname{disc-sym}}
\newcommand{\A}{\operatorname{tor-sym}}
\newcommand{\acts}{\curvearrowright}
\newcommand{\Stab}{\operatorname{Stab}}

\newcommand{\lenght}{\operatorname{lenght}}

\begin{document}
	
	\title{Large and iterated finite group actions on manifolds admitting non-zero degree maps to nilmanifolds}
	\author{Jordi Daura Serrano}
	\address{Jordi Daura Serrano, Department de Màtematiques i Informàtica, Universitat de Barcelona (UB), Gran Via de les Corts Catalanes 585, 08007 Barcelona (Spain)}
	\email{jordi.daura@ub.edu}

	\pagenumbering{arabic}
	\maketitle
	\begin{abstract}
	Let $M$ be a closed connected oriented manifold admitting a non-zero degree map to a nilmanifold $f:M\longrightarrow N/\Gamma$. In the first part of the paper, we study effective finite group actions on $M$. In particular, we prove that $\Homeo(M)$ is Jordan, we bound the discrete degree of symmetry $\D(M)$ of $M$ (introduced in \cite{mundet2021topological}) and we study the number and the size of the stabilizers of an effective action of a finite group $G$ on $M$. We also study the toral rank conjecture and the Carlsson's conjecture for large primes for this class of manifolds. In the second part of the paper, we introduce the concepts of free iterated action of groups and iterated discrete degree of symmetry $\D_2(M)$, which refines $\D(M)$. We prove that if $N/\Gamma$ is a 2-step nilmanifold and $\D_2(M)=\D_2(N/\Gamma)$ then $H^*(M,\Q)\cong H^*(N/\Gamma,\Q)$.
	\end{abstract}
	\date{}

\section{Introduction}

A fundamental problem in the theory of compact transformation groups is to determine which finite groups $G$ act effectively on a given manifold $M$. To solve this problem in full generality for an arbitrary manifold is extremely difficult with the current tools. A way to modify this problem is the following; instead of studying the properties of an action of a finite group $G$ on $M$, we study the properties of the restriction of this action to a suitable subgroup $H$ of $G$ of index bounded above by a constant $C$ only depending on $M$. Let us recall some definitions and problems which follow this philosophy, compiled in the recent survey \cite{riera2023actions}. 

\begin{defn}\label{large finite group actions def}
\begin{itemize}
	\item[1.](Jordan property) A group $\mathcal{G}$ is Jordan if there exists a constant $C$ such that every finite subgroup $G\leq\mathcal{G}$ has an abelian subgroup $A$ such that $[G:A]\leq C$.
	\item[2.](Discrete degree of symmetry) Given a manifold $M$, let 
	$$\mu(M)=\{r\in\mathbb{N}:\text{ $M$ admits an effective action of $(\Z/a)^r$ for arbitrarily large $a$}\}.$$
	
	More explicitly, $r\in \mu(M)$ if there exists an increasing sequence of natural number $\{a_i\}$ and effective group actions of $(\Z/a_i)^r$ on $M$ for each $i$. 
	
	The discrete degree of symmetry of a manifold $M$ is
	$$ \D(M)=\max(\{0\}\cup\mu(M)).$$
	\item[3.](Minkowski property) A group $\mathcal{G}$ is Minkowski if there exists a constant $C$ such that every finite subgroup $G\leq\mathcal{G}$ satisfies $|G|\leq C$. Given a manifold $M$, if $\Homeo(M)$ is Minkowski then we say that $M$ is almost asymmetric. If $M$ does not admit effective actions of non-trivial finite groups, we say that $M$ is asymmetric.
	\item[4.](Stabilizers of an action on a manifold) Given a manifold $M$ and a group $G$ acting effectively on $M$, we define the set of stabilizer subgroups of the action of $G$ on $M$ as
	$$\Stab(G,M)=\{G_x:x\in M\}.$$
	We say that $M$ has few stabilizers if there exists a constant $C$ such that every a finite group $G$ acting effectively on $M$ has a subgroup $H\leq G$ such that $[G:H]\leq C$ and $|\Stab(H,M)|\leq C$.
\end{itemize}
\end{defn}

\begin{rem}\label{Minkowski lemma}
The name of the Minkowski property is motivated by a classical result of Hermann Minkowski which states that $\Gl(n,\Z)$ is Minkowski (see \cite[\S 1]{serre2010bounds}). The Minkowski property was studied in \cite{popov2011makar,prokhorov2014jordan,popov2018jordan,golota2023finite,bandman2024jordan} under the name of bounded finite subgroups property.
\end{rem}

\begin{ques}\label{large finite group actions questions}
Let $M$ a closed connected $n$-dimensional manifold.
\begin{itemize}
	\item[(1)] \cite[\S 2]{riera2023actions} When is $\Homeo(M)$ Jordan?
	\item[(2)] \cite[Question 3.4, Question 3.5]{riera2023actions} Is $\D(M)\leq n$? If $\D(M)=n$, is $M$ necessarily homeomorphic to $T^n$?
	\item[(3)] \cite[Question 8.3]{riera2023actions} When is $M$ almost asymmetric?
	\item[(4)] \cite[Question 12.2]{riera2023actions} Does every closed connected manifold have few stabilizers?
\end{itemize}
\end{ques}

Let us give some historical context about these four questions. \Cref{large finite group actions questions}.(1) was asked around 30 years ago for the diffeomorphism group of closed connected smooth manifolds by Étienne Ghys. It has been answered affirmatively for many  manifolds (see \cite[\S 2]{riera2023actions}). However, it has been shown that there are closed connected manifolds whose diffeomorphism group is not Jordan, the simplest one being $S^2\times T^2$ (see \cite{csikos2014diffeomorphism,riera2017non,szabo2019special,szabo2023constructing,zarhin2014theta} for other examples of non-Jordan groups of homeomorphism). 
There is no general method to determine whether the homeomorphism group of a closed connected manifold is Jordan. We also remark that the Jordan property also has a rich history in the field of algebraic geometry, which started when J.-P. Serre proved that the Cremona group $Cr_2=\operatorname{Bir}(\mathbb{P}^2)$ is Jordan in \cite{serre2009minkowski} and asked if higher rank Cremona groups $Cr_n$ are also Jordan (the question has been affirmatively answered in \cite{prokhorov2016jordan,birkar2016singularities}). For the most recent developments on the Jordan property in algebraic geometry we refer to the surveys \cite{bandman2024jordan,luo2025jordan}, as well as \cite[\S 2]{mundet2021topological} and the introductions of \cite{csikos2022finite} and \cite{golota2023finite}.

The discrete degree of symmetry of a manifold $M$ was introduced in \cite{mundet2021topological} to study rigidity results on rationally hypertoral manifolds (see \cref{large group actions hypertoral manifolds} for one of the main results in \cite{mundet2021topological}). The discrete degree of symmetry can be regarded as the finite group analogue of the toral degree of symmetry of $M$, defined as $$ \A(M)=\max(\{0\}\cup\{r\in\mathbb{N}:T^r \text{ acts effectively on $M$}\}).$$ The toral degree of symmetry has been widely studied (see \cite[Chapter VII. \S2]{hsiang2012cohomology}, \cite[\S 11.7, \S11.8]{lee2010seifert} and the survey \cite{grove2002geometry}). The theorem of invariance of domain imply that if $M$ is a closed manifold of dimension $n$ then $\A(M)\leq n$ and that $\A(M)=n$ if and only if $M\cong T^n$ (for details, see \cite[12.2]{mundet2021topological}). \Cref{large finite group actions questions}.(2) asks if the same results hold for the discrete degree of symmetry. Note that there exist manifolds such that $\A(M)\neq\D(M)$ (see \cite[Theorem 1.11]{mundet2021topological}). 

The discrete degree of symmetry and the toral degree of symmetry are related to other well-known invariants of manifolds, like the the toral rank or the $p$-rank of a manifold. Recall that an action of a Lie group $G$ on a manifold $M$ is almost-free if all the stabilizer subgroups of the action of $G$ on $M$ are finite.
\begin{defn}\label{def ranks}
	Assume that $M$ is a manifold. We define:
	\begin{itemize}
		\item[1.] The toral rank of $M$ is $$\rank(M)=\max\{\{0\}\cup\{r:T^r\text{ acts almost-freely on $M$}\}\}.$$
		\item[2.] Given a prime $p$, the $p$-rank of $M$ is $$\rank_p(M)=\max\{\{0\}\cup\{r:(\Z/p)^r\text{ acts freely on $M$}\}\}.$$
	\end{itemize}
\end{defn} 

If $M$ is a closed connected manifold then it is clear that $\rank(M)\leq \A(M)\leq\D(M)$. There are two important conjectures due to S.Halperin and G.Carlsson regarding the $\rank(M)$ and $\rank_p(M)$ (see \cite[\S 7.3]{felix2008algebraic}).

\begin{conj}\label{toral rank conjecture}
	Let $M$ be a closed connected manifold. Then: 
	\begin{itemize}
		\item[1.] (Toral rank conjecture) $\dim H^*(M,\Q)\geq 2^{\rank(M)}$.
		\item[2.] (Carlsson's conjecture) For all prime $p$, $\dim H^*(M,\Z/p)\geq 2^{\rank_p(M)}$.
	\end{itemize}
\end{conj}

There is a weaker version of the Carlsson's conjecture, which we call the stable Carlsson's conjecture.

\begin{conj}\label{stable carlsson conjecture}(Stable Carlsson's conjecture)
Given a closed connected manifold $M$, there exists a constant $C$ such that $\dim H^*(M,\Z/p)\geq 2^{\rank_p(M)}$ for all prime $p>C$. 
\end{conj}

The stable Carlsson conjecture was answered affirmatively for products of spheres in \cite{hanke2009stable}.

Regarding \cref{large finite group actions questions}.(3), a classical problem in geometric topology is to construct manifolds which do not admit any effective finite group action (see the survey \cite{puppe2007manifolds}). For example, A. Borel constructed the first example of asymmetric manifold in a unpublished work (see \cite{borel1983periodic}). Almost-asymmetry is a weaker property than asymmetry. There are many examples of almost-asymmetrical manifolds, like closed connected oriented surfaces of genus greater than 1. This is a consequence of Hurwicz $84(g-1)$ bound (for details, see \cite{IgnasiMundetiRiera2010Jtft}). We remark that a closed manifold $M$ is almost asymmetric if and only if $\D(M)=0$ (see \cite[Lemma 8.1]{riera2023actions}).

Lastly, \cref{large finite group actions questions}.(4) appeared in \cite[Question 1.9]{csikos2021number} motivated by the following theorem:
\begin{thm}\cite[Theorem 1.3]{csikos2021number}\label{few stabilizers p-groups}
	There exists a constant $C$ such that any finite $p$-group $G$ acting effectively on a closed manifold $M$ has a subgroup $H$ such that $[G:H]\leq C$ and $|\Stab(H,M)|\leq C$.
\end{thm} 
This result was crucial in proving a generalized version of the Jordan property of the homeomorphism group of closed manifolds in \cite{csikos2022finite}, which states that there exists a constant $C$ such that every finite group $G$ acting on $M$ has a nilpotent subgroup $N$ such that $[G:N]\leq C$.

In \cite{daura2024actions}, the author answered these questions for a large class of closed connected aspherical manifolds, which contain closed aspherical locally homogeneous spaces.

\begin{thm}\label{main theorem1 intro}(\cite[Theorem 1.6, Theorem 1.9]{daura2024actions})
	Let $M$ be a closed connected $n$-dimensional aspherical manifold such that $\Zc(\pi_1(M))$ is finitely generated and $\Out(\pi_1(M))$ is Minkowski. Then:
	\begin{itemize}
		\item[1.] $\Homeo(M)$ is Jordan.
		\item[2.] $\D(M)\leq \rank\Zc(\pi_1(M))\leq n$, and $\D(M)=n$ if and only if $M$ is homeomorphic to $T^n$.
		\item[3.] If $\chi(M)\neq 0$ then $M$ is almost-asymmetric.
		\item[4.] If $\Aut(\pi_1(M))$ is Minkowski, then $M$ has few stabilizers.
	\end{itemize}
	Let $\Gamma$ be a lattice of a connected Lie group, then $\Out(\Gamma)$ and $\Aut(\Gamma)$ are Minkowski.
\end{thm}

There has been many efforts to generalize results describing group actions on closed connected aspherical manifolds to a wider class of manifolds (see \cite{schultz1981group,DONNELLY1982443,Gottlieb1985,ku1983group,washiyama1983degree,IgnasiMundetiRiera2010Jtft}). Recently, in \cite{mundet2021topological} I. Mundet i Riera proved:

\begin{thm}\cite[Theorem 1.3, Theorem 1.14]{mundet2021topological}\label{large group actions hypertoral manifolds}
	Let $M$ be a closed connected orientable $n$-dimensional manifold which admits a non-zero degree map $f:M\longrightarrow T^n$. Then:
	\begin{itemize}
		\item[1.] $\Homeo(M)$ is Jordan.
		\item[2.] $\D(M)\leq n$ and if $\D(M)=n$ then there is a ring isomorphism $H^*(M,\Z)\cong H^*(T^n,\Z)$. Moreover, if $\D(M)=n$ and $\pi_1(M)$ is virtually solvable then $M\cong T^n$.
		\item[3.] If $\chi(M)\neq 0$ then $M$ is almost asymmetric.
		\item[4.] $M$ has few stabilizers.
	\end{itemize}
\end{thm}

Our first result generalizes \cref{large group actions hypertoral manifolds} to manifolds admitting a non-zero degree map to a nilmanifold.

\begin{thm}\label{main theorem2 intro}
	Let $M$ be a closed oriented connected $n$-dimensional manifold and $f:M\longrightarrow N/\Gamma$ a non-zero degree map to a closed nilmanifold. Then:
	\begin{itemize}
		\item[1.] $\Homeo(M)$ is Jordan.
		\item[2.] $\D(M)\leq \rank\Zc\Gamma\leq n$ and if $\D(M)=n$ then $H^*(M,\Z)\cong H^*(T^n,\Z)$. Moreover, if $\D(M)=n$ and $\pi_1(M)$ is virtually solvable then $M\cong T^n$.
		\item[3.] If $\chi(M)\neq 0$ then $M$ is almost asymmetric.
		\item[4.] $M$ has few stabilizers.
	\end{itemize}
\end{thm}

As an application of \cref{main theorem2 intro}, we prove:

\begin{prop}\label{TRC hypernilmanifolds intro}
Let $M$ be a closed oriented connected $n$-dimensional manifold and $f:M\longrightarrow N/\Gamma$ a non-zero degree map to a closed nilmanifold. Assume that the toral rank conjecture holds for $N/\Gamma$, then the toral rank conjecture and the stable Carlsson's conjecture hold for $M$.
\end{prop}

\Cref{TRC hypernilmanifolds intro} follows from \cref{main theorem2 intro} and the fact that the toral rank conjecture and the stable Carlsson's conjecture are equivalent for closed aspherical locally homogeneous spaces (see \cref{TRC=SCC}). The toral rank conjecture is true if $N/\Gamma$ is a 2-step nilmanifold (see \cite{deninger1988cohomology}), thus \cref{TRC hypernilmanifolds intro} can be used when $N/\Gamma$ is a 2-step nilmanifold. For other nilmanifolds where it is known that the toral rank conjecture holds see \cite{cairns1997betti,cairns1997new}.

To prove \cref{main theorem2 intro} we introduce a new concept called exporting map, inspired by \cite[Theorem 4.1]{mundet2021topological}. A map between manifolds $f:M\longrightarrow M'$ is an exporting map if there exists a constant $C$ such that every finite group $G$ acting on $M$ has a subgroup $H\leq G$ such that $H$ acts on $M'$, there exists a $H$-equivariant map $f_H:M\longrightarrow M'$ homotopic to $f$ and $[G:H]\leq C$ (see \cref{pE defn}).

It is a natural question to ask whether $\D(M)=\rank\Zc\Gamma$ implies $H^*(M,\Z)\cong H^*(N/\Gamma,\Z)$. This is not true in general (see \cref{discsym  not enough}). Thus, we want to find a new invariant refining $\D(M)$ to study cohomological rigidity for manifolds admitting a non-zero degree map to a nilmanifold. In order to do so, we recall that nilmanifolds are precisely total spaces of iterated principal $S^1$-bundles (see \cite{belegradek2020iterated}). This fact leads to the following definition:

\begin{defn}
	Let $\mathcal{G}=\{G_i\}_{i=1,\dots,n}$ be a collection of groups and let $X$ be a topological space. An iterated action of $\mathcal{G}$ on $X$ (denoted by $ \mathcal{G}\acts X$) is:
\begin{itemize}
	\item[1.] A sequence of surjections of topological spaces
	\[\begin{tikzcd}
		X=X_0 \ar{r}{p_1} & X_1 \ar{r}{p_2} & X_2 \ar{r}{p_3} & \cdots \ar{r}{p_n} & X_n,
	\end{tikzcd}\]
	\item[2.] and a collection of group actions $\{\Phi_i:G_i\longrightarrow \Homeo(X_{i-1})\}_{i=1,\dots,n}$,
\end{itemize}
satisfying the property that the maps $p_i:X_{i-1}\longrightarrow X_i$ are the orbit maps of the action of $G_i$ on $X_{i-1}$.
\end{defn}

The concept of iterated group action has appeared implicitly in the literature. For example, towers of regular self-covering are studied in \cite{baker2001towers,van2018towers,van2021structure,qin2021self}. A regular self-covering of a closed manifold $M$ is a map $p:M\longrightarrow M$ which is a regular covering. We can compose this map with itself to obtain a tower of regular self-coverings
\[\begin{tikzcd}
	M\ar{r}{p} & M \ar{r}{p} & \cdots\ar{r}{p} & M\ar{r}{p} & M.
\end{tikzcd}
\]

Since $p:M\longrightarrow M$ is a regular covering, there exists a finite group $G$ acting freely on $M$ such that $p$ can be seen as the  orbit map $p:M\longrightarrow M/G$. Consequently, the study of towers of regular self-coverings is the study of iterated group actions of $\mathcal{G}=\{G,\overset{n}{\cdots}, G\}$ such that $M_i\cong M$ for all $i$ and all the actions $\Phi_i:G\longrightarrow \Homeo(M)$ are the same. In \cite{figueroa2012half} iterated group actions are used to describe and classify spin orbifolds of the form $S^7/\Gamma$ where $\Gamma$ is a finite group of $\SO(8)$.  Iterated group actions have also been studied when each group $G_i$ is a connected Lie group. For example, in \cite{baues2023isometry} O. Baues and Y. Kamishima study iterated group actions on Riemannian aspherical manifolds where each group $G_i$ is the solvable radical of $\Isom(M_{i-1})^0$. 

The next proposition shows that free iterated actions are a suitable tool to study nilmanifolds. Recall that if a compact Lie group $G$ acts freely on a manifold $M$ then the quotient map $M/G$ is also a manifold of dimension $\dim M-\dim G$ and the orbit map $M\longrightarrow M/G$ is a principal $G$-bundle (see \cite[Chapter II]{bredon1972introduction}).

\begin{prop}\label{free iterated action of tori}
Let $M$ be a closed connected manifold of dimension $n$. Assume that we have a free iterated action $\{T^{b_1},\dots,T^{b_c}\}\acts M$. Then $\sum_{i=1}^c b_i \leq n$ and we have an equality if and only if $M$ is a nilmanifold.
\end{prop}

\begin{proof}
Since all the actions of the iterated action are free, the maps $p_i:M_{i-1}\longrightarrow M_{i}$ are principal $T^{b_i}$-bundles and the quotient $M_{c-1}$ is a closed manifold of dimension $n-\sum _{i=1}^{c-1}b_i$. Therefore, $b_c\leq n-\sum _{i=1}^{c-1}b_i$ and hence $\sum_{i=1}^c b_i \leq n$. If the equality holds, then $b_c=\dim(M_{c-1})$. Since $M_{c-1}$ is closed and connected and it admits a free action of a torus $T^{\dim(M_{c-1})}$ we can conclude that $M_{c-1}\cong T^{b_c}$. This implies that $M$ can be obtained as iterated principal torus bundles and hence $M$ is a nilmanifold.
\end{proof}

Our focus is on free iterated actions of finite groups on manifolds. Let us define the notions that we need to state our results on iterated actions.

\begin{defn}\label{iterated action equivalent}
Let $\mathcal{G}=\{G_i\}_{i=1,...,n}$ and $\mathcal{G}'=\{G'_i\}_{i=1,...,n'}$ be  two collections of finite groups which act freely on $M$. Let $p=p_{n}\circ\cdots \circ p_1$ and $p'=p'_{n'}\circ\cdots\circ p'_1$. We say that the iterated actions $\mathcal{G}\acts M$ and $\mathcal{G}'\acts M$ are equivalent (and we denote it by $\mathcal{G}\acts M\sim\mathcal{G}'\acts M$) if:
\begin{itemize}
	\item[1.] There exists a homeomorphism $f:M_n\longrightarrow M_{n'}$.
	\item[2.] The coverings $p:M\longrightarrow M/\mathcal{G}$ and $f^*p':M\longrightarrow M/\mathcal{G}$ are isomorphic. That is, there exists a homeomorphism $\overline{f}:M\longrightarrow M$ satisfying $p'\circ \overline{f}=f\circ p$.
\end{itemize}
The equivalence class will be denoted by $[\mathcal{G}\acts M]$. If $\mathcal{G}\acts M$ is equivalent to $\{G\}\acts M$ for some group $G$ then we say that $\mathcal{G}\acts M$ is simplifiable.
\end{defn}

With \cref{iterated action equivalent} is straightforward to see that a free iterated action $\mathcal{G}=\{G_1,\dots,G_n\}\acts M$ is simplifiable if and only if $\pi_1(M)\trianglelefteq\pi_1(M_n)$. In particular, all free iterated actions on simply-connected manifolds are simplifiable. However, there exist manifolds with non-trivial fundamental group where all free iterated actions on them are simplifiable (see \cref{free itertaed action simplifiable S1 and T2}). There also many examples of non-simplifiable free iterated actions (see, for example, \cref{iterated action T2 not simplfiable}). Recall that if $G$ is a finite group, then $\rank G$ is the minimum number of elements which are needed to generate $G$. The next definition refines the concept of the discrete degree of symmetry to iterated actions of two groups.

\begin{defn}
Given a free iterated action $\mathcal{A}\acts M$ of finite abelian groups, the rank of the iterated action is $$\rank_{ab}(\mathcal{A}\acts M)=\min\{\sum_{i=1}^{n}\rank A'_i:\{A_1',\cdots,A_n'\}\acts M\in [\mathcal{A}\acts M]\text{, $A_i'$ abelian for all $i$}\}.$$	
We define $\mu_2(M)$ as the set of all pairs $(f,b)\in \mathbb{N}^2$ which satisfy:
\begin{itemize}
	\item[1.] There exist an increasing sequence of prime numbers $\{p_i\}$, a sequence of natural numbers $\{a_i\}$ and a collection of free iterated actions $\{(\Z/p_i^{a_i})^f,(\Z/p_i)^b\}\acts M$ for each $i\in\mathbb{N}$.
	\item[2.] $\rank_{ab}(\{(\Z/p_i^{a_i})^f,(\Z/p_i)^b\}\acts M)=f+b$ for each $i\in\mathbb{N}$.
\end{itemize}

Consider the lexicographic order in $\mathbb{N}^2$, that is $(a,b)\geq (c,d)$ if $a>c$, or $a=c$ and $b\geq d$. Define the iterated discrete degree of symmetry of $M$ as
$$\D_2(M)=\max\{(0,0)\cup\mu_2(M)\}.$$
\end{defn}

Note that there are two significant differences between the conditions used in the definitions of the discrete degree of symmetry and the iterated discrete degree of symmetry. On the iterated discrete degree of symmetry we only consider actions of abelian $p$-groups, and all these actions are assumed to be free. While this hypothesis are made for technical reasons, they do not suppose a big loss of generality in our case, since by \cref{main theorem2 intro}.4, any effective action of an abelian $p$-group on a manifold $M$ admitting a non-zero degree to a nilmanifold is free for large enough $p$.

\begin{thm}\label{iterated discsym nilmanifolds intro}
	Let $N/\Gamma$ be a nilmanifold of dimension $n$, let $a=\rank\Zc\Gamma$ and let $\D_2(N/\Gamma)=(d_1,d_2)$. Then $d_1=a$ and $d_2\leq n-a$. Moreover, $d_2=n-a$ if and only if $N/\Gamma$ is a 2-step nilmanifold or a torus.
\end{thm}

Finally, we prove that the iterated discrete degree of symmetry is the right tool to study cohomological rigidity of manifolds admitting non-zero degree maps to 2-step nilmanifolds.

\begin{thm}\label{main theorem11 intro}
	Let $M$ be a closed connected manifold admitting a non-zero degree map $f:M\longrightarrow N/\Gamma$ to a $2$-step nilmanifold, which is the total space of a principal $T^a$-bundle over $T^b$. Then $\D_2(M)\leq (a,b)$. Moreover, if $\D_2(M)= (a,b)$ then $H^*(M,\Q)\cong H^*(N/\Gamma,\Q)$.
\end{thm}

This paper is divided in five sections. In the first section we introduce the concept of exporting maps and we deduce some of their basic properties. In the second section we prove \cref{main theorem2 intro} and we give some further remarks and consequences of \cref{main theorem2 intro}. We introduce free iterated actions in section 3 and the iterated discrete degree of symmetry in section 4. Finally, section 5 is devoted to the proof of \cref{main theorem11 intro}.

\textbf{Acknowledgements:} I would like to thank Ignasi Mundet i Riera for all the guidance given during this project and for the careful revision of the first draft of this paper. Many thanks to Leopold Zoller for helpful
comments and corrections. 
This work was partially supported by the grant PID2019-104047GB-I00 from the Spanish Ministry of Science and Innovation and the Departament de Recerca i Universitats de la Generalitat de Catalunya (2021 SGR 00697).
\section{Exporting and importing maps}\label{sec: exporting and importing maps}

A key result to prove \cref{large group actions hypertoral manifolds} is the following theorem.

\begin{thm}\label{pE torus}\cite[Theorem 4.1]{mundet2021topological}
	Let $M$ be a closed oriented manifold of dimension $n$ which admits a continuous map $f:M\longrightarrow T^n$ of non-zero degree and let $G$ be a finite group acting effectively on $M$ and trivially on $H^1(M,\Z)$, then there exist a group action of $G$ on $T^n$ and a continuous map $f_G:M\longrightarrow T^n$ which is $G$-equivariant and homotopic to $f$.
\end{thm}

\begin{rem}
	The constructed action of $G$ on $T^n$ is not necessarily effective. If $K$ denotes the kernel of ineffectiveness, then $|K|\leq \deg f$ and the effective action of $G/K$ on $T^n$ is free and by rotations. Thus the induced group morphism $G/K\longrightarrow \Gl(n,\Z)$ is trivial.
\end{rem}

\Cref{pE torus} inspires the following definition:

\begin{defn}\label{pE defn}
	Let $M$ and $M'$ be closed oriented manifolds of the same dimension and let $f:M\longrightarrow M'$ be a continuous map. We say that $f$ exports group actions, or $f$ is an exporting map, if there exists a constant $C$ such that every finite group $G$ acting effectively on $M$ (which we denote by $\phi:G\longrightarrow \Homeo(M)$) has a subgroup $H\leq G$ such that:
	\begin{itemize}
		\item[1.] $[G:H]\leq C$
		\item[2.] There exists an action $H$ on $M'$ (denoted by $\phi':H\longrightarrow \Homeo(M')$).
		\item[3.] There exists an $H$-equivariant map $f_H:M\longrightarrow M'$ homotopic to $f$.
	\end{itemize} 
\end{defn}

By Minkowski's lemma, given a closed connected manifold $M$ there exists a constant $C$ such that any finite group $G$ acting effectively on $M$ has a finite index subgroup $H$ of index $[G:H]\leq C$ such that $H$ acts trivially on $H^1(M,\Z)$. Thus, \cref{pE torus} states that any non-zero degree map $f:M\longrightarrow T^n$ is an exporting map. Another example of this property in the smooth setting is provided by the following theorem of R.Schoen and S.T.Yau.

\begin{thm}\cite[Theorem 8]{schoen1979compact}\label{pE harmonic}
	Let $M$ and $M'$ be closed connected orientable smooth manifolds of the same dimension. Assume that $M'$ has a Riemannian metric of non-positive curvature and there is a non-zero degree smooth map $f:M\longrightarrow M'$ such that $f_*:\pi_1(M)\longrightarrow \pi_1(M')$ is surjective. Let $A(M')$ denote the group of affine transformations on $M'$ (that is, the group of diffeomorphisms preserving the Levi-Civita connection) and let $\overline{A}(M')$ denote the subgroup of the identity component of $A(M')$ generated by the parallel vector fields of $M'$. Given a finite group $G$ acting effectively and smoothly on $M$, suppose that for each $g\in G$ there exists an element $g'\in \overline{A}(M')$ such that $f\circ g$ is freely homotopic to $g'\circ f$ ($g'$ is not necessarily unique or non-trivial). Then there exist a group morphism $\gamma:G\longrightarrow \overline{A}(M')$ and a $\gamma$-equivariant smooth map $f_G:M\longrightarrow M'$ homotopic to $f$. Moreover, $|\ker\gamma|\leq \deg(f)$.
\end{thm} 

If there exists a constant $C$ such that every finite group $G$ acting smoothly and effectively on $M$ has a subgroup $H\leq G$ satisfying the conditions of \cref{pE harmonic} and $[G:H]\leq C$ then $f$ is an exporting map. Some examples can be found in \cite[Theorem 11, Theorem 13]{schoen1979compact}.

\begin{thm}\cite[Theorem 11, Theorem 13]{schoen1979compact}\label{exporitng maps SY}
	Let $M$ and $M'$ be closed connected orientable smooth manifolds of the same dimension. Assume that there is a degree one map $f:M\longrightarrow M'$ such that $f_*:\pi_1(M)\longrightarrow \pi_1(M')$ is surjective. Furthermore, assume that $M'$ satisfy one of the following set of conditions:
	\begin{itemize}
		\item[1.] $M'$ is diffeomorphic to a locally symmetric space $\Gamma\setminus G/H$, where all the factors of $G$ have real rank equal or greater than $2$.
		\item[2.] $M'$ is flat and $\Zc\pi_1(M')$ is trivial.
	\end{itemize}
	Then for any finite group $G$ acting effectively and smoothly on $M$ there exist a smooth effective action of $G$ on $M'$ and a $G$-equivariant map $f_G:M\longrightarrow M'$ homotopic to $f$.
\end{thm}

Note that in \cref{exporitng maps SY} we do not need to replace the finite group by a suitable subgroup of bounded index. 

The aim of this section is to study exporting maps.

\begin{lem}\label{exporting map non-zero degree}
	Let $f:M\longrightarrow M'$ be a non-zero degree exporting map. With the notation as in \cref{pE defn}, we have $|\ker\phi'|\leq \deg(f)$.
\end{lem}

The proof of the lemma is a consequence of the next fact.

\begin{lem}\label{quotient map image degree}\cite[Lemma 4.4]{mundet2021topological}
	Let $M$ be a closed oriented $n$-manifold and suppose that $G$ is a finite group acting on $M$ effectively and preserving the orientation. If we denote by $\pi:M\longrightarrow M/G$ the quotient map and by $d$ the cardinal of $G$ then $\pi^*(H^n(M/G,\Z))\subseteq dH^n(M,\Z)$.  
\end{lem}

\begin{proof}[Proof of \cref{exporting map non-zero degree}]
	Since $f_H$ is $H$-equivariant, we can restrict the action of $H$ on $M$ to the subgroup $\ker\phi'\leq H$, obtaining a commutative diagram
	\[
	\begin{tikzcd}
		M\ar{r}{f_H}\ar{d}{\pi} & M'\\
		M/\ker\phi'\ar{ru}{\overline{f}_H} &\\
	\end{tikzcd}
	\]
	In consequence, $\pi^*\circ \overline{f}^*_H=f^*_H$. Since $f$ and $f_H$ are homotopic, we have $f^*_H=f^*:H^n(M',\Z)\longrightarrow H^n(M,\Z)$. By \cref{quotient map image degree}, we obtain that $|\ker\phi'|$ divides $\deg(f)$. In particular $|\ker\phi'|\leq \deg(f)$.
\end{proof}

The next lemma studies the composition of two exporting maps.

\begin{lem}\label{pE composition}
	Let $M$, $M'$ and $M''$ be closed oriented manifolds and let $f:M\longrightarrow M'$ and $g:M'\longrightarrow M''$ be exporting non-zero degree maps with constants $C$ and $D$ respectively. Then $g\circ f:M\longrightarrow M''$ exports group actions with constant $C\cdot D$.
\end{lem}

\begin{proof}
	Assume that $G$ is a finite group acting effectively on $M$ and take the subgroup $H\leq G$ of \cref{pE defn}. Thus, there exists an action of $H$ on $M'$, denoted by $\phi':H\longrightarrow \Homeo(M')$, a $H$-equivariant map $f_H:M\longrightarrow M'$ homotopic to $f$ and $[G:H]\leq C$. Since $\phi'(H)$ acts effectively on $M'$, there exist a subgroup $K\leq \phi'(H)$, an action $\phi'':K\longrightarrow \Homeo(M'')$ satisfying $[\phi'(H):K]\leq D$, and a $K$-equivariant map $g_K:M'\longrightarrow M''$ homotopic to $g$. We can consider the subgroup $H'=\phi'^{-1}(K)\leq H$ and the action of $H'$ on $M''$ given by the group morphism $\phi''\circ\phi':H'\longrightarrow K\leq \Homeo(M'')$. The map $g_K\circ f_H$ is $H'$-equivariant and homotopic to $g\circ f$. In addition, $[G:H']\leq C\cdot D$ and hence the claim is proved.
\end{proof}

Before stating the main result of this section, we introduce two more properties of large finite group actions on manifolds.

\begin{defn}\label{stabilizers properties def}
	Let $M$ be a closed connected manifold. Then:
	\begin{itemize}
		\item[1.] We say that $M$ has the small stabilizers property if there exist a constant $C$ such that if $G$ is a finite group acting effectively on $M$ then $|G_x|\leq C$ for all $x\in M$.
		\item[2.] We say that $M$ has the almost fixed point property if there exists a constant $C$ such that if $G$ is a finite group acting effectively on $M$ then there exist $x\in M$ such that $[G:G_x]\leq C$.
	\end{itemize}
\end{defn}

The first property is new and it has not been studied yet. We will show that the manifolds studied in this article have this property. The second property was studied in \cite{riera2024jordan}.

\begin{thm}\cite[Theorem 1.5]{riera2024jordan}\label{almost fix point euler char}
	Let $M$ be a closed connected manifold with $\chi(M)\neq 0$. Then $M$ has the almost fixed point property.
\end{thm}

The almost fixed point property and the small stabilizers property are dual in the following sense:

\begin{lem}\label{small stab+almost fixed point=almost asymmetric}
	A closed manifold $M$ with the small stabilizers and the almost fixed point property is almost asymmetric.
\end{lem}

\begin{proof}
	Let $C$ be the constant of the almost fixed point property and $D$ the constant of the small stabilizers property. Assume that we have a finite group $G$ acting effectively on $M$. Then there exists $x\in M$ such that $[G:G_x]\leq C$. In addition, $|G_x|\leq D$. Consequently, $|G|\leq C\cdot D$. Thus, $M$ is almost asymmetric.
\end{proof}

The main result of this section is that properties introduced in \cref{large finite group actions def} behave well with respect exporting maps. It is one of the main ingredients to prove \cref{main theorem2 intro}.

\begin{thm}\label{pE jordan and discsym}
	Let $M$ and $M'$ closed oriented manifolds which admit a non-zero degree exporting map $f:M\longrightarrow M'$. Then:
	\begin{itemize}
		\item[1.] If $\Homeo(M')$ is Jordan, then $\Homeo(M)$ is Jordan.
		\item[2.] $\D(M)\leq \D(M')$.
		\item[3.] If $M'$ has the small stabilizer property then $M$ has the small stabilizer property.
		\item[4.] If $M'$ is almost asymmetric, then $M$ is almost asymmetric.
	\end{itemize} 
\end{thm}

In order to prove this theorem we need the three following group-theoretic lemmas.

\begin{lem}\label{Jordan ses}\cite[Lemma 2.2]{IgnasiMundetiRiera2010Jtft}
	Let $d$ and $r$ be natural numbers. There exists a natural number $C(d,r)$ such that if we have a short exact sequence of groups 
	$$1\longrightarrow K\longrightarrow G \longrightarrow A\longrightarrow 1 $$
	where $|K|\leq d$ and $A$ is abelian and generated by $r$ elements, then $G$ has an abelian subgroup of index at most $C(d,r)$.
\end{lem}

\begin{lem}\label{subgroup elementary abelian groups}\cite[Lemma 2.1]{mundet2021topological}
	Let $a$,$b$ and $C$ be natural numbers and suppose that $G'$ is a subgroup of $(\Z/a)^b$ satisfying $[(\Z/a)^b:G']\leq C $. Then there exist a number $a'$ and a subgroup $G''\leq G'$ such that $G''\cong (\Z/a')^b$ and $C!a'\geq a$. 
\end{lem}

To state the last lemma we need to introduce two new invariants:

\begin{defn}\label{def: discsym primes}
	Let $M$ be a manifold. We define
	$$P\D(M)=\max\{\{0\}\cup\{r:(\Z/p)^r \text{ acts effectively on $M$ for arbitrarily large prime $p$}\}\}.$$ 
	
	For a fixed prime $p$, we define
	$$\D_p(M)=\max\{\{0\}\cup\{r:(\Z/p^s)^r \text{ acts effectively on $M$ for arbitrarily large $s$}\}\}. $$
\end{defn}

\begin{lem}\label{discsym prime version}
	Let $M$ be a closed manifold, then
	$$\D(M)=\max(\{P\D(M)\}\cup\{\D_p(M): \text{$p$ prime} \}).$$ 
\end{lem}

\begin{proof}
	Clearly, $P\D(M)\leq \D(M)$ and $\D_p(M)\leq \D(M)$ for every prime $p$. Conversely, assume that $\D(M)=b$ and therefore there exists a sequence of natural numbers $\{a_i\}_{i\in \mathbb{N}}$ such that $(\Z/a_i)^b$ acts effectively on $M$ for all $i$ and $a_i\longrightarrow \infty$ when $i\longrightarrow\infty$. Let $\mathcal{P}$ be the set of primes which divide $a_i$ for some $i$.
	
	If $|\mathcal{P}|=\infty$ then there exists a subsequence $\{a_k\}_{k\in\mathbb{N}}$ of $\{a_i\}_{i\in\mathbb{N}}$ such that each $a_k$ is divided by a prime $p_k$ satisfying that $p_k<p_{k+1}$. Thus, by taking $(\Z/p_k)^b\leq (\Z/a_k)^b$ we have an effective action of $(\Z/p_k)^b$ on $M$. Consequently, $b\leq P\D(M)$. 
	
	If $|\mathcal{P}|<\infty$ then there exist $m$ primes $p_1$,\dots, $p_m$ such that $a_i=p_1^{x_{1,i}}\dots p_m^{x_{m,i}}$ for all $i$. Since $a_i\longrightarrow \infty$ when $i\longrightarrow\infty$, by the pigeonhole principle there exists a subsequence $\{a_k\}_{k\in\mathbb{N}}$ and a number $l\in \{1,\dots,m\}$ such that $x_{l,k}\longrightarrow \infty$ when $k\longrightarrow \infty$. Thus, we have effective group actions of $(\Z/p_l^{x_{l,k}})^b$ on $M$. Consequently, $b\leq \D_{p_{l}}(M)$. 
	
	By combining these two cases we obtain the desired result.
\end{proof}

We are almost ready to give the proof of \cref{pE jordan and discsym}. The last result we need is a corollary of a theorem by L.N. Mann and J.C. Su.

\begin{thm}\cite{mann1963actions}\label{MannSu thm}
	Let $M$ be a closed manifold of dimension $n$. For a prime $p$, we define $b_p(M)=\sum_{i=0}^{n}\dim H^i(M,\Z/p)$. There exists a number $C_p$ only depending on $n$ and $b_p(M)$ such that if $(\Z/p)^r$ acts effectively on $M$ then $r\leq C_p$.
\end{thm}

\begin{cor}\cite{mann1963actions}\label{bound rank finite groups}
	Let $M$ be a closed manifold of dimension $n$. There exists a number $r$ such that if $A$ is a finite abelian group acting effectively on $M$ then $\rank(A)\leq r$. 
\end{cor}

\begin{proof}[Proof of \cref{pE jordan and discsym}]
	For the first statement of the theorem assume that $G$ is a finite group acting effectively on $M$. Then there exist a subgroup $H$ of $G$ such that $[G:H]\leq C$, an action of $H$ on $M'$ and a continuous map $f_H:M\longrightarrow M'$ which is $H$-equivariant and homotopic to $f$. Let $H_0=\ker\phi'$. By \cref{exporting map non-zero degree}, we known that $|H_0|\leq \deg(f)$. The group $\phi'(H)$ acts effectively on $M'$. Thus there exists an abelian subgroup $A'\leq \phi'(H)$ such that $[\phi'(H):A']\leq C'$, where $C'$ is the Jordan constant of $\Homeo(M')$. By \cref{bound rank finite groups}, there exists a constant $r'$ such that any abelian finite group acting effectively on $M'$ has rank at most $r'$. 
	
	We consider the commutative diagram 
	\[
	\begin{tikzcd}
		1\ar{r}{} & H_0\ar{r}{}\ar{d}{Id} & \phi'^{-1}(A')\ar{r}{\phi'}\ar[hook]{d}{} & A'\ar{r}{}\ar[hook]{d}{} &1\\	
		1\ar{r}{} & H_0\ar{r}{} & H\ar{r}{\phi} & \phi(H)\ar{r}{} &1\\
	\end{tikzcd}
	\]
	Note that $[H:\phi'^{-1}(A')]\leq C'$. We can use \cref{Jordan ses} on the upper short exact sequence to find a constant $D$, which only depends on $d$ and $r'$, and an abelian subgroup $A$ of $\phi'^{-1}(A')$ such that $[\phi'^{-1}(A'):A]\leq D$. In conclusion, we have found a subgroup $A$ of $G$ such that $[G:A]\leq C\cdot C'\cdot D$ and we can conclude that $\Homeo(M)$ is Jordan.
	
	We note that by \cref{subgroup elementary abelian groups} we can assume that given an increasing sequence  $\{a_i\}_{i\in\mathbb{N}}$ such that $a_i\longrightarrow \infty$ and groups $(\Z/a_i)^b$ acting effectively on $M$ then $(\Z/a_i)^b$ also act on $M'$ for all $i$ and that we can replace $f$ by an homotopic equivariant map for each $i$. 
	
	By \cref{discsym prime version} we can divide the prove in two parts. Firstly assume that we have an increasing sequence of primes $\{p_k\}_{k\in\mathbb{N}}$ and groups $(\Z/p_k)^b$ acting effectively on $M$. Like in the first part of the proof, we consider the group action $\phi'_{{p_k,b}}:(\Z/p_k)^b\longrightarrow \Homeo(M')$ AND the exact sequence 
	\[
	\begin{tikzcd}
		1\ar{r}{} &K_{(p_k,b)}\ar{r}{} & (\Z/p_k)^b\ar{r}{\phi'_{{p_k,b}}} & \phi'_{{p_k,b}}((\Z/p_k)^b)\ar{r}{} &1\\
	\end{tikzcd}
	\]
	for each $k$, where $K_{(p_k,b)}=\ker\phi'_{{p_k,b}}$.
	Since $K_{(p_k,b)}$ is a subgroup of $(\Z/p_k)^b$, there exists $x(k)\in \mathbb{N}$ such that
	$|K_{(p_k,b)}|=p_k^{x(k)}$. On the other hand, $|K_{(p_k,b)}|\leq d $. Since $p_k\longrightarrow\infty$ when $k\longrightarrow \infty$, there exist $k_0$ such that $p_k>d$ for all $k\geq k_0$. This implies that $x(k)=0$ for $k\geq k_0$ and that $(\Z/p_k)^b$ acts effectively on $M'$ for $k\geq k_0$. Thus, $P\D(M)\leq \D(M')$.
	
	We fix a prime number $p$ and denote by $c\in \mathbb{N}$ the largest number such that $p^c\leq d$. Assume that we have an increasing sequence $\{a_k\}_{k\in\mathbb{N}}$ such that $(\Z/p^{a_k})^b$ acts effectively on $M$. Since $a(k)\longrightarrow\infty$ when $k\longrightarrow \infty$, there exists a $k_0$ such that $a_{k_0}>c$ for all $k\geq k_0$. Then $K_{(p^{a_k},b)} $ is a subgroup of $(\Z/p^{c})^b\leq (\Z/p^{a_k})^b$ and $$(\Z/p^{a_k-c})^b\cong (\Z/p^{a_k})^b/(\Z/p^{c})^b\leq (\Z/p^{a_k})^b/K_{(p^{a_k},b)}\cong \phi'_{{p^{a_k},b}}((\Z/p^{a_k})^b)$$ for $k\geq k_0$. Hence,  $(\Z/p^{a_k-c})^b$ acts effectively on $M'$ for $k\geq k_0$ and $\D_p(M)\leq \D(M')$.
	
	Joining the two cases we obtain that $\D(M)\leq \D(M')$.
	
	To prove the third part assume that $G$ is a finite group acting effectively on $M$ with a fix point $x$. Let $C$ be the constant provided by the exporting map property and let $H$, $\phi':H\longrightarrow \Homeo(M')$ and $f_H$ be respectively the subgroup of $G$, the action on $M'$ and the continuous map homotopic to $f$ given by the assumptions. Since $x$ is a fixed point of the action of $H$ on $M$, $f_H(x)$ is a fixed point of the action of the effective $\phi'(H)$ on $M'$. If $C'$ is the small stabilizer constant on $M'$, then $|\phi'(H)|\leq C'$. We obtain the exact sequence 
	\[
	\begin{tikzcd}
		1\ar{r}{} &H_0\ar{r}{} &H\ar{r}{\phi'} &\phi'(H)\ar{r}{} &1
	\end{tikzcd}
	\]
	where $|H_0|\leq d=\deg(f)$ and $|\phi'(H)|\leq C'$. In consequence, $|H|\leq C'\cdot d$ and $|G|\leq C\cdot C'\cdot d$.
	
	The proof of the fourth part is analogous to the proof of the third part. If $C'$ denotes now the constant provided by the almost-asymmetric property of $M'$, then we have $|\phi'(H)|\leq C'$ and hence $|G|\leq C\cdot C'\cdot d$.
\end{proof}

Since $\Homeo(S^4)$ is Jordan and $\Homeo(T^2\times S^2)$ is not Jordan, we can deduce:

\begin{cor}
	Any non-zero degree map $f:T^2\times S^2\longrightarrow S^4$ is not an exporting map.
\end{cor}
Note also that $\D(T^2\times S^2)\geq 3> \D(S^4)=2$.

The next definition is the converse of the exporting map property.

\begin{defn}\label{pI defn}
	Let $M$ and $M'$ be closed oriented manifolds of the same dimension and let $f:M\longrightarrow M'$ be continuous map. We say that $f$ imports group actions, or it is an importing map, if there exists a constant $C$ such that any finite group $G$ acting on $M'$  has a subgroup $H\leq G$ such that:
	\begin{itemize}
		\item[1.] $[G:H]\leq C$.
		\item[2.] There exists a finite group $\tilde{H}$ acting effectively on $M$ and a surjective group morphism $\rho:\tilde{H}\longrightarrow H$.
		\item[3.] There exists a map $f_H:M\longrightarrow M'$ homotopic to $f$ which is $\rho$-equivariant ($f_H(hx)=\rho(h)f_H(x)$ for all $x\in M$ and $h\in\tilde{H}$).
	\end{itemize}
\end{defn}

\begin{rem}
	The property of importing group actions is similar to the property of propagating of group actions (see \cite[Definition 3.1]{adem2002topics}). Given a continuous map between closed manifolds $f:M\longrightarrow M'$ and a finite group $G$ acting effectively on $M'$, we that the action of $G$ on $M'$ propagates to $M$ across $f$ if there exists an effective action of $G$ on $M$ and an equivariant map $f_G:M\longrightarrow M'$ homotopic to $f$. Propagation of group actions was used to study group actions on homology spheres (see \cite[\S 3.2.2]{adem2002topics} and references therein).
\end{rem}

Importing and exporting maps have similar properties. The next lemma has an analogous proof to \cref{pE composition}.

\begin{lem}\label{pI composition}
	Let $M$, $M'$ and $M''$ be closed oriented manifolds and let $f:M\longrightarrow M'$ and $g:M'\longrightarrow M''$ be importing maps with constants $C$ and $D$ respectively. Then $g\circ f:M\longrightarrow M''$ imports group actions with constant $C\cdot D$.
\end{lem}

\begin{proof}
	Let $G$ be a finite group acting effectively on $M''$. Then there exist a subgroup $H\leq G$ satisfying that $[G:H]\leq D$, a group $\tilde{H}$ acting effectively on $M'$, a surjective group morphism $\rho':\tilde{H}\longrightarrow H$ and a $\rho$-equivariant map $g_H:M'\longrightarrow M''$. Since $\tilde{H}$ acts effectively on $M'$, there exist a subgroup $K\leq \tilde{H}$ satisfying that $[\tilde{H}:K]\leq C$, a group $\tilde{K}$ acting effectively on $M$, a surjective group morphism $\rho:\tilde{K}\longrightarrow K$ and a $\rho$-equivariant map $f_K:M\longrightarrow M'$. 
	
	We consider now the subgroup $\rho'(K)\leq G$. We note that $[G:\rho'(K)]\leq C\cdot D$, that $\rho'_{|K}\circ\rho: \tilde{K}\longrightarrow \rho'(K)$ is a surjective group morphism and that $g_H\circ f_K$ is $\rho'_{|K}\circ\rho$-equivariant and homotopic to $g\circ f$. Consequently, $g\circ f$ is an importing map.
\end{proof}

\begin{lem}\label{pI kernel non-zero degree}
	Let $M$ and $M'$ be closed oriented manifolds and $f:M\longrightarrow M'$ a non-zero degree importing map. Then $|\ker\rho|\leq \deg(f)$.
\end{lem}

\begin{proof}
	Since $f_H$ is $\rho$-equivariant, we have a commutative diagram 
	\[
	\begin{tikzcd}
		M \ar{r}{f_H}\ar{d}{\pi} & M'\\
		M/\ker\rho\ar{ru}{\overline{f}_H} &
	\end{tikzcd}
	\]
	
	By \cref{quotient map image degree}, $\pi^*(H^n(M/\ker\rho,\Z))\subseteq \deg(f)H^n(M,\Z)$. Therefore, $|\ker\rho|$ divides $\deg(f)$. Thus, $|\ker\rho|\leq \deg(f)$.
\end{proof}

The main example of importing maps are coverings of manifolds.

\begin{lem}\label{pI covering maps}
	Let $p:M\longrightarrow M'$ be a finite covering between closed oriented manifolds. Then $p$ imports group actions.
\end{lem}

\begin{proof}
	Assume that $p:M\longrightarrow M'$ is a $n$-sheeted covering and $G$ is a finite group acting effectively on $M'$. Then $G$ also acts on $\operatorname{Cov}_n(M')$, the set of $n$-sheeted coverings of $M'$ by pull-backs. On the other hand  $\operatorname{Cov}_n(M')\cong\Hom(\pi_1(M'),S_n)/\sim$ where $S_n$ is the $n$-th symmetric group and the equivalence relation is given by conjugation of elements of $S_n$. Therefore $\operatorname{Cov}_n(M')$ is finite, which implies that there exists a constant $C$ only depending on $M'$ and $n$ such that any finite group $G$ acting effectively on $M'$ has a subgroup $H$ which acts trivially on $\operatorname{Cov}_n(M')$ and $[G:H]\leq C$. Then we can lift the action of $H$ on $M'$ to an effective action of a group $\tilde{H}$ on $M$. In addition, there exists a surjective group morphism $\rho:\tilde{H}\longrightarrow H$ which makes the covering map $p:M\longrightarrow N$ $\rho$-equivariant. 
\end{proof}

We have the analogous result to \cref{pE jordan and discsym} for importing maps.

\begin{thm}\label{pI jordan and discsym}
	Let $M$ and $M'$ closed oriented manifolds which admit an importing map $f:M\longrightarrow M'$. Then:
	\begin{itemize}
		\item[1.] If $\Homeo(M)$ is Jordan, then $\Homeo(M')$ is Jordan.
		\item[2.] $\D(M')\leq \D(M)$.
		\item[3.] If $M$ has the almost fixed point property, then $M'$ has the almost fixed point property.
		\item[4.] If $M$ is almost asymmetric, then $M'$ is almost asymmetric.
	\end{itemize} 
\end{thm}

\begin{proof}
	The proof of items 1. and 2. are the same as in the case of finite coverings (see \cite{IgnasiMundetiRiera2010Jtft,mundet2021topological}). We prove the third part in detail.
	
	Let $G$ be a group acting effectively on $M'$. Let $C$ be the constant of the definition of the importing map $f$ and $H\leq G$, $f_H:M\longrightarrow M'$ and $\rho:\tilde{H}\longrightarrow H$ the data provided by the definition. Recall that $[G:H]\leq C$ and $\rho$ is surjective. The group $\tilde{H}$ acts effectively on $M$. Since $M$ has the almost fixed point property with constant $D$ there exists $x\in M$ such that $[\tilde{H}:\tilde{H}_x]\leq D$. We use now that $f_H$ is $\rho$-equivariant, therefore $\phi(\tilde{H}_x)\leq H_{f_H(x)}$. Since $\rho$ is surjective, $[H:H_{f_H(x)}]\leq [H:\rho(\tilde{H}_x)]\leq D$. Finally. we have $[G:G_{f_H(x)}]\leq [G:H_{f_H(x)}]\leq C\cdot D$. We have seen that $M'$ has the almost fixed point property with constant $C\cdot D$.
	
	The proof of the fourth part is analogous to the third part.
\end{proof}

With \cref{pI jordan and discsym} it is straightforward to find example of maps which are not importing. For example, since $\Homeo(T^4)$ is Jordan and $\Homeo(T^2\times S^2)$ is not Jordan, we have:

\begin{cor}
	Any non-zero degree map $f:T^4\longrightarrow T^2\times S^2$ is not an importing map.
\end{cor}

\begin{cor}
	Let $M$ and $M'$ be closed oriented manifolds and $f:M\longrightarrow M'$ a map which exports and imports group actions. Then $\D(M)=\D(M')$ and $\Homeo(M)$ is Jordan if and only if $\Homeo(M')$ is Jordan.
\end{cor}

We will see that finite coverings between nilmanifolds are an example of importing and exporting map.

\section[Large finite group actions and non-zero degree maps to nilmanifolds]{Large finite group actions on manifolds admitting a non-zero degree map to a nilmanifold}\label{sec: group action on hypernilmanifolds}

The goal of this section is to prove \cref{main theorem2 intro} and to give some applications of it. 

\subsection{Proof of \cref{main theorem2 intro}}

The main tool to prove \cref{main theorem2 intro} is the following result:

\begin{thm}\label{thm: exporting map hypernilmanifolds}
	Let $M$ be a closed oriented connected manifold and $f:M\longrightarrow N/\Gamma$ a non-zero degree map to a nilmanifold $N/\Gamma$. Then $f$ is an exporting map.
\end{thm}

We divide the proof of \cref{thm: exporting map hypernilmanifolds} in three parts. 

\textbf{Part 1:} In the first part we reduce the proof of \cref{thm: exporting map hypernilmanifolds} to the case where $f:M\longrightarrow N/\Gamma$ induces a surjective map between fundamental groups. 

First, we note that $f_*(\pi_1(M))$ is a finite index subgroup of $\Gamma$. Indeed, consider the diagram
\[\begin{tikzcd}
	& X \ar{d}{p}\\
	M\ar{ru}{f'}\ar{r}{f} & N/\Gamma
\end{tikzcd}\]
where $X$ is the covering space of $N/\Gamma$ associated to the subgroup $f_*(\pi_1(M))\leq \Gamma$ and $f'$ is a lift of $f$ (which exists by general properties of covering spaces). Since $f=p\circ f'$, the non-zero map  $f^*:H^n(N/\Gamma,\Z)\longrightarrow H^n(M,\Z)$ factors through $H^n(X,\Z)$. This implies that $H^n(X,\Z)$ is not zero and therefore $X$ is a  compact manifold. Since $p:X\longrightarrow N/\Gamma$ is a covering between compact manifolds, $p$ is a finite covering and $[\Gamma:f_*(\pi_1(M))]$ is finite.

Consequently, $f_*(\pi_1(M))=\Gamma'$ is also a lattice of $N$ and $X\cong N/\Gamma'$. By \cref{pE composition}, if we prove that $f'$ and $p$ are exporting maps then $f$ will be an exporting map.

\begin{lem}\label{pE nilmanifold covering}
	The covering $p$ is an exporting map. 
\end{lem}

\begin{proof}
	Since $\Gamma'$ is finitely generated and nilpotent, the group $\Out(\Gamma')$ is Minkowski by \cite{Wehrfritz1994}. Consequently, there exists a constant $C$ depending only on $\Gamma'$ satisfying that, if $G$ is a finite group acting effectively on $N/\Gamma'$ and $H$ is the kernel of the group morphism $\psi:G\longrightarrow \Out(\Gamma')$, then $[G:H]\leq C$. By \cite[Theorem 3.1.16]{lee2010seifert}, $H$ is isomorphic to a subgroup of the torus $\Zc N/\Zc \Gamma'$, and hence it is abelian. 
	
	Let $G$ be a finite group acting effectively on $N/\Gamma'$ and $H=\ker(\psi:G\longrightarrow \Out(\Gamma'))$. We claim that the action of $H$ on $N/\Gamma'$ is free. Given a point $x\in N/\Gamma'$, the isotropy subgroup $H_x$ injects into $\Aut(\Gamma')$. Since the action of $H$ on $N/\Gamma'$ is inner, $H_x$ is a subgroup of $\Inn(\Gamma')=\Gamma'/\Zc\Gamma'$. Since $\Gamma'$ is a finitely generated torsion-free nilpotent group, $\Inn(\Gamma')$ is torsion-free and therefore $H_x$ is trivial, as claimed. 
	
	The action of $H$ can be conjugated to an action of $H\leq\Zc N/\Zc \Gamma'$ obtained by restricting the standard torus action of $\Zc N/\Zc\Gamma'$ on $N/\Gamma'$ (see \cite[Remark 11.7.18.2]{lee2010seifert}). Thus, we can assume without loss of generality that $H\leq\Zc N/\Zc \Gamma'$ and the action of $H$ on $N/\Gamma'$ is induced by the restriction of the torus action $\Zc N/\Zc\Gamma'$ on $N/\Gamma'$. 
	
	The covering $p:N/\Gamma'\longrightarrow N/\Gamma$ is given by $p(n\Gamma')=n\Gamma$ for $n\in N$ and hence it is $N$-equivariant. Since $\Zc\Gamma'=\Zc N\cap \Gamma'$ and $\Zc\Gamma=\Zc N\cap \Gamma$, we have that $\Zc\Gamma'\leq\Zc\Gamma\leq \Zc N$. Consequently, we have a group morphism $\rho:\Zc N/\Zc\Gamma'\longrightarrow \Zc N/\Zc\Gamma$ between tori of the same dimension. 	
	
	We note that $p:N/\Gamma'\longrightarrow N/\Gamma$ is $\rho$-equivariant. Since $H\leq \Zc N/\Zc\Gamma'$, we have an action of $H$ on $N/\Gamma$ 
	given by $h(n\Gamma)=\rho(h)(n\Gamma)$. This implies that $p$ is an exporting map, as we wanted to prove.
\end{proof}

In particular, $p:N/\Gamma'\longrightarrow N/\Gamma$ is an exporting map. It remains to prove that $f'$ is also an exporting map. 

Thus, from now on we will assume that $f:M\longrightarrow N/\Gamma$ induces a surjective morphism $f_*:\pi_1(M)\longrightarrow \Gamma$.

\textbf{Part 2:} Let $i:\Gamma\hookrightarrow N$ denote the inclusion of the lattice $\Gamma$ in $N$. For this part of the proof we consider the set of isomorphism classes $N$-local systems $X(\pi_1(M),N)=\Hom(\pi_1(M),N)/\!\!\sim$, where $\sim$ denotes the equivalence relation given by the conjugation by elements of $N$. An effective action of a finite group $G$ on $M$ induces an action of $G$ on $X(\pi_1(M),N)$ (this action is described explicitly in the proof of \cref{pE local systems}). Our goal is to prove that if $G$ fixes the class $[i\circ f_*]\in X(\pi_1(M),N)$ then there exists an action of $G$ on $N/\Gamma$ and a $G$-equivariant map $f_G:M\longrightarrow N/\Gamma$ which is homotopic to $f$. We will use induction on the dimension of $N$.


We start with the following lemma, which is a generalization of \cite[Lemma 4.2]{mundet2021topological} to non-compact manifolds. The arguments used to prove \cref{pE S1} and \cite[Lemma 4.2]{mundet2021topological} are essentially the same. 


\begin{lem}\label{pE S1}
	Let $M$ be a connected manifold, let $f:M\longrightarrow S^1$ be a continuous map and let $\theta$ be a generator of $H^1(S^1,\Z)$. Suppose that $H$ is a finite group of cardinal $r$ acting effectively on $M$ fixing $f^*\theta$. Consider the group extension $$1\longrightarrow \pi_1(M)\longrightarrow \tilde{H}\longrightarrow H\longrightarrow 1,$$ where the group $\tilde{H}$ acts effectively on the universal cover $\tilde{M}$. Then there exists a group morphism $\tilde{\mu}:\tilde{H}\longrightarrow \R$ and a $\tilde{\mu}$-equivariant map $\tilde{f}_H: \tilde{M}\longrightarrow \R$ such that $\tilde{\mu}_{|\pi_1(M)}=f_*$.
\end{lem}

\begin{proof}
	Define $F:M \to S^1$ by $\zeta(x) = \sum_{h \in H} f(h \cdot x)$ for every $x \in M$. Then $F$ is continuous and constant on the orbits of the action of $H$. Let $\phi_h: M \to M$ be the homeomorphism induced by the action of $h\in H$. By assumption, $\phi_h^* f^* \theta = f^* \theta$ for every $h\in H$. We have $F^* \theta = \sum_{h \in H} \phi_h^* (f^* \theta) = r f^* \theta$.
	
	Consider the $\Z/r$-covering $p_r:S^1\longrightarrow S^1$ defined as $p_r(t)=rt$. Let $q_r:M_r \longrightarrow M$ the pull-back of $p_r$ by $F$. Recall that $M_r=\{(x, t) \in M \times S^1 :F(x)=rt\}$ and $q_r:M_r \longrightarrow M$ satisfies $q_r(x,t)=x$ for $(x,t)\in M_r$. Moreover, $q_r:M_r \longrightarrow M$ has a structure of principal $\Z/r$ given by the action $(x,t)\cdot\alpha=(x,t\alpha)$ for $(x,t)\in M_r$ and $\alpha\in\Z/r$. 
	
	We claim that $q_r$ is a trivial principal bundle. This is equivalent to the triviality of the monodromy of $q_r$, which we denote by $\nu:\pi_1(M, x_0) \to \Z/r$, where $x_0\in M$ is an arbitrary base point. If there existed some $\gamma \in \pi_1(X, x_0)$ such that $\nu(\gamma) \neq 0$, then the pairing of $f^* \theta$ with $[\gamma] \in H_1(X,\Z)$ would not be divisible by $r$, which contradicts $F^* \theta = r f^* \theta$.
	Hence $\nu$ is trivial and consequently, the bundle $q_r$ is a trivial, so we may choose a section $\sigma: M\longrightarrow  M_r$. 
	
	Define $f_H: M\longrightarrow S^1$ by the condition that $\sigma(x) = (x, f_H(x))$. Then $f_H$ is continuous and we have $rf_H(x)=F(x)$ for every $x\in M$. For any $h\in H$, define $\chi_h: M \to S^1$ by $$\chi_h(x) = f_H(h \cdot x) - f_H(x).$$
	
	We have $$r\chi_h(x) = r f_H(h \cdot x) - r f_H(x) = F(h\cdot x) - F(x) = 0$$
	because $F$ is $H$-invariant. Hence $\chi_h$ takes values in $\Z/r\leq S^1$. Consequently, since $\chi_h$ is continuous, it is a constant map. We may thus define a map $\mu: H \longrightarrow \Z/r$ by the condition that $\mu(h)= \chi_h(x)$ for every $x \in X.$
	
	Let $h,h'\in H$ and let $x \in X$. We have $$
	\chi_{hh'}(x) = f_H(hh' \cdot x) - f_H(x) = f_H(hh' \cdot x) - f_H(h' \cdot x) + f_H(h' \cdot x) - f_H(x) = \chi_h(h' \cdot x) + \chi_{h'}(x),$$
	which proves that $\mu(hh') = \mu(h) + \mu(h')$, so $\mu$ is a morphism of groups. 
	From the definition of $\mu$, it is immediate that $f_H$ is $\mu$-equivariant.
	
	To conclude the proof, note that $rf_H^* \theta = F^* \theta = rf^* \theta$. 
	Since $H^1(X; \mathbb{Z})$ has no torsion, we conclude that $f_H^* \theta = f^*\theta$. Thus, $f_*:\pi_1(M)\longrightarrow \Z$ and $f_{H*}:\pi_1(M)\longrightarrow \Z$ are equal. We can lift $f_H$ to a map between the universal coverings $\tilde{f}_H:\tilde{M}\longrightarrow \R$. Since $f_H$ is $\mu$-equivariant, we have a commutative diagram 
	\[\begin{tikzcd}
		1\ar{r}{}&\pi_1(M)\ar{r}{}\ar{d}{f_{H*}}&\tilde{H}\ar{r}{}\ar{d}{\tilde{\mu}}&H\ar{r}{}\ar{d}{\mu}&1\\
		1\ar{r}{}&\Z\ar{r}{}&\Z\ar{r}{\tilde{\pi}}&\Z/r\ar{r}{}&1.
	\end{tikzcd}\]
\end{proof}

\begin{lem}\label{pE local systems}
	Let $M$ be a closed connected oriented manifold and let $f:M\longrightarrow N/\Gamma$ be a map such that $f_*(\pi_1(M))=\Gamma$. Assume that $H$ is a finite group acting effectively on $M$ which fixes the local system $[i\circ f_*]\in X(\pi_1(M),N)$. Consider the group extension $$1\longrightarrow \pi_1(M)\longrightarrow \tilde{H}\longrightarrow H\longrightarrow 1,$$ where the group $\tilde{H}$ acts effectively on the universal cover $\tilde{M}$. Then there exists a group morphism $\tilde{\mu}:\tilde{H}\longrightarrow N$ and a $\tilde{\mu}$-equivariant map $\tilde{f}_H: \tilde{M}\longrightarrow N$ such that $\tilde{\mu}_{|\pi_1(M)}=f_*$.
\end{lem}

Before proving \cref{pE local systems}, we deduce the next corollary.

\begin{cor}
	Let $M$ be a closed connected oriented manifold and $f:M\longrightarrow N/\Gamma$ be a non-zero degree map such that $f_*(\pi_1(M))=\Gamma$. Assume that $H$ is a finite group acting effectively on $M$ which fixes the local system $[i\circ f_*]\in X(\pi_1(M),N)$. Then there exists a group action of $H$ on $N/\Gamma$ and an equivariant map $f_H:M\longrightarrow N/\Gamma$ homotopic to $f$.
\end{cor}

\begin{proof}
	Since $\tilde{\mu}_{|\pi_1(M)}=i\circ f_*$, the map $\tilde{f}_H:\tilde{M}\longrightarrow N$ induces a map $f_H:M\longrightarrow N/\Gamma$. Moreover, $\tilde{\mu}:\tilde{H}\longrightarrow N$ induces a surjective group morphism $\mu':\tilde{H}/\pi_1(M)=H\longrightarrow \tilde{\mu}(\tilde{H})/\Gamma$ such that $f_H$ is $\mu'$-equivariant. Thus, we have an action $H\longrightarrow \tilde{\mu}(\tilde{H})/\Gamma\longrightarrow \Homeo(N/\Gamma)$ and $f_H$ is $H$-equivariant. Finally, since $N/\Gamma$ is a model of the Eilenberg-MacLane space $K(\Gamma,1)$ and $[\tilde{\mu}_{|\pi_1(M)}]=[i\circ f_*]$, the maps $f$ and $f_H$ are homotopic. 
\end{proof}

\begin{rem}
	Note that the action of $\tilde{\mu}(\tilde{H})/\Gamma$ is free and the quotient $(N/\Gamma)/(\tilde{\mu}(\tilde{H})/\Gamma) $ is the nilmanifold $N/\tilde{\mu}(\tilde{H})$.
\end{rem}

\begin{proof}[Proof of \cref{pE local systems}]
	We proceed by induction on the dimension of the nilmanifold $N/\Gamma$. Assume that $\dim(N/\Gamma)=1$. Then $N/\Gamma\cong S^1$ and the statement is a consequence of \cref{pE S1}. For the induction step, we can take central exact sequences 
	\[\begin{tikzcd}
		1\ar{r}{}&\R\ar{r}{}&N\ar{r}{\tilde{\pi}}&N'\ar{r}{}&1
	\end{tikzcd}\]
	and 
	\[\begin{tikzcd}
		1\ar{r}{}&\R\cap\Gamma\cong\Z\ar{r}{}&\Gamma\ar{r}{\pi_*}&\Gamma'\ar{r}{}&1,
	\end{tikzcd}\] where $\pi_*=\tilde{\pi}_{|\Gamma}$. They induce a principal $S^1$-bundle, $\pi:N/\Gamma\longrightarrow N'/\Gamma'$. Since the short exact sequences are central, we can choose a normalized 2-cocycle $c:N'\times N'\longrightarrow \R$ such that $N\cong \R\times_c N'$ and $\Gamma\cong \Z\times_{c_{|\Gamma'\times\Gamma'}}\Gamma'$. Note that $\dim(N'/\Gamma')=\dim(N/\Gamma)-1$.
	
	We denote $\mu=i\circ f_*$. We have an action of $\tilde{H}$ on $X(\pi_1(M),N)$ satisfying that $\tilde{h}[\nu]=[\nu\circ c_{\tilde{h}|\pi_1(M)}]$ for all $\tilde{h}\in\tilde{H}$ and $[\nu]\in X(\pi_1(M),N)$. If the action of $H$ on $X(\pi_1(M),N)$ fixes $[\mu]$ then the action of $\tilde{H}$ also fixes $[\mu]$. This is a consequence of the fact that the action of $\tilde{H}$ on $X(\pi_1(M),N)$ factors through the action of $H$ on $X(\pi_1(M),N)$, since the restriction to $\pi_1(M)$ of the action of $\tilde{H}$ on $X(\pi_1(M),N)$ is trivial.
	
	This condition implies that there exists a map (which is not a group morphism in general) $a:\tilde{H}\longrightarrow N$ such that $\mu\circ c_{\tilde{h}|\pi_1(M)}=c_{a(\tilde{h})}\circ \mu$. Note that $a:\tilde{H}\longrightarrow N$ is not unique. If we have another map $a':\tilde{H}\longrightarrow N$ with the same property, then $c_{a'(\tilde{h})}\circ \mu=\mu\circ c_{\tilde{h}|\pi_1(M)} =c_{a(\tilde{h})}\circ \mu$ and hence $c_{a(\tilde{h})^{-1}a'(\tilde{h})}\circ \mu=\mu$. Therefore, $a(\tilde{h})^{-1}a'(\tilde{h})$ centralizes the lattice $\Gamma$ and hence $a(\tilde{h})^{-1}a'(\tilde{h})\in\Zc N$ for all $\tilde{h}\in\tilde{H}$.
	
	The action of $\tilde{H}$ fixes $[\tilde{\pi}\circ \mu]\in X(\pi_1(M),N')$. Since $\pi_*:\Gamma\longrightarrow\Gamma'$ is surjective, the group morphism $\pi_*\circ f_*:\pi_1(M)\longrightarrow \Gamma'$ is surjective and, by induction, there exists a group morphism $\tilde{\mu}_B:\tilde{H}\longrightarrow N'$ and a $\tilde{\mu}_B$-equivariant map $\tilde{f}'_B:\tilde{M}\longrightarrow N'$ such that $\tilde{\mu}_{B|\pi_1(M)}=\tilde{\pi}\circ \mu$. 
	
	Consequently, we have a surjective group morphism $\mu_B:\tilde{H}/\pi_1(M)\cong H\longrightarrow \tilde{\mu}_B(\tilde{H})/\Gamma'$. Thus, we can define an action of $H$ on $N'/\Gamma'$ (which may not be effective) given by $h(n\Gamma')=\mu_B(h)(n\Gamma')$. The map $\tilde{f}'_B:\tilde{M}\longrightarrow N'$ induces an $H$-equivariant map $f'_B:M\longrightarrow N'/\Gamma'$ homotopic to $\pi\circ f$. Since $f'_{B*}:\pi_1(M)\longrightarrow \Gamma'$ factors through $\pi_*:\Gamma\longrightarrow\Gamma'$, there exists a map $f_B:M\longrightarrow N/\Gamma$ homotopic to $f$ such that $\pi\circ f_B=f'_B$.
	We can lift $f_B$ to a map $\tilde{f}_B:\tilde{M}\longrightarrow N\cong \R\times_c N'$. Note that $\tilde{f}_B$ is $\mu$-equivariant and $\tilde{\pi}\circ \tilde{f}_B$ is $\tilde{\mu}_B$-equivariant.
	
	We now consider the group $\tilde{H}_0=\ker\tilde{\mu}_B$ and $Z_0=\mu^{-1}(\ker\tilde{\pi})\leq\pi_1(M)$ (recall that $\ker\tilde{\pi}\cong\R$). Notice that $Z_0=\tilde{H}_0\cap \pi_1(M)$. Let $g\in \tilde{H}_0\cap \pi_1(M)$. Then $\tilde{\mu}_B(g)=\tilde{\pi}(\mu(g))=0$, which implies that $g\in Z_0$. Conversely, if $g\in Z_0$, then $g\in \pi_1(M)$. Therefore, $\tilde{\mu}_B(g)=\tilde{\pi}(\mu(g))=0$, which implies that $g\in\tilde{H}_0$. Consequently, $g\in \tilde{H}_0\cap \pi_1(M)$. In conclusion, $Z_0=\mu^{-1}(\ker\tilde{\pi})\leq\pi_1(M)$. Therefore $Z_0$ is a normal subgroup of $\tilde{H}_0$. Moreover, $H_0=\tilde{H}_0/Z_0$ is a subgroup of $H$ and hence $Z_0$ is a finite index subgroup of $\tilde{H}_0$. 
	
	Since $Z_0\leq\pi_1(M)$, $\tilde{f}_B:\tilde{M}\longrightarrow N$ induces a map $\tilde{M}/Z_0\longrightarrow N/\Z\cong S^1\times_c N'$. We obtain a map $f_Z:\tilde{M}/Z_0\longrightarrow S^1$ by composing it with the projection to the $S^1$ factor. We also get an effective action of the finite group $H_0$ on $\tilde{M}/Z_0$. We consider $\mu_Z:Z_0\longrightarrow \R$, which is the composition $f_{Z*}:Z_0\longrightarrow \Z$ and the inclusion $i_Z=i_{|\Z}:\Z\longrightarrow \R\leq \Zc N$. The group morphism $\mu_Z$ is the restriction of the map $\mu$ to $Z_0$, $\mu_Z=\mu_{|Z_0}$. 
	
	The action of $H_0$ on $X(Z_0,\R)=\Hom(Z_0,\R)$ induces an action of $\tilde{H}_0$ on $X(Z_0,\R)$ satisfying $\tilde{h}\cdot\nu=\nu\circ c_{\tilde{h}|Z_0}$ for $\nu\in X(Z_0,\R)$ and $\tilde{h}\in \tilde{H}_0$. This action satisfies $$\tilde{h}\cdot\mu_{Z}=\mu_{|Z_0}\circ c_{\tilde{h}|Z_0}=c_{a(\tilde{h})}\circ \mu_{|Z_0}=\mu_{|Z_0}$$ where the last equality holds because the image of $\mu_{|Z_0}$ lies on the center of $N$. Consequently, the group $\tilde{H}_0$ (and hence $H_0$) fixes $\mu_Z\in X(Z_0,\R)$. We can use \cref{pE S1} to conclude that there exists a group morphism $\tilde{\mu}_Z:\tilde{H}_0\longrightarrow \R$ satisfying $\tilde{\mu}_{Z|Z_0}=\mu_Z$ and a $\tilde{\mu}_Z$-equivariant map $\tilde{f}_Z:\tilde{M}\longrightarrow \R$. 
	
	Summarizing the proof until this point, we have obtained two group morphisms $\tilde{\mu}_B$ and $\tilde{\mu}_Z$ satisfying
	\[\begin{tikzcd}
		1\ar{r}{}&\tilde{H}_0\ar{r}{}\ar{d}{\tilde{\mu}_Z}&\tilde{H}\ar{r}{\tilde{\pi}}&\tilde{H}/\tilde{H}_0\ar{r}{}\ar{d}{\overline{\mu}_B}&1\\
		1\ar{r}{}&\R\ar{r}{}&N\ar{r}{\tilde{\pi}}&N'\ar{r}{}&1
	\end{tikzcd}\]
	where $\overline{\mu}_B:\tilde{H}/\tilde{H}_0\longrightarrow N'$ is an injective group morphism induced by $\tilde{\mu}_B$. We also have a map $\tilde{f}:\tilde{M}\longrightarrow N$ of the form $\tilde{f}=(\tilde{f_Z},\tilde{f}'_B)$. Our goal is to construct a group morphism $\tilde{\mu}:\tilde{H}\longrightarrow N$ making the above diagram commutative.
	
	Firstly, we note that $\ker\mu$ is a normal subgroup of $\tilde{H}$, since the action of $\tilde{H}$ fixes $[\mu]$. Let $\tilde{\Gamma}=\tilde{H}/\ker\mu$, let $\tilde{\Gamma}_0=\tilde{H}_0/\ker\mu$ and let $\tilde{\nu}_Z:\tilde{\Gamma}_0\longrightarrow \R$ be the group morphism induced by $\tilde{\mu}_Z$. We have a commutative diagram
	\[\begin{tikzcd}
		1\ar{r}{}&\tilde{\Gamma}_0\ar{r}{}\ar{d}{\tilde{\nu}_Z}&\tilde{\Gamma}\ar{r}{\tilde{\pi}}&\tilde{H}/\tilde{H}_0\ar{r}{}\ar{d}{\overline{\mu}_B}&1\\
		1\ar{r}{}&\R\ar{r}{}&N\ar{r}{\tilde{\pi}}&N'\ar{r}{}&1
	\end{tikzcd}\]
	Let $K$ denote $\ker \tilde{\nu}_Z=\ker\tilde{\mu}_Z/\ker\mu$. We note that $K$ injects in $H_0$ and therefore it is finite. Since $\tilde{\nu}_Z(\tilde{\Gamma}_0)$ is torsion-free, the subgroup $K$ is characteristic in $\tilde{\Gamma}_0$. Thus, $K$ is a normal subgroup of $\tilde{\Gamma}$. Let $\Lambda=\tilde{\Gamma}/K$, $\Lambda_0=\tilde{\Gamma}_0/K\cong\Z$ and $\Lambda'=\tilde{H}/\tilde{H}_0$. We also denote by $\overline{\mu}_B: \Lambda_0\longrightarrow \R$ the group morphism induced by $\tilde{\nu}_Z$. We have a commutative diagram 
	\[\begin{tikzcd}
		1\ar{r}{}&{\Lambda}_0\ar{r}{}\ar{d}{\overline{\mu}_Z}&\Lambda\ar{r}{\tilde{\pi}}&\Lambda'\ar{r}{}\ar{d}{\overline{\mu}_B}&1\\
		1\ar{r}{}&\R\ar{r}{}&N\ar{r}{\tilde{\pi}}&N'\ar{r}{}&1
	\end{tikzcd}\]
	Moreover, note that $\Gamma\cap K$ is the identity element, therefore $\Gamma$ is a finite index subgroup of $\Lambda$. Note also that $\Lambda$ is torsion-free, finitely generated and nilpotent, hence $\Lambda$ is a lattice of $N$. Using the inclusion $i:\Gamma\longrightarrow N$, we can define inclusions $i_B:\Gamma'\longrightarrow N'$ and $i_Z:\Gamma_0\cong\Z\longrightarrow \R$, obtaining the following commutative diagram
	\[\begin{tikzcd}
		1\ar{r}{}&{\Gamma}_0\ar{rr}{}\ar[hook]{rd}{}\ar[swap]{rdd}{i_Z}&&\Gamma\ar{rr}{}\ar[hook]{rd}{}\ar[swap]{rdd}{i}&&\Gamma'\ar{r}{}\ar[hook]{rd}{}\ar[swap]{rdd}{i_B}&1&\\
		&1\ar{r}{}&{\Lambda}_0\ar{rr}{}\ar{d}{\overline{\mu}_Z}&&\Lambda\ar{rr}{\tilde{\pi}}&&\Lambda'\ar{r}{}\ar{d}{\overline{\mu}_B}&1\\
		&1\ar{r}{}&\R\ar{rr}{}&&N\ar{rr}{\tilde{\pi}}&&N'\ar{r}{}&1
	\end{tikzcd}\]
	
	The inclusion $\Gamma\leq \Lambda$ induces the identity map between the real Mal'cev completions $\Gamma_\R$ and $\Lambda_\R$ (see the survey \cite{dekimpe2016users} for the definition and main properties of the Mal'cev completion of torsion-free finitely generated nilpotent groups). Consequently, we have a commutative diagram
	\[\begin{tikzcd}
		1\ar{r}{}&({\Gamma}_0)_\R\ar{rr}{}\ar{rd}{Id}\ar[swap]{rdd}{(i_Z)_\R}&&\Gamma_\R\ar{rr}{}\ar{rd}{Id}\ar[swap]{rdd}{i_\R}&&\Gamma'_\R\ar{r}{}\ar{rd}{Id}\ar[swap]{rdd}{(i_B)_\R}&1&\\
		&1\ar{r}{}&({\Lambda}_0)_\R\ar{rr}{}\ar{d}{(\overline{\mu}_Z)_\R}&&\Lambda_\R\ar{rr}{\tilde{\pi}}&&\Lambda'_\R\ar{r}{}\ar{d}{(\overline{\mu}_B)_\R}&1\\
		&1\ar{r}{}&\R\ar{rr}{}&&N\ar{rr}{\tilde{\pi}}&&N'\ar{r}{}&1
	\end{tikzcd}\]
	Therefore, if we set $\overline{\mu}=i_{\R|\Lambda}:\Lambda\longrightarrow N$, we obtain a commutative diagram
	\[\begin{tikzcd}
		1\ar{r}{}&{\Lambda}_0\ar{r}{}\ar{d}{\overline{\mu}_Z}&\Lambda\ar{r}{\tilde{\pi}}\ar{d}{\overline{\mu}}&\Lambda'\ar{r}{}\ar{d}{\overline{\mu}_B}&1\\
		1\ar{r}{}&\R\ar{r}{}&N\ar{r}{\tilde{\pi}}&N'\ar{r}{}&1
	\end{tikzcd}\]
	The desired group morphism $\tilde{\mu}:\tilde{H}\longrightarrow N$ is obtained by composing the $\overline{\mu}$ with the projection $\tilde{H}\longrightarrow\Lambda$. By construction, the map $\tilde{f}:\tilde{M}\longrightarrow N$ is $\tilde{\mu}$-equivariant and $\tilde{\mu}_{|\pi_1(M)}=\mu$, as we wanted to see.
\end{proof}

\textbf{Part 3:} The last step of the proof of \cite[Theorem 4.1]{mundet2021topological} is a consequence of Minkowski's lemma.

\begin{lem}\label{pE local systems trivial action}
	Let $M$ be a closed connected oriented manifold and let $f:M\longrightarrow N/\Gamma$ be a non-zero degree map such that $f_*(\pi_1(M))=\Gamma$. There exists a constant $C$ such that any finite group $G$ acting effectively on $M$ has a finite subgroup $H\leq G$ such that $[G:H]\leq C$ and $H$ acts trivially on $X(\pi_1(M),N)$.
\end{lem}

\begin{proof}
	We denote by $\pi_1(M)^j$ the $j$-th element of the lower central series. Assume that the nilpotency class of $N$ is $c$, then for any $\nu\in \Hom(\pi_1(M),N)$ we have $\pi_1(M)^{c}\leq \ker \nu$. In consequence, the projection $\pi_1(M)\longrightarrow \pi_1(M)/\pi_1(M)^c$ induces a bijection $\Hom(\pi_1(M)/\pi_1(M)^c,N)\longrightarrow \Hom(\pi_1(M),N)$. This bijection descends to a bijection between the sets of isomorphism classes of $N$-local systems $X(\pi_1(M),N)$ and $X(\pi_1(M)/\pi_1(M)^c,N)$.
	
	Since $\pi_1(M)^{c}$ is a characteristic subgroup of $\pi_1(M)$, any automorphism $\phi:\pi_1(M)\longrightarrow \pi_1(M)$ induces and automorphism $\overline{\phi}:\pi_1(M)/\pi_1(M)^{c}\longrightarrow \pi_1(M)/\pi_1(M)^{c}$ given by $\overline{\phi}(\gamma\pi_1(M)^{c})=\phi(\gamma)\pi_1(M)^{c}$. Hence there is a group morphism $\rho:\Out(\pi_1(M))\longrightarrow \Out(\pi_1(M)/\pi_1(M)^{c})$ which sends $[\phi]$ to $[\overline{\phi}]$. 
	
	Notice that the bijection between $X(\pi_1(M),N)$ and $X(\pi_1(M)/\pi_1(M)^c,N)$ is $\rho$-equivariant. If $\psi: G\longrightarrow \Out(\pi_1(M))$ is the map induced by the action of $G$ on $M$, then we have an action of $G$ on $X(\pi_1(M)/\pi_1(M)^c,N)$ given by $[f]g=[f](\rho\circ\psi)(g)$. Since $\pi_1(M)/\pi_1(M)^{c}$ is finitely generated and nilpotent, $ \Out(\pi_1(M)/\pi_1(M)^c)$ is Minkowski by \cite{Wehrfritz1994}. Hence, there exists a constant $C$ such that any finite subgroup of $\Out(\pi_1(M)/\pi_1(M)^{c})$ is at most of order $C$. Thus, $H=\ker(\rho\circ\psi)$ is a subgroup of $G$ such that $[G:H]\leq C$ which acts trivially on $X(\pi_1(M)/\pi_1(M)^c,N)$ and hence it also acts trivially on $X(\pi_1(M),N)$.
\end{proof}

By combining \cref{pE local systems}, \cref{pE local systems trivial action} and \cref{pE nilmanifold covering} we complete the proof of \cref{thm: exporting map hypernilmanifolds}. 

Finally, we prove \cref{main theorem2 intro}. 

\begin{proof}[Proof of \cref{main theorem2 intro}]
	By \cite[Theorem 1.6, Theorem 1.9]{daura2024actions}, we have that:
	\begin{itemize}
		\item[1.] $\Homeo(N/\Gamma)$ is Jordan.
		\item[2.] $\D(N/\Gamma)\leq\rank\Zc\Gamma$.
		\item[3.] $N/\Gamma$ has small stabilizers.
	\end{itemize}
	
	Thus, the first part of \cref{main theorem2 intro} is a direct consequence of \cref{pE jordan and discsym} and \cref{thm: exporting map hypernilmanifolds}. For the second part, the bound of the discrete degree of symmetry is also a consequence of \cref{pE jordan and discsym} and \cref{thm: exporting map hypernilmanifolds}. If we assume $\D(M)=n$, then $\rank\Zc\Gamma\geq n$. This implies that $\Gamma\cong\Z^n$ and $N/\Gamma\cong T^n$. In consequence, the theorem follows from \cref{large group actions hypertoral manifolds}.
	
	For the third part, note that $M$ has the almost fixed point property by \cref{almost fix point euler char}. Since $M$ has the small stabilizers property (by \cref{pE jordan and discsym} and \cref{thm: exporting map hypernilmanifolds}) and the almost fixed point property, $M$ is almost asymmetric, by \cref{small stab+almost fixed point=almost asymmetric}.
	
	The fourth part is a consequence of $M$ having the small stabilizers property and $\Homeo(M)$ being Jordan, see \cite[Lemma 3.6]{daura2024actions}.
\end{proof}

Note that the conclusion of \cite[Corollary 1.6]{mundet2021topological} also holds for closed connected oriented manifolds admitting a non-zero degree map to a nilmanifold.

\begin{cor}
	Let $M$ be a closed connected oriented $n$-dimensional manifold admitting a non-zero degree map to a nilmanifold $f:M\longrightarrow N/\Gamma$. If $\D(M)=n$ and $\pi_1(M)$ is virtually solvable then $M$ is homeomorphic to $T^n$.
\end{cor}

Let us show some applications and remarks of \cref{main theorem2 intro}, which we divide in four subsections.

\subsection{The toral rank and stable Carlsson's conjecture}

The first application is \cref{TRC hypernilmanifolds intro}, where we show new examples where the toral rank conjectures and the stable Carlsson's conjecture are true. To do so, we need the following lemma.

\begin{lem}\label{TRC=SCC}
Let $M$ be a closed connected aspherical locally homogeneous space. The toral rank conjecture holds for $M$ if and only if the stable Carlsson's conjecture holds for $M$.
\end{lem}

\begin{proof}
Since $H^*(M,\Z)$ is finitely generated, we can use the universal coefficients theorem to conclude that there exists a constant $C_1$ such that for all prime $p>C_1$ we have $\dim H^*(M,\Q)=\dim H^*(M,\Z/p)$. By \cite[Theorem 1.9]{daura2024actions}, $\Out(\pi_1(M))$ and $\Aut(\pi_1(M))$ are Minkowski. Let $C_2$ and $C_3$ denote the Minkowski constant of $\Out(\pi_1(M))$ and $\Aut(\pi_1(M))$ respectively. 

For all primes $p>C_2$, we have that $\rank_p(M)\leq \rank \Zc\pi_1(M)$, by \cite[Theorem 3.1.16]{lee1987manifolds}. In addition, for all primes $p>C_3$ any effective action of a $p$-group on $M$ is free, by \cite[Corollary 3.1.17]{lee2010seifert}. Consequently, $\rank(M)\leq \rank_p(M)$ for $p>C_3$. Finally, $\rank(M)=\rank\Zc\pi_1(M)$ by \cite[Theorem 1.10]{daura2024actions} and \cite[Corollary 3.1.17]{lee2010seifert}. Thus, if $p>\max\{C_2,C_3\}$, then $\rank(M)=\rank_p(M)=\rank\Zc\pi_1(M)$. Set $C=\max\{C_1,C_2,C_3\}$.

We assume now that the toral rank conjecture holds for $M$, hence $\dim H^*(M,\Q)\geq 2^{\rank(M)}=2^{\rank(\Zc\pi_1(M))}$. For all primes $p>C$ we have $\dim H^*(M,\Z/p)=\dim H^*(M,\Q)\geq 2^{\rank(\Zc\pi_1(M))}=2^{\rank_p(M)}$ and therefore the stable Carlsson's conjecture holds for $p>C$.

Now, we prove the converse statement. Thus, we assume that there exists a constant $D$ such that for all primes $p>D$, we have $\dim H^*(M,\Z/p)\geq 2^{\rank_p(M)}$. Choose a prime $p>\max\{D,C\}$. Then $\rank_p(M)=\rank(M)$ and we have $\dim H^*(M,\Q)=\dim H^*(M,\Z/p)\geq 2^{\rank_p(M)}= 2^{\rank(M)}$. Hence, the toral rank conjecture holds for $M$. 
\end{proof}

\begin{proof}[Proof of \cref{TRC hypernilmanifolds intro}]
	Let $f:M\longrightarrow N/\Gamma$ be a non-zero degree map and assume that the toral rank conjecture holds for $N/\Gamma$. Recall that $\D(N/\Gamma)=\rank(N/\Gamma)=\rank\Zc\Gamma$ (this is a consequence of \cref{main theorem1 intro} and \cite[Corollary 11.7.4, Theorem 11.7.7]{lee2010seifert}). 
	Thus, $\rank(M)\leq \A(M)\leq \D(M)\leq \rank\Zc\Gamma$. Since $f$ has non-zero degree, the induced morphism $f^*:H^*(N/\Gamma,\Q)\longrightarrow H^*(M,\Q)$ is injective and therefore 
	$$\dim H^*(M,\Q)\geq \dim H^*(N/\Gamma,\Q)\geq 2^{\rank(\Zc\Gamma)}\geq 2^{\rank(M)}, $$
	as we wanted to see.
	
	Since the toral rank conjecture holds for $N/\Gamma$, the stable Carlsson's conjecture also holds for $N/\Gamma$ by \cref{TRC=SCC}. Let $C$ denote the constant provided stable Carlsson's conjecture.
	
	Let $D$ denote the exporting map constant of $f$. Note that if $(\Z/p)^r$ acts freely on $M$ with $p\geq \max\{D,\deg(f)\}$, then $(\Z/p)^r$ acts freely on $N/\Gamma$. In particular, $\rank_p(M)\leq \rank_p(N/\Gamma)$ for all $p>\max\{D,\deg(f)\}$. Let $C'=\max\{D,\deg(f),C\}$. Then for all $p>C'$ we have $$\dim H^*(M,\Z/p)\geq \dim H^*(N/\Gamma,\Z/p)\geq 2^{\rank_p(N/\Gamma)}\geq 2^{\rank_p(M)},$$
	and the stable Carlsson's conjecture holds for $M$.
\end{proof}

\subsection{Closed connected manifolds with non-zero degree maps to solvmanifolds}

\Cref{main theorem2 intro} is false if we replace nilmanifold by solvmanifold. In order to give an example of a non-zero degree map between solvmanifolds which does not export group actions we need the following elementary group theoretic lemma.

\begin{lem}\label{center semidriect product}
	Let $G_1$ and $G_2$ be groups and let $\phi:G_2\longrightarrow \Aut(G_1)$ be a group morphism. Then $$\Zc(G_1\rtimes_\phi G_2)=\{(g_1,g_2)\in G_1\rtimes_\phi G_2: g_2\in \Zc G_2, g_1\in \operatorname{Fix}(\phi) \text{ and }\phi(g_2)=c_{g_1}  \},$$ 
	where $\operatorname{Fix}(\phi)=\{g_1\in G_1:\phi(g_2)(g_1)=g_1 \text{ for all } g_2\in G_2\}$.
\end{lem}

\begin{proof}
	Let $(g_1,g_2)\in \Zc(G_1\rtimes_\phi G_2)$. For any $(g'_1,g'_2)\in G_1\rtimes_\phi G_2$ we have that $(g_1,g_2)(g'_1g'_2)=(g_1\phi(g_2)(g'_1),g_2g'_2)$ is equal to $(g'_1,g'_2)(g_1g_2)=(g'_1\phi(g'_2)(g_1),g'_2g_2)$. Hence $g_1\phi(g_2)(g'_1)=g'_1\phi(g'_2)(g_1)$ and $g_2g'_2= g'_2g_2$ for all $g_1'\in G_1$ and $g_2'\in G_2$. 
	
	By the second condition, $g_2\in \Zc G_2$. If we set $g'_1=e$, then $g_1=\phi(g'_2)(g_1)$ for all $g'_2\in G_2$ and therefore $g_1\in \operatorname{Fix}(\phi)$. Finally, by using that $g_1=\phi(g'_2)(g_1)$, we obtain that $g_1\phi(g_2)(g'_1)=g'_1g_1$, which implies that $\phi(g_2)(g'_1)=g_1^{-1}g'_1g_1$ and $\phi(g_2)=c_{g_1}$. 
\end{proof}

We consider a non-trivial group morphism $\phi:\Z\longrightarrow \Gl(n,\Z)$ such that $\phi(1)$ has finite order $a$. The group morphism $\phi(1)$ induces a finite order homeomorphism $f:T^n\longrightarrow T^n$ which we use to construct the mapping torus $T^n_f$. We claim that:
\begin{itemize}
	\item[1.] $T_f^n$ is a compact solvmanifold. This is a consequence of the fact that $T^n_f$ is the total space of a fibration $T^n\longrightarrow T^n_f\longrightarrow S^1$. In particular, $\pi_1(T^n_f)=\Z^n\rtimes_\phi \Z$ is polycyclic.
	\item[2.] $T^n_f$ is finitely covered by the torus $T^{n+1}$. This is because the subgroup $\Z^n\rtimes_\phi a\Z \cong \Z^{n+1}$ is a normal subgroup of index $a$. Therefore, we have a regular covering $p:T^{n+1}\longrightarrow T_f^n$. In particular $T^n_f$ is a flat manifold.
\end{itemize} 

\begin{prop}
	We have $\D(T^n_f)\leq n$.
\end{prop}

\begin{proof}
	Since $T^n_f$ is a flat manifold, we know that $\D(T^n_f)=\rank\Zc(\Z^n\rtimes_\phi \Z)$ by \cite{daura2024actions}. By \cref{center semidriect product}, $\Zc(\Z^n\rtimes_\phi \Z)=\operatorname{Fix}(\phi)\times a\Z$. Since $\phi(1)\neq Id$, we obtain that $\rank\operatorname{Fix}(\phi)\leq n-1$ and therefore $\D(T^n_f)=\rank\Zc(\Z^n\rtimes_\phi \Z)\leq n$.
\end{proof}

\begin{cor}
	The map $p:T^{n+1}\longrightarrow T^n_f$ does not export group actions.
\end{cor}

\subsection{Examples where $\D(M)=\rank\Zc\Gamma$}

Now we present examples of closed orientable manifolds $M$ admitting a non-zero degree map to a nilmanifold $f:M\longrightarrow N/\Gamma$ such that $\D(M)=\D(N/\Gamma)$ but $H^*(M,\Z)\ncong H^*(N/\Gamma,\Z)$.  

Let $H$ be the Heisenberg group of dimension 3,

\[H=\{(x,y,z)=\left(\begin{array}{ccc} 1 & x & z\\  0 & 1 & y \\  0 & 0 & 1 \end{array} \right):x,y,z\in \R\}, \] which is a simply connected nilpotent Lie group. Any lattice of $H$ is isomorphic to a lattice of the form
\[\Gamma_k=\{(x,y,z)=\left(\begin{array}{ccc} 1 & x & \tfrac{1}{k}z\\  0 & 1 & y \\  0 & 0 & 1 \end{array} \right):x,y,z\in \Z\}.\]

Consider $\Gamma_1$ and $\Gamma_2$ lattices of $H$. Note that the covering map $p:H/\Gamma_1\longrightarrow H/\Gamma_2$ induced by the inclusion $\Gamma_1\hookrightarrow \Gamma_2$ is a non-zero degree map between and $\D(H/\Gamma_1)=\D(H/\Gamma_2)=1$. However, a straightforward computation shows that $H_1(H/\Gamma,\Z)\cong \Z^2$ and $ H_1(H/\Gamma_2,\Z)\cong \Z^2\oplus\Z/2$, therefore $H^*(H/\Gamma_1,\Z)\ncong H^*(H/\Gamma,\Z)$. 

Note that $H^*(H/\Gamma_1,\Q)\cong H^*(H/\Gamma_2,\Q)$. Our next goal is to prove the next proposition:

\begin{prop}\label{discsym  not enough}
	There exists a closed orientable manifold $M$ admitting a map to a nilmanifold $f:M\longrightarrow N/\Gamma$ satisfying $\deg(f)=1$, $\D(M)=\D(N/\Gamma)=1$ and $H^*(M,\Q)\ncong H^*(N/\Gamma,\Q)$.
\end{prop}

Firstly, we recall some facts on principal $S^1$-bundles.

\begin{lem}
	Let $p:E\longrightarrow B$ be a principal $S^1$-bundle, then $p^*:H^1(B,\Q)\longrightarrow H^1(E,\Q)$ is injective. It is an isomorphism if and only if the first Chern class $c_1(E)\neq 0$.
\end{lem}

\begin{lem}
	Let $M$ be a closed manifold of dimension $n\geq 4$ and let $D\subseteq M$ be a disk. Then the inclusion $i:M\setminus D\longrightarrow M$ induces an isomorphism $i^*:H^2(M,\Q)\longrightarrow H^2(M\setminus D,\Q)$.
\end{lem}

The first lemma is a consequence of the Gysin exact sequence and the second lemma is a consequence of the Mayer-Vietoris sequence.

Given a principal bundle $p:E\longrightarrow B$, then we have a principal $S^1$-bundle $p':E\setminus p^{-1}(D)\longrightarrow B\setminus D$. Then $c_1(E\setminus p^{-1}(D))=i^*c_1(E)$. 

Note that $\partial(E\setminus p^{-1}(D))\cong S^{n-1}\times S^1$ has a right $S^1$ action induced by the action of the principal $S^1$-bundle. 

Given principal $S^1$-bundles $p_i:E_i\longrightarrow B_i$ with $i=1,2$, we can construct a $S^1$-equivariant homeomorphism $f:\partial(E_1\setminus p_1^{-1}(D_1))\longrightarrow \partial(E_2\setminus p_2^{-1}(D_2))$. There is a principal $S^1$-bundle $p: E_1\setminus p_1^{-1}(D_1)\cup_f E_2\setminus p_2^{-1}(D_2)\longrightarrow B_1\#B_2$. 

\begin{lem}
	If $\dim(B_i)\geq 4$ then $$c_1(E_1\setminus p_1^{-1}(D_1)\cup_f E_2\setminus p_2^{-1}(D_2)\longrightarrow B_1\#B_2)=(c_1(E_1),c_1(E_2))\in H^2(B_1,\Q)\oplus H^2(B_2,\Q)$$.
\end{lem}

Finally note that we have $S^1$-equivariant maps of degree one, $f_i:E\longrightarrow E_i$ for $i=1,2$. We are ready to prove \cref{discsym  not enough}.

\begin{proof}[Proof of \cref{discsym  not enough}]
	Now we consider the Heisenberg manifold of dimension $5$, $H_5/\Gamma$, which is the total space of a principal $S^1$-bundle over a torus $T^4$ (see \cite[Chapter I]{goze2013nilpotent} and \cite{gordon1986spectrum}). We also consider a nilmanifold modeled over the filiform Lie group of dimension 5 (for details on filiform Lie algebras and Lie groups, see \cite[\S2.I]{goze2013nilpotent} and \cite{hamrouni2008discrete,remm2017filiform}), $F_5/\Lambda$, which is the total space of a principal $S^1$-bundle over a filiform nilmanifold of dimension $4$, $F_4/\Lambda'$. By construction, $\D(F_5/\Lambda)=1$ and $H^1(F_4/\Lambda')\cong H^1(F_5/\Lambda)\cong \Q^2$. Note that the discrete degree of symmetry of both nilmanifolds is $1$. Let $E=H_5/\Gamma\setminus p_1^{-1}(D_1)\cup_f F_5/\Lambda\setminus p_2^{-1}(D_2)$ and consider the degree one maps $f_1: E\longrightarrow H_5/\Gamma$ and $f_2:E\longrightarrow F_5/\Lambda$. Observe that $p:E\longrightarrow T^4\# F_4/\Lambda'$ is a principal $S^1$-bundle and therefore $\D(E)\geq 1$. On the other hand, $\D(E)\leq 1$, since the maps $f_i$ export group actions. In conclusion, $\D(E)=1$. However, since $c_1(E)\neq 0$, one can compute that $H^1(E,\Q)\cong H^1(T^4\# F_4/\Lambda',\Q)\cong \Q^5$. On the other hand, we have $H^1(H_5/\Gamma,\Q)\cong H^1(T^4,\Q)\cong \Q^4$ and $H^1(F_5/\Lambda,\Q)\cong H^1(F_4/\Lambda',\Q)\cong \Q^2$, therefore $H^*(E,\Q)\ncong H^*(H_5/\Gamma,\Q)$ and $H^*(E,\Q)\ncong H^*(F_5/\Lambda,\Q)$.
\end{proof}

\subsection{Large finite group actions on admissible manifolds}

One possible generalization of closed connected aspherical manifolds are closed connected admissible manifolds (see \cite[Definition 3.2.7]{lee2010seifert}), which are manifolds where the only periodic self-homeomorphisms of $\tilde{M}$ commuting with the deck transformations group $\pi_1(M)$ are elements of $\Zc\pi_1(M)$. Closed connected oriented manifolds admitting a non-zero degree map to a torus or a nilmanifold are admissible manifolds. One could use the same arguments used to prove \cite[Theorem 1.6]{daura2024actions} together with \cite[Theorem 3.2.2]{lee2010seifert} to obtain:

\begin{prop}\label{finite group actions admissible}
	Let $M$ be a closed connected admissible manifold. Assume that $\Out(\pi_1(M))$ is Minkowski and that $\Zc\pi_1(M)$ is finitely generated. Then:
	\begin{itemize}
		\item[1.] $\Homeo(M)$ is Jordan.
		\item[2.] $\D(M)\leq \rank(\Zc\pi_1(M)/\operatorname{Torsion}(\Zc\pi_1(M)))$.
		\item[3.] If $\chi(M)\neq 0$ and $\Aut(\pi_1(M))$ is Minkowski then $M$ is almost asymmetric.
		\item[4.] If $\Aut(\pi_1(M))$ is Minkowski then $M$ has few stabilizers.
	\end{itemize}
\end{prop}

\begin{proof}
	The proofs of items 1, 2 and 4 are the same as in \cite[Theorem 1.6]{daura2024actions}. If $\Aut(\pi_1(M))$ is Minkowski then $M$ has small stabilizers by \cite[Theorem 3.2.2]{lee2010seifert}. Thus, item 3 follows from \cref{almost fix point euler char} and \cref{small stab+almost fixed point=almost asymmetric}.
\end{proof}

An application of \cref{finite group actions admissible} is the following. Suppose that $M$ is a closed connected oriented aspherical manifold satisfying the hypothesis of \cref{main theorem1 intro} and $M'$ is a closed simply-connected manifold of the same dimension as $M$, then $M\# M'$ is a closed admissible manifold such that $\Homeo(M\#M')$ is Jordan, $\D(M\# M')\leq \rank\Zc\pi_1(M)$ and $M\#M'$ has few stabilizers. The discrete degree of symmetry inequality is usually strict. For example, assume that $M'$ is a closed simply-connected even dimensional manifold such that $\chi(M')\neq 2$. Then $\chi(T^n\#M')=\chi(M')-2\neq 0$ and hence $\D(T^n\#M')=0<n=\rank \Zc\pi_1(T^n\#M')$, by \cref{finite group actions admissible}(3).

\Cref{finite group actions admissible} and \cref{main theorem2 intro} have a non-trivial overlap, but neither is more general than the other. On one hand, not all admissible manifolds admit a non-zero degree maps to a nilmanifold. On the other hand, there exist closed connected orientable manifolds admitting anon-zero degree map to a torus such that the outer automorphism group of their fundamental group is not Minkowski. In \cite[Theorem 1.1.1]{bloomberg1975manifolds} it is proved that if $\Gamma$ and $\Lambda$ are isomorphic neither to a non-trivial free product nor to a infinite cyclic group then $\Out(\Gamma*\Lambda)\cong\Aut(\Gamma)\times\Aut(\Lambda)$ if $\Gamma\ncong\Lambda$ and $\Out(\Gamma*\Lambda)\cong(\Aut(\Gamma)\times\Aut(\Lambda))\times\Z/2$ if $\Gamma\cong\Lambda$. In particular, if $M$ is a closed connected 4-manifold with fundamental group the Baumslag-Solitar group $B(m,ml)=\langle a,b| ba^mb^{-1}=a^{ml}\rangle$ with $m,l\geq 2$. Note that $B(m,ml)$ is not a proper free product of groups (see \cite[II.5.13]{lyndon1977combinatorial} or)  Then $\Out(\pi_1(M\#T^4))\cong\Aut(B(m,ml))\times\Gl(4,\Z)$ is not Minkowski (see \cite[Lemma 3.8]{collins1983automorphisms} or \cite{levitt2007automorphism}), but $M\#T^4$ is hypertoral and hence $\Homeo(M\#T^4)$ is Jordan by \cref{large group actions hypertoral manifolds}.

\section{Free iterated actions}\label{sec: free iterated actions}

Recall that if $G$ is a finite group acting freely on a manifold $M$ then $M/G$ is also a manifold and $p:M\longrightarrow M/G$ is a regular covering. Moreover, we have a short exact sequence $1\longrightarrow \pi_1(M)\longrightarrow \pi_1(M/G)\longrightarrow G\longrightarrow 1$.   Therefore, if we have an iterated free action of $\mathcal{G}=\{G_i\}_{i=1,..,n}$ on a manifold $M$, then every $M_i$ is a manifold, $p_i:M_{i-1}\longrightarrow M_i$ is a regular cover and the map $p:M\longrightarrow M_n$ is a covering on $M$. A free iterated action of $\mathcal{G}=\{G_i\}_{i=1,\dots,n}$ on $M$ induces a series of groups 
$$ \pi_1(M)=\pi_1(M_0)\trianglelefteq\pi_1(M_1)\trianglelefteq\cdots \trianglelefteq \pi_1(M_n) $$
where $\pi_1(M_i)/\pi_1(M_{i-1})\cong G_i$. 

\begin{lem}\label{iterated action equivalent fundamental group}
	Let $\mathcal{G}\acts M$ and $\mathcal{G}'\acts M$ be a free iterated actions of finite groups inducing coverings $p:M\longrightarrow M/\mathcal{G}$ and $p':M\longrightarrow M/\mathcal{G}'$ respectively. If $\mathcal{G}\acts M\sim\mathcal{G}'\acts M$ then there exists an isomorphism $\phi:\pi_1(M/\mathcal{G})\longrightarrow \pi_1(M/\mathcal{G}') $ satisfying $\phi(p_*(\pi_1(M)))=p'_*(\pi_1(M))$. 
\end{lem}

\begin{proof}
	Let $f:M/\mathcal{G}\longrightarrow M/\mathcal{G}'$ be the homeomorphism provided by \cref{iterated action equivalent}. By construction, the coverings $f\circ p:M\longrightarrow M/\mathcal{G}'$ and $p': M\longrightarrow M/\mathcal{G}'$ are isomorphic. Therefore, there exists $\gamma\in \pi_1(M/\mathcal{G}')$ such that $f_*(p_*(\pi_1(M))))=c_\gamma(p'_*(\pi_1(M)))$. We can write $\phi(p_*(\pi_1(M)))=p'_*(\pi_1(M))$, where $\phi=c_{\gamma^{-1}}\circ f_*$.
\end{proof}

\Cref{iterated action equivalent fundamental group} shows that to study free iterated actions up to equivalence, we can focus on the inclusion $p_*:\pi_1(M)\longrightarrow \pi_1(M/\mathcal{G})$ induced by the covering. For example: 

\begin{cor}\label{free iterated action simplifiable regular covering}
	Let $\mathcal{G}\acts M$ and $\mathcal{G}'\acts M$ be equivalent free iterated actions of finite groups inducing coverings $p:M\longrightarrow M/\mathcal{G}$ and $p':M\longrightarrow M/\mathcal{G}'$ respectively. Then $p_*(\pi_1(M))\trianglelefteq\pi_1(M/\mathcal{G})$ if and only if $p'_*(\pi_1(M))\trianglelefteq\pi_1(M/\mathcal{G}')$.
\end{cor}

In particular:

\begin{lem}\label{free itertaed action simplifiable}
	A free iterated action of a collection of finite groups $\mathcal{G}=\{G_i\}_{i=1,..,n}$ on $M$ is simplifiable if and only if $\pi_1(M)\trianglelefteq\pi_1(M/\mathcal{G})$.
\end{lem}

\begin{cor}\label{free itertaed action simplifiable when simply connected}
	A free iterated action on a simply connected manifold $M$ is simplifiable.
\end{cor}

However, $M$ being simply connected is not a necessary condition to have simplifiability of all free iterated actions on $M$.

\begin{lem}\label{free itertaed action simplifiable S1 and T2}
	Any free iterated action on $S^1$ or $T^2$ is simplifiable.
\end{lem}

\begin{proof}
	Let $\mathcal{G}\acts S^1$ be a free iterated group action. Then $S^1/\mathcal{G}$ is a closed 1-dimensional manifold and hence $S^1/\mathcal{G}\cong S^1$. This implies that $\pi_1(M)\cong \Z\trianglelefteq \pi_1(S^1/\mathcal{G})\cong \Z$ and hence the free iterated group action is simplifiable. All the groups of $\mathcal{G}$ are cyclic and the simplification is given by group action of a cyclic group of order $|\mathcal{G}|$.
	
	The proof for the second case is similar. Assume that we have a free iterated group action $\mathcal{G}\acts T^2$. Then $T^2/\mathcal{G}$ is homeomorphic to $T^2$ or the Klein bottle $K$. Then the result follows from the fact that any subgroup of $\pi_1(T^2)\cong\Z^2$ or $\pi_1(K)\cong\Z\rtimes\Z$ isomorphic to $\Z^2$ is normal, hence the action is simplifiable.
\end{proof}

\begin{rem}\label{iterated action T2 not simplfiable}
	\Cref{free itertaed action simplifiable S1 and T2} cannot be extended to $T^n$ for $n\geq 3$. Consider the Bieberbach group with presentation
	$$\Gamma=\langle t_1,t_2,t_3,\alpha|[t_i,t_j]=e\text{ $\forall i,j$},\alpha^3=t_1,\alpha t_2\alpha^{-1}=t_3,\alpha t_3\alpha^{-1}=t_2^{-1}t_3^{-1}\rangle.$$
	
	The group generated by $t_1$, $t_2$ and $t_3$ is normal and isomorphic to $\Z^3$, which we denote by $Z$. Note that $\Gamma/Z\cong \Z/3$. Let $Z'$ be the subgroup generated by $t_1$, $t_2$ and $t_3^2$, which is also isomorphic to $\Z^3$. We have a normal series $Z'\trianglelefteq Z \trianglelefteq \Gamma$ with $\Gamma/ Z\cong \Z/3$ and $Z/Z'\cong \Z/2$. On the other hand, $Z'$ is not normal in $\Gamma$, since $\alpha t_2\alpha^{-1}=t_3\notin Z'$. Now we can define a free iterated action $\{\Z/2,\Z/3\}\acts T^3$ such that $\pi_1(T^3)\cong Z'$, $\pi_1(T^3/\Z/2)\cong Z$ and $\pi_1((T^3/\Z/2)/\Z/3)\cong \Gamma$. This free iterated action is not simplifiable.
\end{rem}

We have seen that free iterated actions produce a covering map that is not necessarily regular. Conversely, given a finite covering map $q:M\longrightarrow M'$ we can ask whether there exists an iterated action $\mathcal{G}$ on $M$ such that $p:M\longrightarrow M/\mathcal{G} $ is isomorphic to $q:M\longrightarrow M'$. The next example shows that this does not happen in general.

\begin{ex}

			Let $M$ be a closed flat manifold with holonomy group the alternate group $A_5$. It exist because every finite group is the holonomy group of a flat manifold (see \cite[Chapter III.5]{charlap2012bieberbach}). Let $n=\dim M$. We take the short exact sequence
			\[\begin{tikzcd}
				1\ar{r}{} &\Z^n\ar{r}{}&\pi_1(M)\ar{r}{\rho}&A_5\ar{r}{}&1
			\end{tikzcd}\]
			Take a non-trivial subgroup $G\leq A_5$ and consider the finite covering of closed flat manifold $q:M'\longrightarrow M$, where  $\pi_1(M')=\rho^{-1}(G)$. If $q$ were induced by a free iterated action, then there would exist a group $\Gamma$ such that $\pi_1(M')\leq \Gamma\trianglelefteq \pi_1(M)$. This would imply that $G\leq \Gamma/\Z^n\trianglelefteq A_5$, which is not possible since $A_5$ is simple.
		\end{ex}
		
		Let $k$ be a natural number and let $\Cov_k(M)$ be the set of all coverings of $M$ of $k$-sheets up to equivalence of coverings. Let $p:\tilde{M}\longrightarrow M$ be a $k$-covering and pick $x\in M$. We can enumerate the points of the fiber $p^{-1}(x)=\{x_1,\dots,x_k\}$. Given $\alpha\in\pi_1(M,x)$ there is a unique lift $a_i:I\longrightarrow \tilde{M}$ such that $a_i(0)=x_i$. Thus, we can define an element in the group of permutations of $k$ letters $\sigma_\alpha$  such that $x_i$ goes to $a_i(1)$. If we remove the choice of the base point $x$, then $\Cov_k(M)\cong \Hom(\pi_1(M),S_k)/\sim$, where $S_k$ is the permutation group of $k$ elements acting by conjugations on $\Hom(\pi_1(M),S_k)$.
		
		Assume that a finite group $G$ acts effectively on $M$, then we have an action of $G$ on $\Cov_k(M)$ given by the pull-back of each element of $G$. Explicitly, $g[\tilde{M}\longrightarrow M]=[g^*\tilde{M}\longrightarrow M]$ for all $g\in G$ and $[\tilde{M}\longrightarrow M]\in\Cov(M)_k $. The induced action of $G$ on  $\Hom(\pi,S_k)/S_k$ is given by $g[f:\pi\longrightarrow S_k]=[f\circ g_*:\pi\longrightarrow S_k]$, where $g_*:\pi_1(M)\longrightarrow \pi_1(M)$ is the group morphism induced by $g$ on the fundamental group.
		
		We also recall the lifting condition of continuous maps. Assume that we have $f:M'\longrightarrow M$ a continuous map such that $f(y)=x$ for some $y\in M'$. Then, there exists $\tilde{f}:M'\longrightarrow \tilde{M}$ such that $f=p\circ\tilde{f}$ if and only if $f_*(\pi_1(M',y))\leq p_*(\pi_1(\tilde{M},\tilde{x}))$. In particular, a homeomorphism $f:M\longrightarrow M$ can be lifted to a homeomorphism  $\tilde{f}:\tilde{M}\longrightarrow\tilde{M}$ if and only if $f_*(p_*(\pi_1(\tilde{M},\tilde{y})))\leq p_*(\pi_1(\tilde{M},\tilde{x}))$.

		\begin{lem}\label{free itertaed actions simplifiable coverings}
			Assume that we have a free iterated action of $\mathcal{G}=\{G_1,G_2\} $ on a manifold $M$. Then, the following statements are equivalent:
			\begin{itemize}
				\item[(1)] $\mathcal{G}$ is simplifiable
				\item[(2)] We can lift ${g_2}:M/G_1\longrightarrow M/G_1$ to a homeomorphism $\tilde{g}_2:M\longrightarrow M$ for all $g_2\in G_2$.
				\item[(3)] The action of $G_2$ on $\Cov_{|G_1|}(M/G_1)$ fixes $[p_1:M\longrightarrow M/G_1]$.
			\end{itemize}
		\end{lem}
		
		\begin{proof}
			We will prove the chain of implications $(1)\implies (2)\implies (3)\implies (1)$. If the action of $\mathcal{G}$ is simplifiable then there exists a group $G$ fitting in the exact sequence $1\longrightarrow G_1\longrightarrow G\longrightarrow G_2\longrightarrow 1$ which acts freely on $M$ and $p:M\longrightarrow M/G=p_2\circ p_1$. Given any $g_2\in G_2$ we choose an element $\tilde{g}_2\in G$ which inside the preimage of $G\longrightarrow G_2$. The induced homeomorphism $\tilde{g}_2:M\longrightarrow M$ is a lift of ${g_2}:M/G_1\longrightarrow M/G_1$.
			
			We now prove the second implication. Note that if we can lift $g_2:M\longrightarrow M$, then $p_1:M\longrightarrow M/G_1$ and the pullback $g_2^*M\longrightarrow M/G_1$ are isomorphic as regular $G_1$-coverings. If $[p_1]\in \Cov_{|G_1|}(M/G_1)$ denotes the class of the covering $p_1:M\longrightarrow M/G_1$ then $g_2[p_1]=[p_1]$ for all $g_2\in G_2$.
			
			Finally, if $G_2$ fixes $[p_1]$ then there exists a group $G$ of the form $1\longrightarrow G_1\longrightarrow G\longrightarrow G_2\longrightarrow 1$ acting freely on $M$. This action extends the action of $G_1$ on $M$ and it also covers the action of $G_2$ on $M/G_1$. Therefore $p:M\longrightarrow M/G=p_2\circ p_1$ and the free iterated action is simplifiable.
		\end{proof}

\section{The iterated discrete degree of symmetry}\label{sec: discsym2}

Recall that if $G$ is a finite group, then $\rank G$ is the minimum number of elements which are needed to generate $G$. We want to extend this notion to iterated group actions. 

\begin{defn}
	Given a free iterated action $\mathcal{G}\acts X$, the rank of the iterated action is $$\rank(\mathcal{G}\acts X)=\min\{\sum_{i=1}^{n}\rank G'_i:\mathcal{G}'=\{G_1',\dots,G_n'\}\acts X\in [\mathcal{G}\acts X]\}.$$
	The iterated rank of the space $X$ is $$\rank(X)=\max\{\rank(\mathcal{G}\acts X): \text{ free iterated action }\mathcal{G}\acts X\}.$$
\end{defn}

\begin{lem}\label{rank iterated aciton simplifiable}
	Assume that we have a free iterated action of $\mathcal{G}=\{G_1,\dots,G_n\}$ on $X$. If $\mathcal{G}\acts X$ is simplifiable with a group $G$, then $\rank(\mathcal{G}\acts X)=\rank G$.
\end{lem}

\begin{proof}
	By definition, $\rank(\mathcal{G}\acts X)\leq\rank G$. Since $\mathcal{G}\acts X$ is simplifiable there exists a subnormal series $G^0=\{e\}\trianglelefteq G^1\trianglelefteq\cdots\trianglelefteq G^n=G$ such that $G^i/G^{i-1}\cong G_{i}$. In particular, $\rank G^{i}\leq \rank G^{i-1}+\rank G_i$ for all $1\leq i\leq n$. This implies that $\rank G=\rank G^n\leq\sum_{i=1}^{n}\rank G_i+\rank G^0=\rank(\mathcal{G}\acts X)$.
\end{proof}

Note that $\rank(\mathcal{G}\acts X)$ depends on the free iterated group action, see \cref{example iterated action torus} below.

We also note that $l(X)\leq \rank (X)$ and therefore we cannot bound the iterated rank of a closed manifold. To generalize the discrete degree of symmetry we will need a notion of rank which only uses abelian groups.

\begin{defn}\label{iterated abelian rank def}
	Let $\mathcal{G}=\{G_1,\dots,G_n\}$ act freely on $X$ and assume that $G_i$ is solvable for all $i$. The abelian rank of the iterated action is $$\rank_{ab}(\mathcal{G}\acts X)=\min\{\sum_{i=1}^{m}\rank A'_i:\{A_1',\dots,A_m'\}\acts X\in [\mathcal{G}\acts X], \text{ $A_i'$ abelian for all $i$}\}.$$
\end{defn} 

Note that $\rank(\mathcal{G}\acts X)\leq \rank_{ab}(\mathcal{G}\acts X)$. We would like to define an invariant of free iterated actions with similar properties to the discrete degree of symmetry. An essential property of the discrete degree of symmetry is that if $M$ is a closed manifold, then $\D(M)<\infty$. The proof of this fact is a direct consequence of \cref{MannSu thm}. We can generalize this theorem to the context of free iterated actions. Recall that given a closed manifold $M$ of dimension $n$, we define $b(M)=\sum_{i=1}^{n} \rank H_i(M,\Z)$, where $\rank$ is understood to be the minimum number of generators needed to generate $H_i(M,\Z)$ (note that $b(M)$ is not the Betti number of $H_*(M,\Z)$, since we also count the torsion part of $H_*(M,\Z)$). If $p$ is a prime number, then $b_p(M)=\sum_{i=1}^{n} \dim H_i(M,\Z/p)$. The next results generalizes \cref{MannSu thm} to the context of free iterated actions.

\begin{thm}\label{iterated MannSu thm}
	Let $M$ be a closed connected $n$-dimensional manifold. There exists a sequence of numbers $\{f_i\}_{i\in\mathbb{N}}$ depending only on $n$ and $b(M)$ such that for any prime $p$, any sequence of positive integers $\{k_i\}$ and any free iterated action $\{(\Z/p^{k_i})^{a_i}\}_{i=1,\dots,r}\acts M$, the numbers $a_i$ satisfy $a_i\leq f_i$ for all $i$.
\end{thm}

To prove \cref{iterated MannSu thm}, we need the following lemma.

\begin{lem}\label{cohomology bound Borel fibration}
	Let $M$ be a closed connected manifold and $p$ a prime number. Assume that we have an effective action of a finite $p$-group $G$ on $M$. Then $$\dim H_k(M/G,\Z/p)\leq\sum_{i+j=k}\dim H_i(BG,\Z/p)\dim H_j(M,\Z/p).$$
\end{lem} 

We refer to \cite[Lemma 1.46]{saez2019finite} for a detailed proof of the lemma.

\begin{proof}[Proof of \cref{iterated MannSu thm}]
	We construct the sequence $f_i$ recursively. For the first step, if $(\Z/p^{k_1})^{a_1}$ acts freely on $M$ then so it does the group $(\Z/p)^{a_1}$. Thus, we can take $f_1=f(n,b_p(M))$, where $f$ is the function defined in \cref{MannSu thm}. 
	
	By \cref{cohomology bound Borel fibration}, if $\Z/p^{k_1}$ acts freely on $M$, then 
	$$\dim H_k(M/(\Z/p^{k_1}),\Z/p)\leq \sum_{i+j=k} \dim H_i(B\Z/p^{k_1},\Z/p)\dim H_j(M,\Z/p)$$ for any $k$. Moreover, $H_i(B\Z/p^{k_1},\Z/p)\cong\Z/p$ for all $i\geq 0$ (see \cite[Chapter III, \S1]{brown2012cohomology}). Thus,
	$\dim H_k(M/(\Z/p^{k_1}),\Z/p)\leq b_p(M)$ for all $k$ and therefore $b_p(M/(\Z/p^{k_1}))\leq nb_p(M)$. Using this inequality recursively, we obtain that $b_p(M/(\Z/p^{k_1})^{a_1})\leq n^{a_1}b_p(M)\leq n^{f_1}b_p(M)$. 
	
	We can use \cref{MannSu thm} on $M/(\Z/p^{k_1})^{a_1}$ to deduce that $$a_2\leq f(n,b_p(M/(\Z/p^{k_1})^{a_1}))\leq f(n,n^{f_1}b_p(M)).$$ Set $f_2=f(n,n^{f_1}b_p(M))$. Repeating the same argument as above, we obtain $$f_i=f(n,n^{f_1+\dots+f_{i-1}}b_p(M))$$ for all $1\leq i\leq r$. 
	
	By the universal coefficients theorem, $b_p(M)\leq 2b(M)$. Thus, by replacing  $b_p(M)$ with $2b(M)$ we obtain a bound not depending on the prime $p$.
\end{proof}

We are ready to define an iterated discrete degree of symmetry for free iterated actions of length $2$. In $\mathbb{N}^2$ we have a partial order relation called the lexicographic order defined as follows: If $(n,m),(n',m')\in \mathbb{N}^2$ then $(n,m)\geq (n',m')$ if and only if $n>n'$ or $n=n'$ and $m\geq m'$. 

\begin{defn}
	We define $\mu_2(M)$ as the set of all pairs $(f,b)\in \mathbb{N}^2$ which satisfy:
	\begin{itemize}
		\item[1.] There exist an increasing sequence of prime numbers $\{p_i\}$, a sequence of natural numbers $\{a_i\}$ and a collection of free iterated actions $\{(\Z/p_i^{a_i})^f,(\Z/p_i)^b\}\acts M$ for each $i\in\mathbb{N}$.
		\item[2.] $\rank_{ab}(\{(\Z/p_i^{a_i})^f,(\Z/p_i)^b\}\acts M)=f+b$ for each $i\in\mathbb{N}$.
	\end{itemize}
	
	We define the iterated discrete degree of symmetry of $M$ as
	$$\D_2(M)=\max\{(0,0)\cup\mu_2(M)\}.$$
\end{defn}

To understand better the definition of the iterated discrete degree of symmetry, let us compute it for tori.

\begin{lem}\label{iterated discsym Tn}
	We have $\D_2(T^n)=(n,0)$.
\end{lem}

\begin{proof}
	Let $C$ be the Minkowski constant of $\Gl(n,\Z)$. Assume that $\D_2(T^n)=(d_1,d_2)$. Since $T^n$ admits actions of $(\Z/p)^n$ for any prime $p$ and $\D(T^n)=n$, we have $d_1=n$. By hypothesis there exists an increasing sequence of prime numbers $\{p_i\}_{i\in\mathbb{N}}$ with $p_i>C$ and free iterated group actions of $\{(\Z/p_i^{a_i})^n,(\Z/p_i)^{d_2}\}\acts T^n$ such that $\rank_{ab}(\{(\Z/p_i^{a_i})^n,(\Z/p_i)^{d_2}\}\acts T^n)=n+d_2$.
	
	The action of $(\Z/p_i^{a_i})^n$ on $T^n$ is by rotations for all $i$, therefore $T^n/(\Z/p_i^{a_i})^n\cong T^n$.  Since $p_i>C$ the group $(\Z/p_i)^{d_2}$ also acts by rotations on $T^n$, hence $T^n/(\{(\Z/p_i^{a_i})^n,(\Z/p_i)^{d_2}\})$ is homeomorphic to $T^n$. In consequence, the free iterated action  $\{(\Z/p_i^{a_i})^n,(\Z/p_i)^{d_2}\}\acts T^n$ is simplifiable for all $i$. The simplification gives a group $G_i$ and a short exact sequence
	$$1\longrightarrow  (\Z/p_i^{a_i})^n\longrightarrow G_i\longrightarrow (\Z/p_i)^{d_2}\longrightarrow 1.$$
	
	Moreover, we also have that a short exact sequence 
	$$1\longrightarrow \Z^n\longrightarrow \Z^n\longrightarrow G_i\longrightarrow 1 $$
	induced by the regular covering $p:T^n\longrightarrow T^n/\{(\Z/p_i^{a_i})^n,(\Z/p_i)^{d_2}\}$. This implies that $G_i$ is abelian and $\rank G_i\leq n$. On the other hand, $$\rank G_i=\rank_{ab}(\{(\Z/p_i^{a_i})^n,(\Z/p_i)^{d_2}\}\acts T^n)=n+d_2\leq n,$$ which implies that $d_2=0$.
\end{proof}

\begin{rem}\label{example iterated action torus}
	As an example, we consider two different free iterated actions of $\{\Z/p,\Z/p\}\acts T^2$, as shown in the Figure 4.1.
	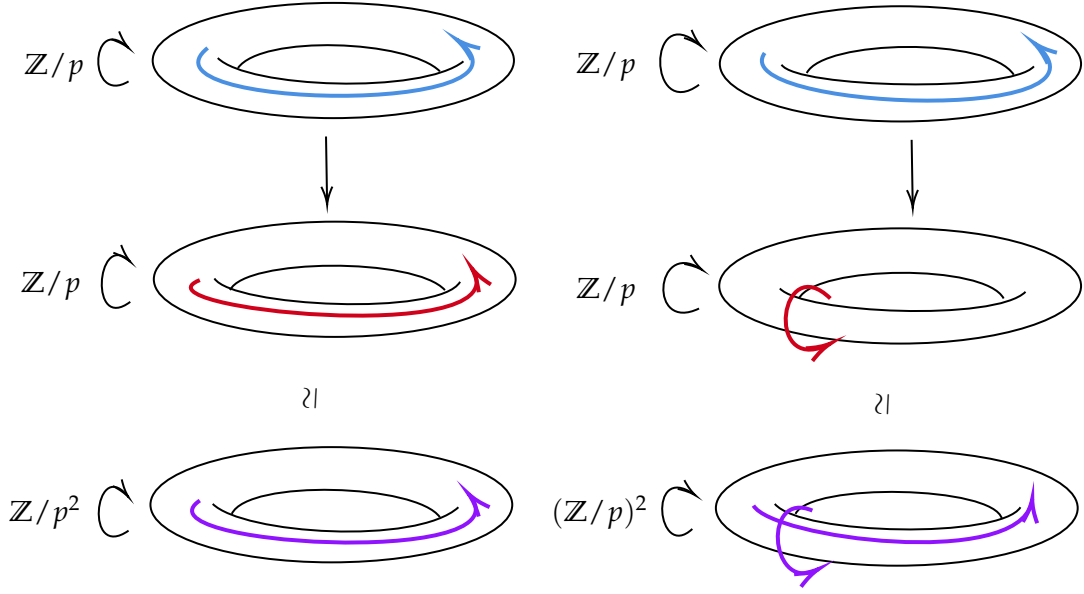
\begin{figure}[h]
		\centering

		\tikzset{every picture/.style={line width=0.75pt}} 
		
		\begin{tikzpicture}[x=0.63pt,y=0.63pt,yscale=-1,xscale=1]
			
			\draw   (92,85.5) .. controls (92,66.45) and (140.05,51) .. (199.33,51) .. controls (258.61,51) and (306.67,66.45) .. (306.67,85.5) .. controls (306.67,104.55) and (258.61,120) .. (199.33,120) .. controls (140.05,120) and (92,104.55) .. (92,85.5) -- cycle ;
			\draw   (92.67,216.5) .. controls (92.67,197.45) and (140.8,182) .. (200.17,182) .. controls (259.54,182) and (307.67,197.45) .. (307.67,216.5) .. controls (307.67,235.55) and (259.54,251) .. (200.17,251) .. controls (140.8,251) and (92.67,235.55) .. (92.67,216.5) -- cycle ;
			\draw   (90.67,351.5) .. controls (90.67,332.45) and (138.8,317) .. (198.17,317) .. controls (257.54,317) and (305.67,332.45) .. (305.67,351.5) .. controls (305.67,370.55) and (257.54,386) .. (198.17,386) .. controls (138.8,386) and (90.67,370.55) .. (90.67,351.5) -- cycle ;
			\draw    (77.67,98) .. controls (53.17,112.7) and (54.6,58.25) .. (76.31,76.77) ;
			\draw [shift={(77.67,78)}, rotate = 223.73] [color={rgb, 255:red, 0; green, 0; blue, 0 }  ][line width=0.75]    (10.93,-3.29) .. controls (6.95,-1.4) and (3.31,-0.3) .. (0,0) .. controls (3.31,0.3) and (6.95,1.4) .. (10.93,3.29)   ;
			\draw    (78.67,230) .. controls (54.29,248.53) and (58.44,189.1) .. (77.19,204.65) ;
			\draw [shift={(78.67,206)}, rotate = 225] [color={rgb, 255:red, 0; green, 0; blue, 0 }  ][line width=0.75]    (10.93,-3.29) .. controls (6.95,-1.4) and (3.31,-0.3) .. (0,0) .. controls (3.31,0.3) and (6.95,1.4) .. (10.93,3.29)   ;
			\draw    (77.67,367) .. controls (51.21,385.62) and (57.4,323.57) .. (76.48,345.55) ;
			\draw [shift={(77.67,347)}, rotate = 232.43] [color={rgb, 255:red, 0; green, 0; blue, 0 }  ][line width=0.75]    (10.93,-3.29) .. controls (6.95,-1.4) and (3.31,-0.3) .. (0,0) .. controls (3.31,0.3) and (6.95,1.4) .. (10.93,3.29)   ;
			\draw    (195,131) -- (195.63,170) ;
			\draw [shift={(195.67,172)}, rotate = 269.07] [color={rgb, 255:red, 0; green, 0; blue, 0 }  ][line width=0.75]    (10.93,-3.29) .. controls (6.95,-1.4) and (3.31,-0.3) .. (0,0) .. controls (3.31,0.3) and (6.95,1.4) .. (10.93,3.29)   ;
			\draw   (429,87.5) .. controls (429,68.45) and (477.05,53) .. (536.33,53) .. controls (595.61,53) and (643.67,68.45) .. (643.67,87.5) .. controls (643.67,106.55) and (595.61,122) .. (536.33,122) .. controls (477.05,122) and (429,106.55) .. (429,87.5) -- cycle ;
			\draw   (428.67,220.5) .. controls (428.67,201.45) and (476.8,186) .. (536.17,186) .. controls (595.54,186) and (643.67,201.45) .. (643.67,220.5) .. controls (643.67,239.55) and (595.54,255) .. (536.17,255) .. controls (476.8,255) and (428.67,239.55) .. (428.67,220.5) -- cycle ;
			\draw   (426.67,353.5) .. controls (426.67,334.45) and (474.8,319) .. (534.17,319) .. controls (593.54,319) and (641.67,334.45) .. (641.67,353.5) .. controls (641.67,372.55) and (593.54,388) .. (534.17,388) .. controls (474.8,388) and (426.67,372.55) .. (426.67,353.5) -- cycle ;
			\draw    (543,133) -- (543.63,172) ;
			\draw [shift={(543.67,174)}, rotate = 269.07] [color={rgb, 255:red, 0; green, 0; blue, 0 }  ][line width=0.75]    (10.93,-3.29) .. controls (6.95,-1.4) and (3.31,-0.3) .. (0,0) .. controls (3.31,0.3) and (6.95,1.4) .. (10.93,3.29)   ;
			\draw    (416.67,100) .. controls (386.13,122.66) and (385.67,49.26) .. (416.57,75.72) ;
			\draw [shift={(418,77)}, rotate = 222.86] [color={rgb, 255:red, 0; green, 0; blue, 0 }  ][line width=0.75]    (10.93,-3.29) .. controls (6.95,-1.4) and (3.31,-0.3) .. (0,0) .. controls (3.31,0.3) and (6.95,1.4) .. (10.93,3.29)   ;
			\draw    (416.67,234) .. controls (386.29,248.7) and (390.48,193.29) .. (417.01,209.89) ;
			\draw [shift={(418.67,211)}, rotate = 215.54] [color={rgb, 255:red, 0; green, 0; blue, 0 }  ][line width=0.75]    (10.93,-3.29) .. controls (6.95,-1.4) and (3.31,-0.3) .. (0,0) .. controls (3.31,0.3) and (6.95,1.4) .. (10.93,3.29)   ;
			\draw    (412.67,366) .. controls (391.65,385.7) and (385.84,325.84) .. (413.38,345.07) ;
			\draw [shift={(414.67,346)}, rotate = 217.18] [color={rgb, 255:red, 0; green, 0; blue, 0 }  ][line width=0.75]    (10.93,-3.29) .. controls (6.95,-1.4) and (3.31,-0.3) .. (0,0) .. controls (3.31,0.3) and (6.95,1.4) .. (10.93,3.29)   ;
			\draw    (130.67,84) .. controls (148.67,102) and (251.67,106) .. (276.67,86) ;
			\draw    (465.67,88) .. controls (474.67,102) and (601.67,105) .. (615.67,89) ;
			\draw    (128.67,216) .. controls (139.67,233) and (262.67,239) .. (274.67,218) ;
			\draw    (464.67,222) .. controls (478.67,237) and (592.67,242) .. (610.67,224) ;
			\draw    (127.67,351) .. controls (145.67,369) and (248.67,376) .. (273.67,353) ;
			\draw    (461.67,351) .. controls (477.67,370) and (595.67,372) .. (607.67,353) ;
			\draw    (142.67,90) .. controls (156.67,70) and (243.67,73) .. (262.67,92) ;
			\draw    (480.67,93) .. controls (490.67,73) and (588.67,71) .. (603.67,94) ;
			\draw    (138,223) .. controls (146.67,204) and (247.67,203) .. (265.67,223) ;
			\draw    (476,228) .. controls (486.67,208) and (586.67,207) .. (597.67,228) ;
			\draw    (139,357) .. controls (148.67,336) and (248.67,339) .. (262.67,359) ;
			\draw    (474,357) .. controls (482.67,338) and (591.67,340) .. (596.67,360) ;
			\draw [color={rgb, 255:red, 74; green, 144; blue, 226 }  ,draw opacity=1 ][line width=1.5]    (124.67,78) .. controls (79.13,114.63) and (328.6,118.92) .. (274.4,74.37) ;
			\draw [shift={(272.67,73)}, rotate = 37.02] [color={rgb, 255:red, 74; green, 144; blue, 226 }  ,draw opacity=1 ][line width=1.5]    (14.21,-4.28) .. controls (9.04,-1.82) and (4.3,-0.39) .. (0,0) .. controls (4.3,0.39) and (9.04,1.82) .. (14.21,4.28)   ;
			\draw [color={rgb, 255:red, 208; green, 2; blue, 27 }  ,draw opacity=1 ][line width=1.5]    (120,217) .. controls (77.32,233.75) and (305.65,257.28) .. (282.88,212.11) ;
			\draw [shift={(281.67,210)}, rotate = 57.14] [color={rgb, 255:red, 208; green, 2; blue, 27 }  ,draw opacity=1 ][line width=1.5]    (14.21,-4.28) .. controls (9.04,-1.82) and (4.3,-0.39) .. (0,0) .. controls (4.3,0.39) and (9.04,1.82) .. (14.21,4.28)   ;
			\draw [color={rgb, 255:red, 144; green, 19; blue, 254 }  ,draw opacity=1 ][line width=1.5]    (120,349) .. controls (80.27,375.6) and (311.56,389.58) .. (281.24,346.99) ;
			\draw [shift={(279.67,345)}, rotate = 49.09] [color={rgb, 255:red, 144; green, 19; blue, 254 }  ,draw opacity=1 ][line width=1.5]    (14.21,-4.28) .. controls (9.04,-1.82) and (4.3,-0.39) .. (0,0) .. controls (4.3,0.39) and (9.04,1.82) .. (14.21,4.28)   ;
			\draw [color={rgb, 255:red, 74; green, 144; blue, 226 }  ,draw opacity=1 ][line width=1.5]    (455.67,81) .. controls (431.91,114.66) and (670.81,124.8) .. (617.38,76.48) ;
			\draw [shift={(615.67,75)}, rotate = 39.81] [color={rgb, 255:red, 74; green, 144; blue, 226 }  ,draw opacity=1 ][line width=1.5]    (14.21,-4.28) .. controls (9.04,-1.82) and (4.3,-0.39) .. (0,0) .. controls (4.3,0.39) and (9.04,1.82) .. (14.21,4.28)   ;
			\draw [color={rgb, 255:red, 208; green, 2; blue, 27 }  ,draw opacity=1 ][line width=1.5]    (494.67,228) .. controls (464.94,198.6) and (455.06,272.92) .. (492.33,256.14) ;
			\draw [shift={(494.67,255)}, rotate = 152.3] [color={rgb, 255:red, 208; green, 2; blue, 27 }  ,draw opacity=1 ][line width=1.5]    (14.21,-4.28) .. controls (9.04,-1.82) and (4.3,-0.39) .. (0,0) .. controls (4.3,0.39) and (9.04,1.82) .. (14.21,4.28)   ;
			\draw [color={rgb, 255:red, 144; green, 19; blue, 254 }  ,draw opacity=1 ][line width=1.5]    (449,352) .. controls (460.43,371.6) and (604.72,390.24) .. (614.25,350.51) ;
			\draw [shift={(614.67,348)}, rotate = 95.31] [color={rgb, 255:red, 144; green, 19; blue, 254 }  ,draw opacity=1 ][line width=1.5]    (14.21,-4.28) .. controls (9.04,-1.82) and (4.3,-0.39) .. (0,0) .. controls (4.3,0.39) and (9.04,1.82) .. (14.21,4.28)   ;
			\draw [color={rgb, 255:red, 144; green, 19; blue, 254 }  ,draw opacity=1 ][line width=1.5]    (484,355) .. controls (454.58,342.39) and (458.4,402.23) .. (485.12,391.23) ;
			\draw [shift={(487.67,390)}, rotate = 151.11] [color={rgb, 255:red, 144; green, 19; blue, 254 }  ,draw opacity=1 ][line width=1.5]    (14.21,-4.28) .. controls (9.04,-1.82) and (4.3,-0.39) .. (0,0) .. controls (4.3,0.39) and (9.04,1.82) .. (14.21,4.28)   ;
			
			\draw (14,79.4) node [anchor=north west][inner sep=0.75pt]    {$\mathbb{Z} /p$};
			\draw (12,209.4) node [anchor=north west][inner sep=0.75pt]    {$\mathbb{Z} /p$};
			\draw (5,344.4) node [anchor=north west][inner sep=0.75pt]    {$\mathbb{Z} /p^{2}$};
			\draw (179.28,296.22) node [anchor=north west][inner sep=0.75pt]  [rotate=-273.09,xslant=-0.02]  {$\simeq $};
			\draw (342,78.4) node [anchor=north west][inner sep=0.75pt]    {$\mathbb{Z} /p$};
			\draw (342,213.4) node [anchor=north west][inner sep=0.75pt]    {$\mathbb{Z} /p$};
			\draw (327,344.4) node [anchor=north west][inner sep=0.75pt]    {$(\mathbb{Z} /p)^{2}$};
			\draw (519.4,300.15) node [anchor=north west][inner sep=0.75pt]  [rotate=-272.52]  {$\simeq $};
		\end{tikzpicture}
		\caption{Free iterated actions on $T^2$}
	\end{figure}
	
	Let us describe explicitly these two free iterated actions. The first step of both free iterated actions is by rotations on the horizontal plane. More explicitly, given $g\in\Z/p$ and $[x,y]\in T^2=\R^2/\Z^2$, we have $g[x,y]=[x+\tfrac{g}{p},y]$. The quotient is $T^2/(\Z/p)\cong T^2$. The second step of the first iterated action is by rotations on the horizontal plane. The free iterated action is simplifiable by an action of $\Z/p^2$ on $T^2$ such that $g[x,y]=[x+\tfrac{g}{p^2},y]$ for $g\in\Z/p^2$ and $[x,y]\in T^2$. Note that $\rank_{ab}(\{\Z/p,\Z/p\}\acts T^2)=1<2$. Hence, this free iterated action is not considered in the definition of the iterated discrete degree of symmetry.
	
	The second step of the second iterated actions is by rotations such that $g[x,y]=[x,y+\tfrac{g}{p}]$ for $g\in\Z/p$ and $[x,y]\in T^2$. The free iterated action is simplifiable by an action of $(\Z/p)^2$ on $T^2$ such that $(g,h)[x,y]=[x+\tfrac{g}{p},y+\tfrac{h}{p}]$ for $(g,h)\in(\Z/p)^2$ and $[x,y]\in T^2$. In this case $\rank_{ab}(\{\Z/p,\Z/p\}\acts T^2)=2$, but it does not satisfy the maximality condition on the definition of the iterated discrete degree of symmetry, since we have a free iterated action $\{(\Z/p)^2,\{e\}\}\acts T^2$. 
\end{rem}

\begin{cor}
	Let $M$ be a closed manifold. There exists $(f,b)\in\mathbb{N}^2$ such that $\D_2(M)\leq (f,b)$.
\end{cor}

\begin{proof}
	We can choose $f=f_1$ and $b=f_2$ from \cref{iterated MannSu thm}.
\end{proof}

Note that if $\D_2(M)=(d_1,d_2)$ then $d_1\leq\D(M)$. We want to investigate the value $d_2$. 

\begin{rem}
	If $M$ is a manifold such that $\D(M)=0$ then $\D_2(M)=(0,0)$.
\end{rem}

The next lemma shows a case where $d_2=0$, so we do not have large iterated group actions.

\begin{lem}\label{iterated discsym Jordan}
	Let $M$ be a closed manifold. Assume that $l(M)=1$ and $\Homeo(M)$ is Jordan. Then $\D_2(M)=(d,0)$.
\end{lem}

\begin{proof}
	Assume that $C$ is the Jordan constant of $\Homeo(M)$ and that $\D_2(M)=(d_1,d_2)$. We have an increasing sequence of prime numbers $\{p_i\}_{i\in\mathbb{N}}$ and free iterated group actions of $\{(\Z/p_i^{a_i})^{d_1},(\Z/p_i)^{d_2}\}$ on $M$ such that $\rank_{ab}(\{(\Z/p_i^{a_i})^{d_1},(\Z/p_i)^{d_2}\}\acts M)=d_1+d_2$.
	We may assume without loss of generality that $p_i>C$. Since all actions are simplifiable, for each $i$ there exists a $p_i$-group $G_i$ acting freely on $M$ which fits in the short exact sequence
	$$1\longrightarrow  (\Z/p_i^{a_i})^{d_1}\longrightarrow G_i\longrightarrow (\Z/p_i)^{d_2}\longrightarrow 1.$$
	
	Any proper subgroup $H$ of $G_i$ has index $[G:H]>p_i>C$. Since $\Homeo(M)$ is Jordan of constant $C$, we can conclude that $G_i$ is abelian. Consequently, we have $\rank G_i=\rank_{ab}(\{(\Z/p_i^{a_i})^{d_1},(\Z/p_i)^{d_2}\}\acts M)=d_1+d_2$. Therefore, $G_i\cong (\Z/p_i^{a_i})^{d_1}\oplus (\Z/p_i)^{d_2}$. We can take a subgroup $(\Z/p_i)^{d_1+d_2}\leq G_i$, which acts freely on $M$ for all $i$. This implies that $(d_1+d_2,0)\leq (d_1,d_2)$ and since $d_1,d_2\geq 0$, $d_2=0$.
\end{proof}

\begin{prop}
	Let $p:M'\longrightarrow M$ be a regular finite covering. Then $\D_2(M')\geq \D_2(M)$.
\end{prop}

\begin{proof}
	Firstly, we need to prove the following group theoretic fact. Let $G$ be a finite group and $p$ a prime such that $p>|G|!$. If $G'$ is a group which fits in the short exact sequence $1\longrightarrow G\longrightarrow G'\longrightarrow (\Z/p^r)^b\longrightarrow 1$ then $G'\cong G\times (\Z/p^r)^b$.
	
	The abstract kernel $\psi:(\Z/p^r)^b\longrightarrow \Out(G)$ is trivial since $|\Out(G)|\leq |G|!<p$. This leads to the following diagram
	\[
	\begin{tikzcd}
		& 1\ar{d}{}& 1\ar{d}{}& & \\
		1\ar{r}{} & \Zc G \ar{r}{}\ar{d}{} & H\ar{r}{}\ar{d}{}& (\Z/p^r)^b\ar{r}{}\ar{d}{Id} &1\\
		1\ar{r}{} & G \ar{r}{}\ar{d}{} & G'\ar{r}{}\ar{d}{}& (\Z/p^r)^b\ar{r}{} &1\\
		& \Inn G\ar{d}{}\ar{r}{Id} & \Inn G\ar{d}{} & &\\
		& 1 &1 & & \\
	\end{tikzcd}
	\]
	where the above short exact sequence is central. Since $(\Z/p^r)^b$ is abelian, there is a cocycle $\beta: (\Z/p^r)^b\times (\Z/p^r)^b\longrightarrow \Zc G$ which is a group morphism. More explicitly, $\beta(x,y)=[\tilde{x},\tilde{y}]$, where $\tilde{x},\tilde{y}\in H$ are preimages of $x$ and $y$ respectively. We use that $|\Zc G|<p$ to conclude that $\beta$ is trivial. This implies that the central short exact sequence is trivial and $H\cong \Zc G\times(\Z/p^r)^b$.
	
	We can construct a group morphism $(\Z/p^r)^b\longrightarrow H\longrightarrow G$ which makes the short exact sequence split. Therefore $G'\cong G\rtimes_\phi (\Z/p^r)^b$. The group morphism $\phi:(\Z/p^r)^b\longrightarrow \Aut(G)$ is trivial since $|\Aut(G)|\leq |G|!<p$. In conclusion, $G'\cong G\times (\Z/p^r)^b$.
	
	Assume now that we have a regular $G$-covering  $p:M'\longrightarrow M$ and $\D_2(M)=(d_1,d_2)$. There is an increasing sequence of prime numbers $\{p_i\}_{i\in\mathbb{N}}$ such that $p_i>|G|!$ for all $i$ and free iterated group actions of $\{(\Z/p_i^{a_i})^{d_1},(\Z/p_i)^{d_2}\}\acts M$. They induce free iterated actions of $\{G,(\Z/p_i^{a_i})^{d_1},(\Z/p_i)^{d_2}\}\acts M'$.
	
	Since $p_i>|G|!$, the actions of $(\Z/p_i^{a_i})^{d_1}$ on $\Cov_{|G|}(M)$ are trivial. By \cref{free itertaed actions simplifiable coverings} the free iterated action $\{G,(\Z/p_i^{a_i})^{d_1},(\Z/p_i)^{d_2}\}\acts M'$ is equivalent to a free iterated action of $\{G_i',(\Z/p_i)^{d_2}\}\acts M'$ for all $i$. For each $i$ there exists a short exact sequence $$1\longrightarrow G\longrightarrow G'_i\longrightarrow (\Z/p_i^{a_i})^{d_1}\longrightarrow 1.$$ Since $p_i>|G|!$, $G_i'\cong G\times (\Z/p_i^{a_i})^{d_1}$. Thus, the free iterated action $\{G_i',(\Z/p_i)^{d_2}\}\acts M'$ is equivalent to $\{(\Z/p_i^{a_i})^{d_1},G,(\Z/p_i)^{d_2}\}\acts M'$ for all $i$. Repeating the same argument we can show that $\{(\Z/p_i^{a_i})^{d_1},G,(\Z/p_i)^{d_2}\}\acts M'$ is equivalent to $\{(\Z/p_i^{a_i})^{d_1},(\Z/p_i)^{d_2},G\}\acts M'$ for all $i$.
	
	We have free iterated actions $\{(\Z/p_i^{a_i})^{d_1},(\Z/p_i)^{d_2}\}\acts M'$ for all $i$, thus $\D_2(M)=(d_1,d_2)\leq \D_2(M')$. 
\end{proof}   

We start now the study of the iterated discrete degree of symmetry on nilmanifolds. We need 3 preliminary lemmas.

\begin{lem}\cite[Lemma 3.1.14]{lee2010seifert}\label{aspherical manifolds big diagram}
	Let $G$ be a finite group acting effectively on a closed manifold $M$. Then, there is a commutative diagram with exact rows and columns
	\[\begin{tikzcd}
		& 1\ar{d}{} & 1 \ar{d}{}& 1 \ar{d}{} & \\
		1\ar{r}{}& \Zc\pi_1(M)\ar{r}{}\ar{d}{} & C_{\tilde{G}}(\pi_1(M)) \ar{r}{}\ar{d}{}& \ker\psi\ar{r}{}\ar{d}{} & 1\\
		1\ar{r}{}& \pi_1(M)\ar{r}{}\ar{d}{} & \tilde{G} \ar{r}{}\ar{d}{}& G\ar{r}{}\ar{d}{\psi} & 1\\
		1\ar{r}{}& \Inn(\pi_1(M))\ar{r}{}\ar{d}{} & \Aut(\pi_1(M)) \ar{r}{}& \Out(\pi_1(M))\ar{r}{} & 1\\
		&1 &&&
	\end{tikzcd}\]
	where $C_{\tilde{G}}(\pi_1(M))$ is the centralizer of $\pi_1(M)$ in $\tilde{G}$.
\end{lem}

\begin{lem}\label{inner actions aspherical manifolds}\cite[Proposition 3.1.21]{lee2010seifert}
	Let $M$ be a closed aspherical manifold and let $\Zc\pi_1(M)$ be finitely generated. Assume that we have a free action of an abelian group $A$ on $M$ such that $\psi:A\longrightarrow \Out(\pi_1(M))$ is trivial. Then $\Zc\pi_1(M/A)=C_{\pi_1(M/A)}(\pi_1(M))$ and it is an extension of $\Zc\pi_1(M)$ by $A$. In particular, $\rank(\Zc\pi_1(M/A))=\rank(\Zc\pi_1(M))$.
\end{lem}

\begin{lem}\label{inner actions simplifiable}
	Let $\{A,A'\}\acts M$ be a free iterated action of abelian groups on a closed connected aspherical manifold such that $\psi:A\longrightarrow \Out(\pi_1(M))$ and $\psi':A'\longrightarrow\Out(\pi_1(M/A))$ are trivial. Then $\{A,A'\}\acts M$ is simplifiable by an abelian group.
\end{lem}

\begin{proof}
	We have a commutative diagram 
	\[
	\begin{tikzcd}
		1\ar{d}{} & 1\ar{d}{} & 1\ar{d}{} \\
		\Zc \pi_1(M) \ar{r}{}\ar{d}{} &\Zc \pi_1(M/A)\ar{r}{}\ar{d}{}& \Zc \pi_1((M/A)/A')\ar{d}{}\\
		\pi_1(M) \ar{r}{}\ar{d}{p} & \pi_1(M/A)\ar{r}{}\ar{d}{q}& \pi_1((M/A)/A')\ar{d}{q'} \\
		\Inn \pi_1(M)\ar{r}{Id}\ar{d}{} & \Inn \pi_1(M)\ar{r}{Id}\ar{d}{} &\Inn \pi_1(M)\ar{d}{}\\
		1& 1 & 1
	\end{tikzcd}
	\]
	Given $\gamma\in\pi_1(M)$ and $g'\in \pi_1((M/A)/A')$, we want to see that $g'\gamma g'^{-1}\in\pi_1(M)$. Let $\gamma'\in \pi_1(M)$ such that $p(\gamma')=q'(g')$, then $g'\gamma g'^{-1}=(g'\gamma'^{-1})\gamma'' (g'\gamma'^{-1})^{-1}$ with $\gamma''\in\pi_1(M)$. Since $q'(g'\gamma'^{-1})$ is trivial then $g'\gamma'^{-1}\in\Zc \pi_1((M/A)/A')$ and therefore $g'\gamma g'^{-1}=\gamma''\in\pi_1(M)$, as we wanted to see. 
	
	Note that $\pi_1((M/A)/A')/\pi_1(M)\cong \Zc \pi_1((M/A)/A')/\Zc\pi_1(M)$, which is abelian of rank at most $\rank\Zc\pi_1(M)$.
\end{proof}

The next result shows that the definition of the iterated degree of symmetry is suitable to study 2-step nilmanifolds.

\begin{thm}\label{iterated discsym nilmanifolds}
	Let $N/\Gamma$ be a nilmanifold of dimension $n$, $f=\rank\Zc\Gamma$ and $\D_2(N/\Gamma)=(d_1,d_2)$. Then $d_1=f$ and $d_2\leq n-f$. If $d_2= n-f$ then $N/\Gamma$ is a 2-step nilmanifold. Conversely, if $N/\Gamma$ is a 2-step nilmanifold then $\D_2(N/\Gamma)=(f,n-f)$.
\end{thm}

\begin{proof}
	Let $\{p_i\}_{i\in \mathbb{N}}$ be the increasing sequence of primes and $\{(\Z/p_i^{a_i})^{d_1},(\Z/p_i)^{d_2}\}\acts N/\Gamma$ be the free iterated actions provided by the condition $\D_2(N/\Gamma)=(d_1,d_2)$. Let $C$ be the Minkowski constant of $\Out(\Gamma)$. Assume that all $p_i>C$ and hence the group morphism $\psi_i:(\Z/p_i^{a_i})^{d_1}\longrightarrow \Out(\Gamma)$ is trivial for all $i$. By \cref{aspherical manifolds big diagram} and \cref{inner actions aspherical manifolds}, each action induces a commutative diagram of fundamental groups 
	\[
	\begin{tikzcd}
		& 1\ar{d}{}& 1\ar{d}{}& & \\
		1\ar{r}{} & \Zc\Gamma\cong \Z^f \ar{r}{}\ar{d}{} & \Zc\Gamma_i\cong\Z^f\ar{r}{}\ar{d}{}& (\Z/p_i^{a_i})^{f}\ar{r}{}\ar{d}{Id} &1\\
		1\ar{r}{} & \Gamma \ar{r}{}\ar{d}{} & \Gamma_i\ar{r}{}\ar{d}{}& (\Z/p_i^{a_i})^{f}\ar{r}{} &1\\
		& \Gamma'\ar{d}{}\ar{r}{Id} & \Gamma'\ar{d}{} & &\\
		& 1 &1 & & \\
	\end{tikzcd}
	\]	
	where $\Gamma_i$ is the fundamental group of the nilmanifold $(N/\Gamma)/(\Z/p_i^{a_i})^{f}$ and the identity map of the third row is induced by the inclusion $\Inn(\Gamma)\longrightarrow\Aut(\Gamma)$. Moreover, notice that $\rank\Zc\Gamma_i=f$ for all $i$.
	
	We claim that the group morphism $\psi_i':(\Z/p_i)^{d_2}\longrightarrow \Out(\Gamma_i)$ is injective for big enough $i$. If not, there exists $\Z/p_i\leq \ker\psi_i'$ such that the action $\Z/p_i\acts N/\Gamma_i$ is inner. The first step of the iterated action is also inner. Thus, by \cref{inner actions simplifiable}, the free iterated action $\{(\Z/p_i^{a_i})^{f},\Z/p_i\}\acts N/\Gamma$ is equivalent to a free action of an abelian $p_i$-group $G_i$ of $\rank G_i=f$. The free iterated action $\{(\Z/p_i^{a_i})^{d_1},(\Z/p_i)^{d_2}\}\acts N/\Gamma$ is equivalent to $\{G_i,(\Z/p_i)^{d_2-1}\}\acts N/\Gamma$, which contradicts the fact that $\rank_{ab}(\{(\Z/p_i^{a_i})^{d_1},(\Z/p_i)^{d_2}\}\acts N/\Gamma)=f+d_2$.
	
	In consequence, we have a commutative diagram
	\[
	\begin{tikzcd}
		& 1\ar{d}{}& 1\ar{d}{}& & \\
		& \Z^f \ar{r}{Id}\ar{d}{} & \Z^f\ar{d}{}&  &\\
		1\ar{r}{} & \Gamma_i \ar{r}{}\ar{d}{} & \Gamma_{i_2}\ar{r}{}\ar{d}{}& (\Z/p_i)^{d_2}\ar{r}{}\ar{d}{\phi} &1\\
		1\ar{r}{} & \Gamma'\ar{d}{}\ar{r}{} & \Gamma_i'\ar{d}{}\ar{r}{} &(\Z/p_i)^{d_2}\ar{r}{} &1\\
		& 1 &1 & & \\
	\end{tikzcd}
	\]	
	where $\Gamma_{i_2}$ is the fundamental group of the quotient $(N/\Gamma)/\{(\Z/p_i^{a_i})^{d_1},(\Z/p_i)^{d_2}\}$. We claim that $\Gamma_i'$ is torsion-free for $i$ large enough. Firstly, $\Gamma_i'\leq\Aut(\Gamma_i)$. For each $i$ we have an injective morphism $\Aut(\Gamma_i)\longrightarrow \Aut(\Gamma_\Q)$, since all lattices $\Gamma_i\leq N$ are commensurable to $\Gamma$. Moreover, $\Aut(\Gamma_\Q)=\Aut(\mathcal{L}(\Gamma_\Q))\leq\Gl(m,\Q)$ for some $m$, where $\mathcal{L}(\Gamma_\Q)$ denotes the rational Lie algebra of $\Gamma_\Q$. Since any finite subgroup of $\Gl(m,\Q)$ is conjugated to a finite subgroup of $\Gl(m,\Z)$, we can conclude that each $\Gamma_i'$ is Minkowski with a constant $C'$ which does not depend on $i$. Finally, since $\Gamma'$ is torsion-free the torsion of $\Gamma_i'$ injects in $(\Z/p_i)^{d_2}$. Thus, $\Gamma_i'$ is torsion-free if $p_i>C$. 
	
	Since $\Out(\Gamma')$ is Minkowski with a constant not depending on $i$, we have a commutative diagram for $i$ large enough.
	\[
	\begin{tikzcd}
		& 1\ar{d}{}& 1\ar{d}{}& & \\
		1\ar{r}{} & \Zc \Gamma' \ar{r}{}\ar{d}{} & C_{\Gamma_{i}'}(\Gamma')\ar{r}{}\ar{d}{}& (\Z/p_i)^{d_2}\ar{r}{}\ar{d}{Id} &1\\
		1\ar{r}{} & \Gamma' \ar{r}{}\ar{d}{} & \Gamma_{i}'\ar{r}{}\ar{d}{}& (\Z/p_i)^{d_2}\ar{r}{} &1\\
		& \Inn\Gamma'\ar{d}{}\ar{r}{Id} & \Inn\Gamma'\ar{d}{} & &\\
		& 1 &1 & & \\
	\end{tikzcd}
	\]	
	
	Since $\Gamma_i'$ is torsion-free, the group $C_{\Gamma_i'}(\Gamma')$ is torsion-free. The first row of the diagram is central, hence $C_{\Gamma_{i}'}(\Gamma')$ is abelian and $d_2\leq \rank\Zc\Gamma'$. Furthermore, $\Gamma'$ is the fundamental group of the nilmanifold $N'/\Gamma'$, obtained as the orbit space of the free action of $\Zc N/\Zc\Gamma\cong T^f$ on $N/\Gamma$. Therefore, $\rank\Zc\Gamma'\leq\dim(N'/\Gamma')=n-f$. In conclusion, $d_2\leq n-f$. If $d_2=n-f$ then $N'/\Gamma'\cong T^{n-f}$ and $N/\Gamma$ is a $2$-step nilmanifold with $\rank\Zc\Gamma=f$, as we wanted to see.
	
	Let us now prove the converse implication. Suppose that $N/\Gamma$ is a 2-step nilmanifold with $f=\rank\Zc\Gamma$ and let $b=n-f$. Then $\Gamma\cong \Z^f\times_c\Z^b$, where $c:\Z^b\times\Z^b \longrightarrow \Z^f$ is a 2-cocycle representing a cohomology class $[c]\in H^2(\Z^b,\Z^f)=H^2(T^b,\Z^f)$ which determines the principal $T^f$-bundle $\pi:N/\Gamma\longrightarrow T^b$. The cocycle can be chosen to be of the form $$c((x_1,\dots,x_b),(y_1,\dots,y_b))=(\sum_{1\leq i<j\leq b} c_{i,j}^1x_iy_j,\dots,\sum_{1\leq i<j\leq b} c_{i,j}^fx_iy_j)$$ where $c_{i,j}^k\in\Z$ for all $1\leq i,j\leq f$ and $1\leq k\leq b$.
	
	For any prime $p$ we can define a free iterated action $\{(\Z/p^{2})^{f},(\Z/p)^{b}\}\acts N/\Gamma$. The first action is by right multiplication on the fiber $T^f$. The orbit space is a nilmanifold with fundamental group $\Gamma_p\cong (\tfrac{1}{p^2}\Z)^f\times_c\Z^b$, where $c:\Z^b\times\Z^b \longrightarrow \Z^f\subseteq (\tfrac{1}{p^2}\Z)^f$. The second step of the iterated action is by rotations on the torus of the basis $T^b$. More explicitly, we define the lattice $\Gamma'_p\cong (\tfrac{1}{p^2}\Z)^f\times_{c_p}(\tfrac{1}{p}\Z)^b$, where $c_p:(\tfrac{1}{p}\Z)^b\times(\tfrac{1}{p}\Z)^b \longrightarrow \Z^f\subseteq (\tfrac{1}{p^2}\Z)^f$ is of the form $$c((\tfrac{1}{p}x_1,\dots,\tfrac{1}{p}x_b),(\tfrac{1}{p}y_1,\dots,\tfrac{1}{p}y_b))=(\sum_{1\leq i<j\leq b} \tfrac{c_{i,j}^1}{p^2}x_iy_j,\dots,\sum_{1\leq i<j\leq b} \tfrac{c_{i,j}^f}{p^2}x_iy_j).$$ 
	We have that $\Gamma_p\trianglelefteq \Gamma_p'$ and  $\Gamma_p'/\Gamma_p\cong(\Z/p)^b$, which defines a free action of $(\Z/p)^b$ on $N/\Gamma_p$. 
	
	Moreover, $\rank_{ab}(\{(\Z/p^{2})^{f},(\Z/p)^{b}\}\acts N/\Gamma)=n$, therefore $(f,b)\leq \D_2(N/\Gamma)$. But the first part of the theorem implies that $\D_2(N/\Gamma)\leq(f,b)$. In consequence $\D_2(N/\Gamma)=(f,b)$.
\end{proof}

\section[Free iterated group actions and non-zero degree map to a nilmanifold]{Free iterated group actions on manifolds admitting a non-zero degree map to a nilmanifold}\label{sec: iterated actions hypernilmanifolds}

The aim of this section is to study the relation between the iterated discrete degree of symmetry and rigidity on manifolds which admit a non-zero degree map to a nilmanifold. In particular, we prove \cref{main theorem11 intro}. We start with a result on closed connected oriented manifolds admitting a non-zero degree map to a torus.

\begin{prop}\label{discsym2 hypertoral}
	Let $M$ be a closed oriented connected manifold of dimension $n$ and $f:M\longrightarrow T^n$ a non-zero degree map. Then $\D_2(M)\leq (n,0)$.
\end{prop}

\begin{proof}
	Suppose that $\D_2(M)=(d_1,d_2)$. Let $\{p_i\}_{i\in\mathbb{N}}$ be the increasing sequence of primes and let $\{(\Z/p_i^{a_i})^{d_1},(\Z/p_i)^{d_2}\}\acts M$ be the collection of free iterated actions given by the definition of the iterated discrete degree of symmetry. Note that we can assume without loss of generality that $p_i>\max\{\deg(f),C\}$ for all $i$, where $C$ is the Minkowski constant of $\Gl(r,\Z)$ with $r=\rank H^1(M,\Z)$. Consequently, the induced action of $(\Z/p_i^{a_i})^{d_1}$ on $H^1(M,\Z)$ is trivial for all $i$.
	
	By \cref{pE torus}, for each $i$ there exists a group morphism $\eta_i:(\Z/p_i^{a_i})^{d_1}\longrightarrow T^n$, which is injective since $|\ker\eta_i|\leq d<p_i$. Moreover, for each $i$ there exists a $\eta_i$-equivariant map $f_i:M\longrightarrow T^n$ homotopic to $f$. We consider the commutative diagram
	\[
	\begin{tikzcd}
		M \ar{r}{f_i}\ar{d}{} & T^n\ar{d}{}\\
		{M}_i\cong M/(\Z/p_i^{a_i})^{d_1} \ar{r}{f'_i} & T^n/(\Z/p_i^{a_i})^{d_1}\cong T^n
	\end{tikzcd}
	\]
	
	Note that $\deg(f'_i)=\deg(f)$ for all $i$. The first cohomology group does not have torsion, hence $\rank H^1({M}_i,\Z)=\rank H^1({M}_i,\Q)=\rank H^1(M,\Q)^{(\Z/p_i^{a_i})^{d_1}}=\rank H^1(M,\Q)=r$, where the last equality holds because the action of $(\Z/p_i^{a_i})^{d_1}$ on $H^1(M,\Q)$ is trivial. Since we have assumed that $p_i> C$ for all $i$ then the action of $(\Z/p_i)^{d_2}$ on $H^1({M}_i,\Z)$ is trivial for all $i$.
	
	Consequently, by \cref{pE torus}, for each $i$ there exists an injective group morphism $\eta'_i:(\Z/p_i)^{d_2}\longrightarrow T^n$ and a $\eta'_i$-equivariant map $h_i':{M}_i\longrightarrow T^n$ homotopic to $f'_i$. Since $T^n\longrightarrow T^n/(\Z/p_i^{a_i})^{d_1}$ is a principal $(\Z/p_i^{a_i})^{d_1}$-bundle, for each $i$ there exists a $(\Z/p_i^{a_i})^{d_1}$-equivariant homeomorphism between the total space of the pull-back $T^n\longrightarrow T^n/(\Z/p_i^{a_i})^{d_1}$ by $h'_i$ and $M$, which we denote by $\phi_i:(h'_i)^*T^n\longrightarrow M$. It induces a commutative diagram
	\[
	\begin{tikzcd}
		(h'_i)^*T^n\ar{rd}{h_i}\ar{d}{\phi_i} &\\	
		M \ar{r}{f_i} & T^n\\
	\end{tikzcd}
	\]
	Note that each $h_i$ is a $\{(\Z/p_i^{a_i})^{d_1},(\Z/p_i)^{d_2}\}$-equivariant map of degree $\deg(f)$. We have a commutative diagram 
	\[
	\begin{tikzcd}
		M \ar{r}{h_i}\ar{d}{} & T^n\ar{d}{}\\
		M_i \ar{r}{h_i'}\ar{d}{} & T^n/(\Z/p_i^{a_i})^{d_1}\cong T^n\ar{d}{}\\
		M_i/(\Z/p_i)^{d_2} \ar{r}{} & T^n/(\Z/p_i)^{d_2}\cong T^n
	\end{tikzcd}
	\]
	The induced iterated action $\{(\Z/p_i^{a_i})^{d_1},(\Z/p_i)^{d_2}\}\acts T^n$ is simplifiable by an abelian $p_i$-group $G_i$ for each $i$. 
	
	The map $h_i$ maps bijectively iterated orbits of $\{(\Z/p_i^{a_i})^{d_1},(\Z/p_i)^{d_2}\}\acts T^n$ to orbits of the action of $G_i$ on $T^n$. In consequence, we can simplify $\{(\Z/p_i^{a_i})^{d_1},(\Z/p_i)^{d_2}\}\acts M$ with a free action of $G_i$ on $M$. This implies that $d_1+d_2=\rank_{ab}(\{(\Z/p_i^{a_i})^{d_1},(\Z/p_i)^{d_2}\})=\rank(G_i)\leq n$. If $d_1<n$ then $\D_2(M)<(n,0)=\D_2(T^n)$. If $d_1=n$ then $d_2=0$ and therefore $\D_2(M)=(n,0)=\D_2(T^n)$.
	
\end{proof}

\begin{thm}\label{free iterated actions rigidity hypernilmanifolds}
	Let $M$ be a closed oriented connected $n$-dimensional manifold $M$ and $f:M\longrightarrow N/\Gamma$ a non-zero degree map where $N/\Gamma$ is a $2$-step nilmanifold. Then:
	\begin{itemize}
		\item[1.] $\D_2(M)\leq \D_2(N/\Gamma)$.
		\item[2.] If the map $f_*:\pi_1(M)\longrightarrow \Gamma$ is surjective and $\D_2(M)= \D_2(N/\Gamma)$ then $H^*(M,\Z)\cong H^*(N/\Gamma,\Z)$.
	\end{itemize}
\end{thm}

Before proving the theorem, let us show what happens if we remove the hypothesis of $f_*:\pi_1(M)\longrightarrow \Gamma$ being surjective from \cref{free iterated actions rigidity hypernilmanifolds}.2.

\begin{cor}\label{iterated action rational cohomology}
	Let $M$ be a closed oriented connected $n$-dimensional manifold $M$ and $f:M\longrightarrow N/\Gamma$ a non-zero degree map where $N/\Gamma$ is a $2$-step nilmanifold. If $\D_2(M)= \D_2(N/\Gamma)$ then $H^*(M,\Q)\cong H^*(N/\Gamma,\Q)$.
\end{cor}  

\begin{proof}
	The condition $d=\deg(f)\neq0$ implies that $[\Gamma:f_*\pi_1(M)]<\infty$. Therefore $f_*\pi_1(M)$ is a lattice of $N$ and we have a finite covering $N/f_*\pi_1(M)\longrightarrow N/\Gamma$. We can lift the map $f$ to a map $f'$, obtaining a commutative diagram 
	\[
	\begin{tikzcd}
		& N/f_*\pi_1(M)\ar{d}{}\\
		M\ar{r}{f}\ar{ru}{f'}& N/\Gamma
	\end{tikzcd}
	\]
	
	Note that $\deg(f')\neq 0$ and $f'_*:\pi_1(M)\longrightarrow f_*\pi_1(M)$ is surjective. Thus, we have a chain of isomorphisms $$H^*(M,\Q)\cong H^*(N/f_*\pi_1(M),\Q)\cong H^*(\mathcal{L}(N),\Q)\cong H^*(N/\Gamma,\Q).$$ 
\end{proof}

The proof of \cref{free iterated actions rigidity hypernilmanifolds} is a generalization of the arguments used to prove \cref{large group actions hypertoral manifolds}.2. in \cite{mundet2021topological}. We will divide the proof in six parts. The first two parts are analogous to \cref{discsym2 hypertoral} and will prove the first part of \cref{free iterated actions rigidity hypernilmanifolds}. We denote by $\tilde{M}$ the total space of the pull-back of $\rho:N\longrightarrow N/\Gamma$ by $f$. In the third and fourth part of the proof, we assume that $\D_2(M)=\D_2(N/\Gamma)$ and we discuss the structure of $H^*(\tilde{M},\Z)$ as $\Z\Gamma$-module. In part 5, we use non-commutative ring theory to prove that $\tilde{M}$ is an acyclic manifold. From this fact we deduce that $H^*(M,\Z)\cong H^*(N/\Gamma,\Z)$ in part 6.

Before starting the proof we fix some notation. We set $a=\dim \Zc N$ and $b=\dim N/\Zc N$. We have a projection map $\pi:N\longrightarrow \R^b$ with $\ker\pi=\Zc N$. thus, we have a short exact sequence \[\begin{tikzcd}
	1\ar{r}{}& \R^a\ar{r}{}& N\ar{r}{\pi}& \R^b\ar{r}{}& 1
\end{tikzcd}\]
Since $\Zc N$ is lattice hereditary, we have $\Zc\Gamma=\Zc N\cap \Gamma\cong \Z^a$ and $\pi(\Gamma)\cong \Z^b$. We have a normalized 2-cocycle $c:\Z^b\times \Z^b\longrightarrow \Z^a$ such that $\Gamma\cong \Z^a\times_c \Z^b$ and $N\cong \R^a\times_c \R^b$. We can write the 2-cocycle $c$ $$c((x_1,\dots,x_b),(y_1,\dots,y_b))=(\sum_{1\leq j,k\leq b}c_{jk}^1x_jy_k,\dots,\sum_{1\leq j,k\leq b}c_{jk}^ax_jy_k)$$
where $c_{jk}^l\in\Q$ and $c_{jk}^l=-c_{kj}^l$ for all $j,k,l$ and $c(\pi(\Gamma)\times\pi(\Gamma))\subseteq \Zc\Gamma$. 

We denote $e_i=((0,\dots,1,\dots,0),0)\in \Z^a\times_c \Z^b$, where the $1$ is in the $i$-th position. Similarly, we denote $e_i'=(0,(0,\dots,1,\dots,0))\in \Z^a\times_c \Z^b$, where the $1$ is in the $i$-th position. Note that the set $\{e_1,\dots,e_a,e_1',\dots,e'_b\}$ generates $\Gamma$.

Denote by $\tilde{M}$ the pull-back by $f$ of the universal cover $N\longrightarrow N/\Gamma$. 

\textbf{Part 1. First step of the iterated actions:} Assume $\D_2(M)=(d_1,d_2)$ and $\D_2(N/\Gamma)=(a,b)$. 

Let $\{p_i\}$ be the sequence of increasing primes and let $\{a_i\}$ be a sequence of natural numbers such that we have a free action $\{(\Z/p_i^{a_i})^{d_1},(\Z/p_i)^{d_2}\}\acts M$ for each $i$. We can assume without loss of generality that $p_i>\max\{\deg(f),C\}$, where $C$ is the exporting map constant of $f$ (see \cref{pE defn} and \cref{thm: exporting map hypernilmanifolds}). In consequence, there exists an inner free action $(\Z/p_i^{a_i})^{d_1}$ on $N/\Gamma$ and a $(\Z/p_i^{a_i})^{d_1}$-equivariant map $f_i:M\longrightarrow N/\Gamma$ which is homotopic to $f$ for each $i$. In particular, $d_1\leq a$. If the inequality is strict then the first part of the theorem is proven. Thus, we will assume that $d_1=a$.

\textbf{Part 2. Second step of the iterated action:} Before we continue with the proof we introduce some notation. We denote by $M_i$ the orbit space $M/(\Z/p_i^{a_i})^{d_1}$ and by $\pi_i:M\longrightarrow M_i$ the orbit map. The quotient $(N/\Gamma)/(\Z/p_i^{a_i})^{d_1}$ is a nilmanifold with the same simply connected nilpotent Lie group $N$ and lattice $\Gamma_i$, which fits into the short exact sequence
\[
\begin{tikzcd}
	1 \ar{r}{}& \Gamma \ar{r}{}& \Gamma_i \ar{r}{}& (\Z/p_i^{a_i})^{d_1}  \ar{r}{}&1 .
\end{tikzcd}
\]

The map $f_i$ induces a map $f_i':M_i\longrightarrow N/\Gamma_i$ for each $i$. We have a commutative diagram
\[
\begin{tikzcd}
	M\ar{dd}{\pi_i}\ar{rr}{f_i} && N/\Gamma\ar{dd}{}\ar{rd}{q}&\\
	&&& T^b\ar{dd}{id}\\
	M_i\ar{rr}{f'_i} && N/\Gamma_i\ar{rd}{q_i}&\\
	&&& T^b
\end{tikzcd}
\]

Note that $q_i\circ f_i'\circ\pi_i$ is homotopic to $q\circ f$ and $\deg(f_i')=\deg(f)$. 

By Minkowski's lemma, there exists a constant $C'$ such that $(\Z/p_i^{a_i})^{d_1}$ acts trivially on $H^*(M,\Q)$ for all $p_i\geq C'$. The map $\pi_i$ induces an isomorphism in cohomology $H^*(M_i,\Q)\cong H^*(M,\Q)^{(\Z/p_i^{a_i})^{d_1}}=H^*(M,\Q)$. Since the first cohomology group has no torsion, we have $H^1(M_i,\Z)\cong H^1(M,\Z)$. The action of $(\Z/p_i)^{d_2}$ on $M_i$ induces an action of $(\Z/p_i)^{d_2}$ on $H^1(M_i,\Z)$. Since $H^1(M_i,\Z)$ does not depend on $i$ up to isomorphism, we can use Minkowski's lemma again to conclude that there exists a constant $C''>C'$ such that the action of $(\Z/p_i)^{d_2}$ on $H^1(M_i,\Z)$ is trivial for $p_i\geq C''$. Thus, we can assume without loss of generality that the action of $(\Z/p_i)^{d_2}$ on $H^1(M_i,\Z)$ is trivial.

We consider the map $q_i\circ f_i':M_i\longrightarrow T^b$, for each $i$. Since $(\Z/p_i)^{d_2}$ acts trivially on $H^1(M_i,\Z)$, there exist a group morphism $\eta_i':(\Z/p_i)^{d_2}\longrightarrow T^b$ and a $\eta_i'$-equivariant map $F_i:M_i\longrightarrow T^b$ homotopically equivalent to $q_i\circ f_i'$. If $\eta_i'$ is injective then $\D_2(M)=(a,d_2)\leq (a,b)=\D_2(N/\Gamma)$ and the first part of the theorem would be proved.

Our next goal is to prove that $\eta_i'$ is injective. We divide the proof in two lemmas, since the first lemma will be also useful in other steps of the proof. Recall that $X(\pi_1(M_i),N)$ is the set of isomorphism classes of $N$-local systems, $X(\pi_1(M_i),N)=\Hom(\pi_1(M),N)/\sim$.

\begin{lem}\label{eta trivial action}
	There exists a constant $D$ such that the action of $(\Z/p_i)^{d_2}$ on $X(\pi_1(M_i),N)$ is trivial for $p_i\geq D$.
\end{lem}

\begin{proof}
	We will prove that the restriction of the action of $(\Z/p_i)^{d_2}$ on $X(\pi_1(M_i),N)$ to any cyclic subgroup of $\Z/p_i$ is trivial. Thus, we will study actions of $\Z/p_i$ on $X(\pi_1(M_i),N)$. Recall that we also assume that $p_i>C''$ and therefore $\Z/p_i$ acts trivially on $H^1(M_i,\Z)$.  
	
	Let us fix some more notation. Let $\phi_i\in\Homeo(M_i)$ be the homeomorphism induced by the action of $\Z/p_i$ on $M_i$ corresponding to $\overline{1}\in\Z/p_i$, $\pi:N\longrightarrow \R^b$ is the projection such that $\ker \pi=\Zc N\cong \R^a$ and $\iota_i:\Gamma_i\longrightarrow N$ is an injective morphism. Set $\mu_i=\iota_i\circ f'_{i*}$ and consider the representative $[\mu_i]\in X(\pi_1(M_i),N)$. We consider the morphism $q_i\circ \mu_i\in X(\pi_1(M_i),\R^b)=\Hom(\pi_1(M_i),\R^b)$ (note that we do not have the conjugation equivalence relation since $\R^b$ is abelian). Since $\Z/p_i$ acts trivially on $H^1(M_i,\Z)$, we have $q_i\circ \mu_i \circ \phi_{i*}=q_i\circ \mu_i$. Therefore, we can define a map $\zeta_i:\pi_1(M_i) \longrightarrow \R^f$ such that 
	$$\zeta_i(\alpha)=\mu_i(\phi_{i*}(\alpha))\mu_i(\alpha)^{-1} $$ for $\alpha\in\pi_1(M_i)$. It is well defined since the image is inside $\ker \pi=\Zc N\cong \R^a$ and it is also a group morphism. Indeed, we have
	
	\begin{align*}
		\zeta_i(\alpha\beta)&=\mu_i(\phi_{i*}(\alpha))\mu_i(\phi_{i*}(\beta))\mu_i(\beta)^{-1}\mu_i(\alpha)^{-1}\\&=\mu_i(\phi_{i*}(\alpha))\mu_i(\alpha)^{-1}\mu_i(\phi_{i*}(\beta))\mu_i(\beta)^{-1}=\zeta_i(\alpha)\zeta_i(\beta)
	\end{align*}
	
	for $\alpha,\beta\in \pi_1(M_i)$, where we use that $\mu_i(\phi_{i*}(\beta))\mu_i(\beta)^{-1}\in\Zc N$. Consequently, $\zeta_i\in\Hom(\pi_1(M_i),\R^a)$. Note that $\Z/p_i$ acts on $\Hom(\pi_1(M_i),\R^a)$ by precomposition, therefore it fixes a lattice of $\R^a$. Thus, by Minkowski lemma, there exists a constant $D>C''$ such that $\Z/p_i$ acts trivially on $\Hom(\pi_1(M_i),\R^a)$. From now on we assume that $p_i\geq D$. This implies that $\zeta_i\circ\phi_{i*}=\zeta_i$.
	
	On the other hand $[\mu_i\circ \phi^p_{i*}]=[\mu_i]$ and therefore there exists $n\in N$ such that $n\mu_i (\phi^p_{i*}(\alpha))=\mu_i(\alpha)n$ for all $\alpha\in \pi_1(M_i)$. A computation shows that
	$$[\mu_i(\alpha),n]=\sum_{j=0}^{p_i-1}\zeta_i(\phi^{j}_{i*}(\alpha))=\zeta_i(\alpha)^{p_i}.$$
	
	Recall that $N=\R^a\times_c\R^b$ where $c:\R^b\times \R^b\longrightarrow \R^a$ is a normalized $2$-cocycle such that $c(\pi(\Gamma)\times\pi(\Gamma))\subseteq\Zc\Gamma$. Since $\pi(\Gamma)=\pi(\Gamma_i)$ for all $i$, we have $c(\pi(\Gamma_i)\times\pi(\Gamma_i))\subseteq\Zc\Gamma\subseteq\Zc\Gamma_i$ for all $i$. Using $N\cong \R^a\times_c \R^b$, we can write $n=(x,y)$ and $\mu_i(\alpha)=(\mu_i'(\alpha),\pi(\mu_i(\alpha)))$.
	
	A straightforward computation shows that the conjugation by $n$ takes the form 
	$$c_n(\mu_i(\alpha))=(\mu_i'(\alpha)+c(y,\pi(\mu_i(\alpha)))-c(\pi(\mu_i(\alpha)),y),\pi(\mu_i(\alpha))).$$ 
	Hence
	$$\zeta_i(\alpha)^{p_i}=c(y,\pi(\mu_i(\alpha)))-c(\pi(\mu_i(\alpha)),y).$$
	We define $m=(x,\tfrac{y}{p})\in \R^a\times_c\R^b$, which satisfies
	$$\zeta_i(\alpha)=c(\tfrac{y}{p},\pi(\mu_i(\alpha)))-c(\pi(\mu_i(\alpha)),\tfrac{y}{p}).$$ Summarising the previous computations, we showed that $\mu_i\circ\phi_{i*}=c_{m}\circ\mu_i$ which implies that $[\mu_i\circ\phi_{i*}]=[\mu_i]\in X(\pi_1(M_i),N)$, as we wanted to prove.
\end{proof}

\begin{lem}\label{eta injective}
	The morphism $\eta_i'$ is injective for all $p_i\geq D$.
\end{lem}
\begin{proof}
	Assume that $\eta'_i$ is not injective and take $\Z/p_i\leq \ker\eta'_i$. Then, the free iterated action of $\{(\Z/p_i^{a_i})^{a},\Z/p_i,(\Z/p_i)^{d_2-1}\}\acts M$ is equivalent to $\{(\Z/p_i^{a_i})^a,(\Z/p_i)^{d_2}\}\acts M$. We will prove that $\{(\Z/p^{a_i}_i)^{a},\Z/p_i\}\acts M$ is simplifiable by an abelian group of rank $a$. This will imply that $\rank_{ab}(\{(\Z/p_i^{a_i})^a,(\Z/p_i)^{d_2}\}\acts M)<a+d_2$, which is a contradiction with the assumption that $\D_2(M)=(a,d_2)$.
	
	Firstly, note that $F_i(gx)=\eta'_i(g)F_i(x)=F_i(x)$ for all $g\in\Z/p$ and $x\in M_i$. Using the homotopy lifting property, we can replace $f_i'$ by a homotopic map $f''_i:M_i\longrightarrow N/\Gamma_i$ such that $f''_i=q_i\circ F_i$. Note that the orbits of the action of $\Z/p_i$ are inside the fibers of $q_i:N/\Gamma_i\longrightarrow T^b$.
	
	By \cref{eta trivial action} and \cref{pE local systems trivial action}, there exists a $\Z/p_i$ action on $N/\Gamma_i$ and a $\Z/p_i$-equivariant map $h'_i:M_i\longrightarrow N/\Gamma_i$ homotopic to $f_i''$, and hence homotopic to $f_i'$. The orbit space $(N/\Gamma_i)/(\Z/p_i)$ is a nilmanifold $N/\Gamma_i'$. The equality $\eta_i'(\Z/p_i)=0$ implies that $\pi(\Gamma_i)=\pi(\Gamma_i')$ which leads to the commutative diagram \[\begin{tikzcd}
		& 1\ar{d}{}& 1 \ar{d}{}& & \\
		1\ar{r}{}& \Zc\Gamma_i\ar{d}{}\ar{r}{}& \Zc\Gamma_i \ar{d}{}\ar{r}{}& \Z/p_i \ar{d}{id}\ar{r}{}& 1\\
		1\ar{r}{}& \Gamma_i\ar{d}{}\ar{r}{}& \Gamma'_i \ar{d}{}\ar{r}{p}&\Z/p_i \ar{r}{}& 1\\
		& \pi(\Gamma_i)\ar{d}{}\ar{r}{id}&  \pi(\Gamma'_i) \ar{d}{}& & \\	
		& 1& 1 &  & \\
	\end{tikzcd}\]
	
	The action of $\Z/p_i$ on $N/\Gamma_i$ is inner, hence $\{(\Z/p^{a_i}_i)^{a},\Z/p_i\}\acts N/\Gamma$ is simplifiable by an abelian group $A_i$ (see \cref{inner actions simplifiable}). Thus, $\rank_{ab}(A_i)\leq a=\rank\Zc\Gamma$. In conclusion $\{(\Z/p_i^{a_i})^{a},(\Z/p_i)^{d_2}\}\acts M$ is equivalent to $\{A_i,(\Z/p_i)^{d_2-1}\}\acts M$ and therefore $\rank_{ab}(\{(\Z/p_i^{a_i})^{a},(\Z/p_i)^{d_2}\}\acts M)\leq a+d_2-1$, as we wanted to see.
\end{proof}

Since $\eta'_i:(\Z/p_i)^{d_2}\longrightarrow T^b$ is injective, $d_2\leq b$. Thus, the proof of item 1 of \cref{free iterated actions rigidity hypernilmanifolds} is completed.

If $\eta_i':(\Z/p_i)^{d_2}\longrightarrow T^b$ is injective then the action of $(\Z/p_i)^{d_2}$ on $\Hom(\pi_1(M_i),\R^b)$ is trivial. By \cref{eta trivial action}, the action of $(\Z/p_i)^{d_2}$ on $\Hom(\pi_1(M_i),N)$ is trivial. Thus, we can construct a $(\Z/p_i)^{d_2}$ action on $N/\Gamma_i$ and a $(\Z/p_i)^{d_2}$-equivariant map $h'_i:M_i\longrightarrow N/\Gamma_i$ such that $q_i\circ h'_i:M_i\longrightarrow T^b$ is $\eta_i'$-equivariant.

\textbf{Part 3. $H^*(\tilde{M},\Z)$ as a $\Z\Gamma$-module:} The action of $\Gamma$ on $\tilde{M}$ induces an action on $H^*(\tilde{M},\Z)$, which we denote by $\Phi:\Gamma\longrightarrow \Aut_\Z(H^*(\tilde{M},\Z))$. Thus, $H^*(\tilde{M},\Z)$ has a structure of $\Z\Gamma$-module. To prove the following lemma we use the same argument as in \cite[Lemma 8.2]{mundet2021topological}

\begin{lem}\label{cohomology fg group ring}
	$H^*(\tilde{M},\Z)$ is finitely generated as a $\Z\Gamma$-module.
\end{lem}

\begin{proof}
	Recall that compact manifolds are Euclidean Neighbourhood Retracts (see, for example, \cite[Corollary A.9]{hatcher2002algebraic}) and therefore we can identify $M$ with a closed subset of $\R^m$ for some $m$ large enough such that there exist an open neighbourhood $U\subseteq \R^m$ of $M$ and a retraction $r:U\longrightarrow M$. Given $x\in M$ let $B_x$ denote an open ball centred at $x$ and contained in $U$. By the compactness of $M$ there exists a collection of point $x_1,\dots, x_s$ in $M$ such that $M\subseteq\bigcup_{i=1}^s B_{x_i}=B$. Let $F=f\circ r:B\longrightarrow N/\Gamma$, $B_i=B_{x_i}$ and $F_i=f\circ r:B_i\longrightarrow N/\Gamma$.
	
	Since $B_{i}$ is contractible, the principal $\Gamma$-bundle $F_i^*\pi:F_i^*N\longrightarrow B_i$ obtained by pulling back the universal covering $\pi:N\longrightarrow N/\Gamma$ by $F_i$ is trivial. This implies that for every subset $S\subseteq B_i$, we have $H^*((F_i^*\pi)^{-1}(S),\Z)\cong H^*(S,\Z)\otimes_\Z\Z\Gamma$. So if $H^*(S,\Z)$ is finitely generated then $H^*((F_i^*\pi)^{-1}(S),\Z)$ is finitely generated as $\Z\Gamma$-module.
	
	Let $B_{\leq j}=B_1\cup\dots\cup B_j$ and $F_{\leq j}=f\circ r_{|B_{\leq j}}:B_{\leq j}\longrightarrow N/\Gamma$. To ease notation, we set $X_i=(F_i^*\pi)^{-1}(B_i)$ and $X_{\leq i}=(F_{\leq i}^*\pi)^{-1}(B_{\leq i})$. We will prove using the Mayer-Vietoris long exact sequence that $H^*(X_{\leq j},\Z)$ is finitely generated as $\Z\Gamma$-module. 
	
	If $j=1$ then $H^*(X_{\leq 1},\Z)\cong H^*(X_{1},\Z)\cong\Z\Gamma$ and hence it is finitely generated. Assume now that $H^*(X_{\leq j-1},\Z)$ is a finitely generated $\Z\Gamma$-module. By using that $B_{\leq j}=B_{\leq j-1}\cup B_j$, we have a long exact sequence of $\Z\Gamma$-modules:
	$$
	\dots \longrightarrow H^{k}(X_{\leq j},\Z)\longrightarrow H^{k}(X_{\leq j-1},\Z)\oplus H^{k}(X_{j},\Z)\longrightarrow H^{k}(X_{\leq j}\cap X_{j},\Z)\longrightarrow \dots 
	$$
	
	Since $H^*(B_{\leq j-1}\cap B_j,\Z)$ is finitely generated and it is a subset of $B_j$, the cohomology group $H^{k}(X_{\leq j}\cap X_{j},\Z)$ is a finitely generated $\Z\Gamma$-module. By induction hypothesis, we have that $H^{k}(X_{\leq j-1},\Z)\oplus H^{k}(X_{j},\Z)$ is finitely generated. We can conclude that $H^k(X_{\leq j},\Z)$ is finitely generated. Finally, $H^k(X_{\leq j},\Z)=0$ for $k>m$ implies that $H^*(X_{\leq j},\Z)$ is a finitely generated $\Z\Gamma$-module.
	
	Since we have an inclusion $i:M\longrightarrow B$ and a retraction $r:B\longrightarrow M$, the $\Z\Gamma$-module $H^*(\tilde{M},\Z)$ is a $\Z\Gamma$-submodule of $H^*((F^*\pi)^{-1}(B),\Z)$. Since $H^*((F^*\pi)^{-1}(B),\Z)$ is finitely generated and $\Z\Gamma$ is Noetherian (see \cite[Theorem 1.16]{goodearl2004introduction}), $H^*(\tilde{M},\Z)$ is finitely generated as $\Z\Gamma$-module.
\end{proof}

\textbf{Part 4. Setting for the proof of the second part of the theorem:} We suppose now that $(d_1,d_2)=(a,b)$. 
By \cref{thm: exporting map hypernilmanifolds}, there exists a $\{(\Z/p_i^{a_i})^a,(\Z/p_i)^b\}$-equivariant map $h_i:M\longrightarrow N/\Gamma$ homotopic to $f$ for all $i$. Hence, for each $i$ there exist non-zero degree maps $h_i':M_i\longrightarrow N/\Gamma_i$ and $h_i'':M_i/(\Z/p_i)^b=M_i'\longrightarrow N/\Gamma_i'$ such that the following diagram commutes
\[
\begin{tikzcd}
	M\ar{d}{}\ar{r}{h_i} & N/\Gamma \ar{d}{}\\
	M_i\ar{d}{}\ar{r}{h_i'} & N/\Gamma_i \ar{d}{}\\
	M_i'\ar{r}{h_i''} & N/\Gamma_i'
\end{tikzcd}
\]
The vertical arrows are the orbit maps of the iterated group actions on $M$ and $N/\Gamma$. 

In this part of the proof we show that there exists group morphism $\Phi'_i:\Gamma'_i\longrightarrow \Aut_\Z(H^*(\tilde{M},\Z))$ such that $\Phi'_{i|\Gamma}=\Phi$ for each $i$. We will construct $\Phi'_i$ in two steps. First, we construct a group morphism $\Phi_i:\Gamma_i\longrightarrow \Aut_\Z(H^*(\tilde{M},\Z))$ such that $\Phi_{i|\Gamma}=\Phi$. Thereafter, we construct a group morphism $\Phi'_i:\Gamma'_i\longrightarrow \Aut_\Z(H^*(\tilde{M},\Z))$ such that $\Phi'_{i|\Gamma_i}=\Phi_i$. 

\begin{lem}\label{Gamma-modules 1}
	The action of $\Gamma_i$ on $H^*(\tilde{M},\Z)$ induces a morphism $\Phi_i:\Gamma_i\longrightarrow \Aut_\Z(H^*(\tilde{M},\Z))$ such that $\Phi_{i|\Gamma}=\Phi$ for all $i$.
\end{lem}

\begin{proof} 
	We denote by $\tilde{M}_i$ the total space of the pull-back of $N\longrightarrow N/\Gamma$ by $h_i:M\longrightarrow N/\Gamma$. Similarly, we denote by $\tilde{M}_i'$ the pull-back of $N\longrightarrow N/\Gamma_i$ by $h_i':M_i\longrightarrow N/\Gamma_i$. 
	
	We also have a free action of $\Gamma_i'$ on $\tilde{M}_i'$. We consider the commutative diagram
	\[\begin{tikzcd}
		\tilde{M}_i\ar{dd}{}\ar{rr}{}\ar{rd}{\zeta_i'} & & N\ar{rd}{Id_N}\ar{dd}{}&\\
		& \tilde{M}_i'\ar{dd}{}\ar{rr}{} & & N\ar{dd}{}\\
		M\ar{rr}{h_i}\ar{rd}{} & & N/\Gamma\ar{rd}{}&\\
		& M'_i\ar{rr}{h_i'} & & N/\Gamma_i
	\end{tikzcd}\]
	
	The maps $\tilde{M}_i'\longrightarrow M_i$ and $\tilde{M}_i\longrightarrow M_i$ are isomorphic coverings. Thus, $\tilde{M}_i$ admits a free action of $\Gamma_i$ and the map $\zeta_i':\tilde{M}_i\longrightarrow \tilde{M}'_i$ is a $\Gamma_i$ equivariant homeomorphism. By construction, the restriction to $\Gamma$ of the action of $\Gamma_i$ on $\tilde{M}_i$ is the action induced by the pull-back of $N\longrightarrow N/\Gamma$ by $h_i:M\longrightarrow N/\Gamma$. Thus, the action of $\Gamma_i$ on $\tilde{M}_i$ induces a group morphism $\Psi'_i:\Gamma_i\longrightarrow \Aut_\Z(H^*(\tilde{M_i},\Z))$ such that $\Psi'_{i|\Gamma}=\Psi_i$.

	Since, $h_i$ is homotopic to $f$, there exists a $\Gamma$-equivariant homeomorphism $\zeta_i:\tilde{M}\longrightarrow \tilde{M}_i$ which induces an isomorphism of $\Z\Gamma$-modules $\zeta_i^*:H^*(\tilde{M}_i,\Z)\longrightarrow H^*(\tilde{M},\Z)$. We define $\Phi_i:\Gamma_i\longrightarrow \Aut_\Z(H^*(\tilde{M},\Z))$ as $\Phi_i(\gamma)=\zeta_i^*\circ\Psi_i(\gamma)\circ(\zeta_i^*)^{-1}$. Since $\zeta_i^*$ is an isomorphism of $\Z\Gamma$, we have $\Phi_{i|\Gamma}=\Phi$.
\end{proof}

The proof of the next lemma is analogous to the proof of \cref{Gamma-modules 1}.

\begin{lem}\label{Gamma-module 2}
	The action of $\Gamma'_i$ on $H^*(\tilde{M},\Z)$ induces a morphism $\Phi'_i:\Gamma'_i\longrightarrow \Aut_\Z(H^*(\tilde{M},\Z))$ such that $\Phi'_{i|\Gamma}=\Phi$.
\end{lem}

\begin{proof}
	We continue using the same notation introduced in \cref{Gamma-modules 1}. We denote by $\tilde{M}_i''$ the pull-back of $N\longrightarrow N/\Gamma'_i$ by $h_i'':M_i\longrightarrow N/\Gamma'_i$.
	
	We also have a free action of $\Gamma'_i$ on $\tilde{M}_i'$. We consider the commutative diagram
	\[\begin{tikzcd}
		\tilde{M}'_i\ar{dd}{}\ar{rr}{}\ar{rd}{\zeta_i''} & & N\ar{rd}{Id_N}\ar{dd}{}&\\
		& \tilde{M}''_i\ar{dd}{}\ar{rr}{} & & N\ar{dd}{}\\
		M_i\ar{rr}{h'_i}\ar{rd}{} & & N/\Gamma_i\ar{rd}{}&\\
		& M'_i\ar{rr}{h_i''} & & N/\Gamma'_i
	\end{tikzcd}\]
	
	The maps $\tilde{M}_i'\longrightarrow M_i'$ and $\tilde{M}''_i\longrightarrow M_i'$ are isomorphic coverings. Thus, $\tilde{M}_i'$ admits a free action of $\Gamma_i'$ and the map $\zeta_i'':\tilde{M}_i\longrightarrow \tilde{M}'_i$ is a $\Gamma_i'$-equivariant homeomorphism. By construction, the restriction to $\Gamma_i$ of the action of $\Gamma_i'$ on $\tilde{M}_i$ is the action induced by the pull-back of $N\longrightarrow N/\Gamma$ by $h_i:M\longrightarrow N/\Gamma$. Thus, the action of $\Gamma_i'$ on $\tilde{M}_i'$ induces a group morphism $\Psi''_i:\Gamma_i'\longrightarrow \Aut_\Z(H^*(\tilde{M}_i',\Z))$ satisfying $({\zeta_i'}^*\circ \Psi''_i \circ ({\zeta_i'}^*)^{-1})_{|\Gamma_i}=\Psi_i'$, where ${\zeta_i'}^*:H^*(\tilde{M}'_i,\Z)\longrightarrow H^*(\tilde{M}_i,\Z)$ is the isomorphism of $\Z\Gamma_i$-modules induced by $\zeta_i'$ defined in \cref{Gamma-modules 1}.
	
	Finally, we can define $\Phi_i''$ as $\Phi_i''(\gamma)=({\zeta_i^*\circ\zeta_i'}^*)\circ \Psi''_i(\gamma) \circ ({\zeta_i^*\circ\zeta_i'}^*)^{-1}$. Since ${\zeta_i'}^*$ is an isomorphism of  $\Z\Gamma_i$-modules, we have $\Phi''_{i|\Gamma_i}=\Phi_i$.
\end{proof}

\textbf{Part 5. $H^*(\tilde{M},\Z)$ is a finitely generated $\Z$-module:} 
Our objective in part 5 of the proof will be to prove that $H^*(\tilde{M},\Z)$ is finitely generated as a $\Z$-module. 

Recall that $\Gamma_i/\Gamma\cong \Zc\Gamma_i/\Zc\Gamma\cong (\Z/p_i^{a_i})^a$ and $\Gamma'_i/\Gamma_i\cong q(\Gamma_i)/q(\Gamma)\cong (\Z/p_i)^b$. Note that $\Gamma_i\cong (\tfrac{1}{p_i^{a_i}}\Z)^f\times_c\Z^b$ and $\Gamma_i'\cong (\tfrac{1}{p_i^{a_i}}\Z)^f\times_c(\tfrac{1}{p_i}\Z)^b$. The lattice $\Gamma_i'$ is generated $\{\tfrac{1}{p_i^{a_i}}e_1,\dots,\tfrac{1}{p_i^{a_i}}e_a,e_1',\dots,e'_b\}$ and $\Gamma_i''$ is generated by $\{\tfrac{1}{p_i^{a_i}}e_1,\dots,\tfrac{1}{p_i^{a_i}}e_a,\tfrac{1}{p_i}e_1',\dots,\tfrac{1}{p_i}e'_b\}$.

By \cref{Gamma-modules 1} and \cref{Gamma-module 2}, for each $1\leq j\leq a$ there exists a collection $\{w_{j,i}\}_{i\in\mathbb{N}}$ of automorphisms of $H^*(M,\Z)$ satisfying that $(w_{j,i})^{p_i^{a_i}}=\Phi(e_j)$. Similarly, for each $1\leq j\leq b$ there exists a collection $\{w'_{j,i}\}_{i\in\mathbb{N}}$ of automorphisms of $H^*(M,\Z)$ satisfying that $(w_{j,i}')^{p_i}=\Phi(e'_j)$.

Recall that $\Z\Gamma$ is an iterated skew-Laurent ring generated by $\{e_1^{\pm1},\dots,e_a^{\pm1},{e_1'}^{\pm1},\dots,{e_b'}^{\pm1}\}$ (see \cite[pg. xvii]{goodearl2004introduction}). Explicitly, $\Z\Gamma\cong \Z[e_1^{\pm1},\dots,e_a^{\pm1},{e_1'}^{\pm1},\dots,{e_b'}^{\pm1};\alpha_1,\cdots,\alpha_{a+b}]$ where the automorphism are defined using the cocycle $c$. 

The main theorem of this part is the generalization of \cite[Theorem 6.1]{mundet2021topological}. Be aware that the ring $R$ involved in the theorem below is in general not commutative. We refer to \cite{bell1988notes} and \cite{goodearl2004introduction} for an introduction to the theory of non-commutative Noetherian rings and, in particular, to the theory of localization on those rings.

\begin{thm}\label{non commutative finitely generation modules}
	Let $R$ be a prime Noetherian ring such that any prime ideal $\mathfrak{p}$ is right localizable. Given an automorphism $\alpha:R\longrightarrow R$, we consider the skew-polynomial ring $R[z;\alpha]$. Suppose that $X$ is a finitely generated $R[z;\alpha]$-module and that there exists a sequence of positive integers $r_j\longrightarrow\infty$ and a collection of automorphisms $w_j:X\longrightarrow X$ such that $w_j^{r_j}$ coincides with the multiplication by $z$ on the right. Then $X$ is finitely generated as $R$-module. 
\end{thm}

\begin{proof}
	Let $S=\{x_1,\dots,x_s\}$ be a generating set of $X$ as a $R[z;\alpha]$-module and let $X_0\subseteq X$ be the $R$-module generated by $S$. Consider the increasing sequence of finitely generated $R$-modules $X_0\subseteq X_1\subseteq X_2\subseteq\cdots$ defined by the condition $X_i=X_{i-1}+X_{i-1}z$. Multiplication by $z$ induces surjective morphisms of $R$-modules $\mu_i:X_{i-1}/X_{i-2}\longrightarrow X_{i}/X_{i-1}$ and thus we can define surjective maps $\nu_i=\mu_i\circ\cdots\circ \mu_1:X_0\longrightarrow X_{i}/X_{i-1}$. We have an increasing sequence $K_0\subseteq K_1\subseteq K_2\subseteq \cdots$ of submodules of $X_0$ where $K_i=\ker\nu_i$. Since $R$ is Noetherian and $X_0$ is finitely generated there exists a $i_0$ such that $K_i=K_{i_0}$ for all $i\geq i_0$. In particular, $\mu_i$ is an isomorphism for all $i_0$.
	
	Let $Y=X_{i_0}/X_{i_0-1}$. If $Y=0$ then $X=X_{i_0-1}$ and we are done. Thus we will assume that $Y\neq 0$ and reach a contradiction. Since $Y$ is a finitely generated $R$-module and $R$ is Noetherian, there exists an increasing series of submodules $0=Y_0\subseteq Y_1\subseteq \cdots \subseteq Y_r=Y$ such that $Y_i/Y_{i-1}$ is a prime module (see \cite[Proposition 3.13, Proposition 3.14]{goodearl2004introduction}). Let $\mathfrak{p}$ denote a minimal prime ideal of the collection of the associated prime ideals $\{\mathfrak{p}_1,\dots,\mathfrak{p}_r\}$. We can now consider the right localization $R_\mathfrak{p}$ and we can also consider the localization $Y_\mathfrak{p}$. Since localizing is an exact functor (see \cite[Lemma 1.3]{bell1988notes}), $(Y_{i\mathfrak{p}}/Y_{i-1\mathfrak{p}})=(Y_{i}/Y_{i-1})_{\mathfrak{p}}$. Moreover, $R_\mathfrak{p}/\mathfrak{p}R_\mathfrak{p}$ is simple Artinian (see \cite[Theorem 1.10]{bell1988notes} and \cite[Corollary 2.2]{bell1988notes}), hence the length $\lenght(Y_\mathfrak{p})=\lambda $ of the composition series is finite. 
	
	We will now prove that there cannot exists any $R[z;\alpha]$-module automorphism $w$ such that $w^r=z$ with $r>\lambda$ by reaching a contradiction. Thus, assume that there exists an automorphism $w:X\longrightarrow X$ such that $w^r=z$ and $r>\lambda$. Consider the following $R$-modules defined recursively as $X'_0=X_0$ and $X_i'=X'_{i-1}+w(X'_{i-1})$ for $i\geq 1$. Using the same arguments as above, there exists a $i_0'$ such that $\mu_i':X'_{i-1}/X'_{i-2}\longrightarrow X'_{i}/X'_{i-1}$ is an isomorphism for all $i\geq i_0'$. Let $Y'=X'_{i_0'}/X'_{i_0'-1}$.
	
	It is clear that $X_i\subseteq X'_{ri}$ for all $i$. On the other hand, since $X$ is finitely generated as a $R[z;\alpha]$-module we can write $$w^j(x_k)=x_1P_{jk1}+\cdots+x_sP_{jks}$$ for $j=1,\dots, r-1$ and $k=1,\dots, s$, where $P_{jkl}$ are polynomials in $z$ with coefficients in $R$. Let $e=\max\deg P_{jkl}$. Then $w^j(x_k)\in X_e$ for $j=1,\dots r-1$ and $k=1,\dots, s$. This implies that $X_i'\subseteq X_{[i/r]+e}$ for any $i$. Thus, we have inclusions $X_{i}\subseteq X'_{ir}\subseteq X_{i+e}$. 
	
	The next step is to localize at $\mathfrak{p}$ in order to use the length of the composition series. Firstly, we consider the inclusions $X_{i+L}\subseteq X'_{(i+L)r}\subseteq X_{i+L+e}$ where $L$ is  large number which we will determine below. If we fix $i$ such that $i\geq i_0$ and $i\geq i_0'r$ these inclusions imply that $$ 0\leq \lenght(X_{i+L+e,\mathfrak{p}}/X'_{(i+L)r,\mathfrak{p}})\leq \lenght(X_{i+L+e,\mathfrak{p}}/X_{(i+L)r,\mathfrak{p}})=e\lambda.$$
	
	Moreover, the chain of inclusion $X_{i,\mathfrak{p}}\subseteq X'_{ir,\mathfrak{p}}\subseteq X'_{(i+1)r,\mathfrak{p}}\subseteq \cdots\subseteq X'_{(i+L)r,\mathfrak{p}}\subseteq X_{i+L+e,\mathfrak{p}}$ imply that $$(i+L)\lambda=\lenght(X'_{(i+L)r,\mathfrak{p}}/X'_{ir,\mathfrak{p}})+\lenght(X_{i+L+e,\mathfrak{p}}/X'_{(i+L)r,\mathfrak{p}}) $$ and 
	$$\lenght(X'_{(i+L)r,\mathfrak{p}}/X'_{ir,\mathfrak{p}})=rL\lenght(Y'_\mathfrak{p}). $$
	
	In conclusion, $$\lambda/r\leq \lenght(Y'_\mathfrak{p})\leq \tfrac{L+e}{Lr}\lambda.$$
	
	Note that the lower bound is inside the interval $(0,1)$ and if $L$ is big enough then the upper bound also is inside the interval $(0,1)$, contradicting the fact that $\lenght(Y'_\mathfrak{p})$ is an integer.
\end{proof}

We want now to extend \Cref{non commutative finitely generation modules} to skew-Laurent rings. Thus, we consider the skew-Laurent ring $R[z^{\pm 1};\alpha]$, where $\alpha\in \Aut(R)$. We now construct an iterated skew-polynomial ring as follows (see \cite[Exercise 1R]{goodearl2004introduction}). First, we consider the skew-polynomial ring $R[t_+;\alpha_+]$ where $\alpha_+=\alpha$. We now define a map $\alpha_-:R[t_+;\alpha_+]\longrightarrow R[t_+;\alpha_+]$ satisfying that $\alpha_-(\sum t^ir_i)=\sum t^i\alpha^{-1}(r_i)$. Using that $rt=t\alpha^{-1}(r)$ and that $\alpha$ is an automorphism, it is straightforward to prove that $\alpha_-$ is an automorphism of $R[t_+;\alpha_+]$. Thus, we can consider the iterated skew-polynomial ring $(R[t_+;\alpha_+])[t_-;\alpha_-]$. Consider the map $\mu: (R[t_+;\alpha_+])[t_-;\alpha_-]\longrightarrow R[z^{\pm1};\alpha]$ defined as $\mu(\sum t_-^i(\sum t_+^jr_{ij}))=\sum z^{j-i}r_{ij}$. As before, a routine check shows that $\mu$ is a surjective ring morphism. Thus, if $X$ is a finitely generated $R[z_{\pm1};\alpha]$-module then $X$ is finitely generated as a $(R[t_+;\alpha_+])[t_-;\alpha_-]$-module. Assume that there exists a sequence of positive integers $r_i\longrightarrow\infty$ and a collection of $R[z^{\pm1};\alpha]$-automorphisms $w_i:X\longrightarrow X$ such that $w_i^{r_i}$ coincides with the multiplication by $z$ on the right. We define automorphisms $w_{i+}:X\longrightarrow X$ and $w_{i-}:X\longrightarrow X$ satisfying that $w_{i+}=w_i$ and $w_{i-}=w_i^{-1}$ for all $i$. By construction, $w_{i+}^{r_i}$ coincides with the right multiplication by $t_+$ and $w_{i+}^{r_i}$ coincides with the right multiplication by $t_-$. Thus, we have:

\begin{cor}\label{non commutative finitely generation modules skew-Laurent}
	Let $R$ be a prime Noetherian ring such that any prime ideal $\mathfrak{p}$ is right localizable. Suppose that $M$ is a finitely generated $R[z^{{\pm1}};\alpha]$-module and that there exists a sequence of positive integers $r_j\longrightarrow\infty$ and a collection of automorphisms $w_j:M\longrightarrow M$ such that $w_j^{r_j}$ coincides with the right multiplication by $z$. Finally, assume that $R[t_+;\alpha_+]$ defined as above is also a prime Noetherian ring such that any prime ideal $\mathfrak{p}$ is right localizable. Then $M$ is finitely generated as $R$-module.
\end{cor}

Recall that $\tilde{M}$ denotes the total space of the pull-back of the covering $N\longrightarrow N/\Gamma$ by $f:M\longrightarrow N/\Gamma$. Since the group ring $\Z\Gamma$ satisfies the conditions of \cref{non commutative finitely generation modules skew-Laurent} (see \cite[(A) Connell's Theorem]{lam1991first} and \cite[pg. 17, Appendix B]{bell1988notes}) and they are iterated skew-Laurent rings, we can use downward induction to prove:

\begin{cor}\label{cohomology fg group}
	$H^*(\tilde{M},\Z)$ is a finitely generated $\Z$-module.
\end{cor}

\begin{proof}
	By \cref{cohomology fg group ring}, $H^*(\tilde{M},\Z)$ is a finitely generated as $\Z\Gamma$-module. Recall also that $\Z\Gamma\cong \Z[e_1^{\pm1},\dots,e_a^{\pm1},{e_1'}^{\pm1},\dots,{e_b'}^{\pm1};\alpha_1,\cdots,\alpha_{a+b}]$. By \cref{Gamma-modules 1} and \cref{Gamma-module 2}, there exists a collection of $\{w'_{b,i}\}_{i\in\mathbb{N}}$ of automorphisms satisfying that $(w_{b,i}')^{p_i}$ coincides with the multiplication of $e'_b$. Thus, by \cref{non commutative finitely generation modules skew-Laurent}, $H^*(\tilde{M},\Z)$ is a finitely generated as $\Z\Gamma\cong \Z[e_1^{\pm1},\dots,e_a^{\pm1},{e_1'}^{\pm1},\dots,{e'}^{\pm1}_{b-1};\alpha_1,\cdots,\alpha_{a+b-1}]$. We can repeat this process with each generator to conclude that $H^*(\tilde{M},\Z)$ is a finitely generated as $\Z$-module.
\end{proof}

\textbf{Part 6. Concluding the proof} The last step of the proof \cref{free iterated actions rigidity hypernilmanifolds} is to prove that $H^*(\tilde{M},\Z)$ is acyclic (that is, $H^0(\tilde{M},\Z)\cong\Z$ and $H^i(\tilde{M},\Z)=0$ for all $i>0$). We will prove this statement by contradiction.

Assume that $\tilde{M}$ is not $\Z$-acyclic. Note that $f_*(\pi_1(M))=\Gamma$ implies $H^0(\tilde{M},\Z)\cong \Z$. Hence if $\tilde{M}$ is not $\Z$-acyclic there exists $j>0$ such that $H^j(\tilde{M},\Z)\neq 0$.

By the universal coefficients theorem, there exists a prime $l$ such that $\tilde{M}$ is not $\Z/l$-acyclic. Let $k=\max\{j:H^j(\tilde{M},\Z/l)\neq 0\}>0$. Since $H^*(\tilde{M},\Z)$ is a finitely generated abelian group, then $\Aut(H^*(\tilde{M},\Z/l))$ is finite. Let $\Phi_{(l)}:\Gamma\longrightarrow \Aut(H^*(\tilde{M},\Z/l))$ be the group morphism induced by the action of $\Gamma$ on $\tilde{M}$. The kernel of the morphism $\Phi_{(l)}:\Gamma\longrightarrow \Aut(H^*(\tilde{M},\Z/l))$, which we denote by $\Lambda$, has finite index in $\Gamma$. Hence $\Lambda$ is a lattice of $N$. Consider the diagram
\[\begin{tikzcd}
	\tilde{M}\times_\Lambda N\ar{d}{\pi}\ar{r}{\Theta} & \tilde{M}/\Lambda\\
	N/\Lambda &
\end{tikzcd}\]
where $\Theta$ and $\pi$ are the natural projections. Since $N$ is contractible the map $\Theta$ is a homotopy equivalence and $H^j(\tilde{M}\times_\Lambda N,\Z/l)\cong H^j(\tilde{M}/\Lambda,\Z/l)=0$ for $j>n$.

To compute the cohomology of the other fibration we need to use the Serre spectral sequence. The monodromy action of $\Lambda$ on $H^j(\tilde{M},\Z/l)$ is trivial, thus
$$E_2^{r,s}=H^r(N/\Lambda,H^s(\tilde{M},\Z/l))\cong H^r(N/\Lambda,\Z/l)\otimes H^s(\tilde{M},\Z/l)\implies H^{r+s}(\tilde{M}\times_\Lambda N,\Z/l).$$

We have $E_2^{k,n}\neq 0$, but for dimensional reasons $E_2^{k,n}$ does not belong to the image of any differential. Furthermore, $E_2^{k,n}$ is not inside the kernel of any differential. Consequently, $H^{k+n}(\tilde{M}\times_\Lambda N,\Z/l)\neq 0$ which is a contradiction because $\tilde{M}\times_\Lambda N\cong \tilde{M}/\Lambda$ and $\tilde{M}/\Lambda$ has dimension $n$.

Finally, we take the fibration \[\begin{tikzcd}
	\tilde{M}\times_\Gamma N\ar{d}{\pi}\\
	N/\Gamma
\end{tikzcd}\]
with fiber $\tilde{M}$. Since $\tilde{M}$ is acyclic the Serre spectral sequence collapses on the second page. This implies that $H^*(M,\Z)\cong H^*(\tilde{M}\times_\Gamma N,\Z)\cong H^*(N/\Gamma,\Z)$ as we wanted to see.

\begin{cor}
	With the same assumptions as in \cref{free iterated actions rigidity hypernilmanifolds}.2, suppose also that $\pi_1(M)$ is virtually solvable. Then $M$ is homeomorphic to $N/\Gamma$.
\end{cor}

\begin{proof}
	Firstly, recall that $f:M\longrightarrow N/\Gamma$ induces a surjective map $f_*:\pi_1(M)\longrightarrow \Gamma$. Since $\pi_1(M)$ is virtually solvable there exists a finite covering $q:M'\longrightarrow M$ such that $\pi_1(M')$ is solvable. We have a commutative diagram\[
	\begin{tikzcd}
		M' \ar{d}{r}\ar{r}{f'} & N/\Gamma'\ar{d}{}\\
		M	\ar{r}{f} & N/\Gamma
	\end{tikzcd}\]
	
	where $f'_*:\pi_1(M')\longrightarrow \Gamma'$ is surjective. Notice that $\D_2(N/\Gamma')=\D_2(N/\Gamma)$ and that $\D_2(M')\geq \D_2(M)$. Consequently, we have $\D_2(N/\Gamma')=\D_2(M')$. The next step is to prove that $M'$ is a nilmanifold.  
	
	The acyclic manifold $\tilde{M}'$ has solvable fundamental group. This implies, since $H_1(\tilde{M}',\Z)$ is trivial, that $\tilde{M}'$ is simply connected. Consequently, $\tilde{M}'$ is contractible by Hurewicz theorem (see \cite[Corollary 4.33]{hatcher2002algebraic}) and $M'$ is a closed connected aspherical manifold. The map $\pi:\tilde{M}' \times_{\Gamma'}N\longrightarrow N/\Gamma'$ is a homotopy equivalence. Since $\Theta$ is also a homotopy equivalence we can construct a homotopy equivalence $M'\longrightarrow N/\Gamma'$. Since the Borel conjecture is true for nilmanifolds, we can conclude that $M'$ is homeomorphic to $N/\Gamma$.
	
	Finally, $M$ is also a closed connected aspherical manifold, since the finite covering $r:M'\longrightarrow M$ is obtained via pullback by $f$ of the covering of nilmanifolds $N/\Gamma'\longrightarrow N/\Gamma$, we have $\pi_1(M)\cong \Gamma$. Thus, $M$ is homeomorphic to $N/\Gamma$, as we wanted to see.
\end{proof}

\begin{rem}
	Using the same argument as in \cref{iterated action rational cohomology}, we can remove the assumption that $f_*:\pi_1(M)\longrightarrow \Gamma$ is surjective. In this case, if $\pi_1(M)$ is virtually solvable then $M$ is homeomorphic to a nilmanifold and $\pi_1(M)$ is commensurable to $\Gamma$.
\end{rem}

\selectlanguage{english}

\begin{thebibliography}{CPS21}
	
	\bibitem[AD02]{adem2002topics}
	Alejandro Adem and James~F Davis.
	\newblock Topics in transformation groups.
	\newblock {\em Handbook of geometric topology}, pages 1--54, 2002.
	
	\bibitem[BBS01]{baker2001towers}
	Mark Baker, Michel Boileau, and Wang Shicheng.
	\newblock Towers of covers of hyperbolic 3-manifolds.
	\newblock {\em Rend.Istit. Mat. Univ. Trieste}, 32, 35–43, 2001.
	
	\bibitem[Bel88]{bell1988notes}
	Allen~D Bell.
	\newblock {\em Notes on localization in noncommutative Noetherian rings}.
	\newblock Departamento de {\'A}lgebra y Fundamentos, Universidad de Granada,
	1988.
	
	\bibitem[Bel20]{belegradek2020iterated}
	Igor Belegradek.
	\newblock Iterated circle bundles and infranilmanifolds.
	\newblock {\em Osaka Journal of Mathematics}, 57(1):165--168, 2020.
	
	\bibitem[Bir21]{birkar2016singularities}
	Caucher Birkar.
	\newblock Singularities of linear systems and boundedness of Fano varieties.
	\newblock {\em Annals of Mathematics},193(2), 347-405, 2021.
	
	\bibitem[BK23]{baues2023isometry}
	Oliver Baues and Yoshinobu Kamishima.
	\newblock Isometry groups with radical, and aspherical Riemannian manifolds
	with large symmetry, i.
	\newblock {\em Geometry \& Topology}, 27(1):1--50, 2023.
	
	\bibitem[Blo75]{bloomberg1975manifolds}
	Edward~M Bloomberg.
	\newblock Manifolds with no periodic homeomorphisms.
	\newblock {\em Transactions of the American Mathematical Society}, 202:67--78,
	1975.
	
	\bibitem[Bor83]{borel1983periodic}
	Armand Borel.
	\newblock On periodic maps of certain $K(\pi,1)$.
	\newblock {\em Collected Papers III}, pages 57--60, 1983.
	
	\bibitem[Bre72]{bredon1972introduction}
	Glen~E Bredon.
	\newblock {\em Introduction to compact transformation groups}.
	\newblock Academic press, 1972.
	
	\bibitem[Bro12]{brown2012cohomology}
	Kenneth~S Brown.
	\newblock {\em Cohomology of groups}, volume~87.
	\newblock Springer Science \& Business Media, 2012.
	
	\bibitem[BZ24]{bandman2024jordan}
	Tatiana Bandman and Yuri~G Zarhin.
	\newblock Jordan groups and geometric properties of manifolds.
	\newblock {\em Arnold Mathematical Journal}, pages 1--15, 2024.
	
	\bibitem[Cha12]{charlap2012bieberbach}
	Leonard~S Charlap.
	\newblock {\em Bieberbach groups and flat manifolds}.
	\newblock Springer Science \& Business Media, 2012.
	
	\bibitem[CJ97]{cairns1997new}
	Grant Cairns and Barry Jessup.
	\newblock New bounds on the Betti numbers of nilpotent Lie algebras.
	\newblock {\em Communications in algebra}, 25(2):415--430, 1997.
	
	\bibitem[CJP97]{cairns1997betti}
	Grant Cairns, Barry Jessup, and Jane Pitkethly.
	\newblock On the Betti numbers of nilpotent Lie algebras of small dimension.
	\newblock In {\em Integrable Systems and Foliations: Feuilletages et
		Syst{\`e}mes Int{\'e}grables}, pages 19--31. Springer, 1997.
	
	\bibitem[CL83]{collins1983automorphisms}
	Donald~J Collins and Frank Levin.
	\newblock Automorphisms and Hopficity of certain Baumslag-Solitar groups.
	\newblock {\em Archiv der Mathematik}, 40(1):385--400, 1983.
		
	\bibitem[CMiRPS21]{csikos2021number}
	Bal{\'a}zs Csik{\'o}s, Ignasi Mundet i Riera, L{\'a}szl{\'o} Pyber, Endre Szab{\'o}.
	\newblock On the number of stabilizer subgroups in a finite group acting on a
	manifold.
	\newblock {\em arXiv preprint arXiv:2111.14450}, 2021.
	
	\bibitem[CPS14]{csikos2014diffeomorphism}
	Bal{\'a}zs Csik{\'o}s, L{\'a}szl{\'o} Pyber, and Endre Szab{\'o}.
	\newblock Diffeomorphism groups of compact 4-manifolds are not always Jordan.
	\newblock {\em arXiv preprint arXiv:1411.7524}, 2014.
		
	\bibitem[CPS22]{csikos2022finite}
	Bal{\'a}zs Csik{\'o}s, L{\'a}szl{\'o} Pyber, and Endre Szab{\'o}.
	\newblock Finite subgroups of the homeomorphism group of a compact topological
	manifold are almost nilpotent.
	\newblock {\em arXiv preprint arXiv:2204.13375}, 2022.
	
	\bibitem[Dek16]{dekimpe2016users}
	Karel Dekimpe.
	\newblock A users' guide to infra-nilmanifolds and almost-Bieberbach groups.
	\newblock {\em arXiv preprint arXiv:1603.07654}, 2016.
	
	\bibitem[DS82]{DONNELLY1982443}
	Harold Donnelly and Reinhard Schultz.
	\newblock Compact group actions and maps into aspherical manifolds.
	\newblock {\em Topology}, 21(4):443--455, 1982.
	
	\bibitem[DS88]{deninger1988cohomology}
	Christopher Deninger and Wilhelm Singhof.
	\newblock On the cohomology of nilpotent Lie algebras.
	\newblock {\em Bulletin de la Soci{\'e}t{\'e} Math{\'e}matique de France},
	116(1):3--14, 1988.
	
	\bibitem[DS25]{daura2024actions}
	Jordi Daura~Serrano.
	\newblock Actions of large finite groups on aspherical manifolds.
	\newblock {\em International Mathematics Research Notices}, 2025(10), 2025.
	
	\bibitem[FOM12]{figueroa2012half}
	Jos{\'e} Figueroa-O’Farrill and Paul~de Medeiros.
	\newblock Half-bps M2-brane orbifolds.
	\newblock {\em Advances in Theoretical and Mathematical Physics},
	16(5):1349--1409, 2012.
	
	\bibitem[FOT08]{felix2008algebraic}
	Yves F{\'e}lix, John Oprea, and Daniel Tanr{\'e}.
	\newblock {\em Algebraic models in geometry}.
	\newblock OUP Oxford, 2008.
	
	\bibitem[GK13]{goze2013nilpotent}
	Michel Goze and Yusupdjan Khakimdjanov.
	\newblock {\em Nilpotent Lie Algebras}, volume 361.
	\newblock Springer Science \& Business Media, 2013.
	
	\bibitem[GLO85]{Gottlieb1985}
	Daniel~H. Gottlieb, Kyung~B. Lee, and Murad Ozaydin.
	\newblock Compact group actions and maps into $K(\pi,1)$-spaces.
	\newblock {\em Transactions of the American Mathematical Society},
	287(1):419--429, 1985.
	
	\bibitem[Gol23]{golota2023finite}
	Aleksei Golota.
	\newblock Finite groups acting on compact complex parallelizable manifolds.
	\newblock {\em International Mathematics Research Notices,} 2024(8), 6447-6470, 2024.
	
	\bibitem[Gro02]{grove2002geometry}
	Karsten Grove.
	\newblock Geometry of, and via, symmetries.
	\newblock {\em University Lecture Series-American Mathematical Society},
	27:31--51, 2002.
	
	\bibitem[GW86]{gordon1986spectrum}
	Carolyn~S Gordon and Edward~N Wilson.
	\newblock The spectrum of the Laplacian on Riemannian Heisenberg manifolds.
	\newblock {\em Michigan Mathematical Journal}, 33(2):253--271, 1986.
	
	\bibitem[GW04]{goodearl2004introduction}
	Kenneth~R Goodearl and Robert~B Warfield.
	\newblock {\em An introduction to noncommutative Noetherian rings}.
	\newblock Cambridge university press, 2004.
	
	\bibitem[Ham08]{hamrouni2008discrete}
	Hatem Hamrouni.
	\newblock Discrete cocompact subgroups of the generic filiform nilpotent Lie
	groups.
	\newblock {\em J. Lie Theory}, 18(28):1--16, 2008.
	
	\bibitem[Han09]{hanke2009stable}
	Bernhard Hanke.
	\newblock The stable free rank of symmetry of products of spheres.
	\newblock {\em Inventiones mathematicae}, 178:265--298, 2009.
	
	\bibitem[Hat02]{hatcher2002algebraic}
	Allen Hatcher.
	\newblock {\em Algebraic Topology}.
	\newblock Cambridge University Press, 2002.
	
	\bibitem[Hsi12]{hsiang2012cohomology}
	Wu~Yi Hsiang.
	\newblock {\em Cohomology theory of topological transformation groups},
	volume~85.
	\newblock Springer Science \& Business Media, 2012.
	
	\bibitem[KK83]{ku1983group}
	Hs{\"u}~Tung Ku and Mei~Chin Ku.
	\newblock Group actions on aspherical $A_k(n)$-manifolds.
	\newblock {\em Transactions of the American Mathematical Society},
	278(2):841--859, 1983.
	
	\bibitem[Lam91]{lam1991first}
	Tsit-Yuem~Lam.
	\newblock {\em A First Course in Noncommutative Rings}, volume 131 of {\em
		Graduate Texts in Mathematics/Springer-Verlag}.
	\newblock 1991.
	
	\bibitem[Lev07]{levitt2007automorphism}
	Gilbert Levitt.
	\newblock On the automorphism group of generalized Baumslag--Solitar groups.
	\newblock {\em Geometry \& Topology}, 11(1):473--515, 2007.
	
	\bibitem[LMZ25]{luo2025jordan}
	Yujie Luo, Sheng Meng, and De-Qi Zhang.
	\newblock Jordan property for automorphism groups of compact varieties.
	\newblock {\em arXiv preprint arXiv:2502.16956}, 2025.
	
	\bibitem[LR87]{lee1987manifolds}
	Kyung~Bai Lee and Frank Raymond.
	\newblock Manifolds on which only tori can act.
	\newblock {\em Transactions of the American Mathematical Society},
	304(2):487--499, 1987.
	
	\bibitem[LR10]{lee2010seifert}
	Kyung~Bai Lee and Frank Raymond.
	\newblock {\em Seifert fiberings}.
	\newblock Number 166. American Mathematical Soc., 2010.
	
	\bibitem[LS77]{lyndon1977combinatorial}
	Roger~C Lyndon, Paul~E Schupp.
	\newblock {\em Combinatorial group theory}, volume 188.
	\newblock Springer, 1977.
			
	\bibitem[MiR10]{IgnasiMundetiRiera2010Jtft}
	Ignasi~Mundet i~Riera.
	\newblock Jordan's theorem for the diffeomorphism group of some manifolds.
	\newblock {\em Proceedings of the American Mathematical Society},
	138(6):2253--2262, 2010.
	
	\bibitem[MiR17]{riera2017non}
	Ignasi Mundet~i Riera.
	\newblock Non Jordan groups of diffeomorphisms and actions of compact Lie
	groups on manifolds.
	\newblock {\em Transformation Groups}, 22(2):487--501, 2017.
	
	\bibitem[MiR24a]{mundet2021topological}
	Ignasi Mundet~i Riera.
	\newblock Discrete degree of symmetry of manifolds.
	\newblock {\em Transformation Groups}, pages 1--38, 2024.
	
	\bibitem[MiR24b]{riera2023actions}
	Ignasi Mundet~i Riera.
	\newblock Actions of large finite groups on manifolds.
	\newblock {\em International Journal of Mathematics}, 2441012, 2024.
		
	\bibitem[MiR24c]{riera2024jordan}
	Ignasi Mundet~i Riera.
	\newblock Jordan property for homeomorphism groups and almost fixed point
	property.
	\newblock {\em Publicacions Matem{\`a}tiques}, 68(2):545--557, 2024.
	
	\bibitem[MS63]{mann1963actions}
	Larry~N Mann and Jin-Chen Su.
	\newblock Actions of elementary p-groups on manifolds.
	\newblock {\em Transactions of the American Mathematical Society},
	106(1):115--126, 1963.
	
	\bibitem[Pop11]{popov2011makar}
	Vladimir~L Popov.
	\newblock On the Makar-Limanov, Derksen invariants, and finite automorphism
	groups of algebraic varieties.
	\newblock {\em Affine Algebraic Geometry: The Russell Festschrift, CRM
		Proceedings and Lecture Notes}, 54:289--311, 2011.
	
	\bibitem[Pop18]{popov2018jordan}
	Vladimir~L Popov.
	\newblock The Jordan property for Lie groups and automorphism groups of complex
	spaces.
	\newblock {\em Mathematical Notes}, 103:811--819, 2018.
	
	\bibitem[PS14]{prokhorov2014jordan}
	Yuri Prokhorov and Constantin Shramov.
	\newblock Jordan property for groups of birational selfmaps.
	\newblock {\em Compositio Mathematica}, 150(12):2054--2072, 2014.
	
	\bibitem[PS16]{prokhorov2016jordan}
	Yuri Prokhorov and Constantin Shramov.
	\newblock Jordan property for Cremona groups.
	\newblock {\em American Journal of Mathematics}, 138(2):403--418, 2016.
	
	\bibitem[Pup07]{puppe2007manifolds}
	Volker Puppe.
	\newblock Do manifolds have little symmetry?
	\newblock {\em Journal of Fixed Point Theory and Applications}, 2:85--96, 2007.
	
	\bibitem[QSW24]{qin2021self}
	Lizhen Qin, Yang Su, and Botong Wang.
	\newblock Self-covering, finiteness, and fibering over a circle.
	\newblock {\em Transactions of the American Mathematical Society,} 377(03), 1883-1914, 2024.
		
	\bibitem[Rem17]{remm2017filiform}
	Elisabeth Remm.
	\newblock On filiform Lie algebras. Geometric and algebraic studies.
	\newblock In {\em 13th International Workshop on Differential Geometry and its
		Applications}, volume~63, pages 179--209. Editura Academiei Rom{\^a}ne, 2017.
	
	\bibitem[SC19]{saez2019finite}
	Carles S{\'a}ez~Calvo.
	\newblock Finite groups acting on smooth and symplectic 4-manifolds.
	\newblock PhD Thesis, Universitat de Barcelona, 2019.
	
	\bibitem[Sch81]{schultz1981group}
	Reinhard Schultz.
	\newblock Group actions on hypertoral manifolds. II.
	\newblock  {\em Journal für die reine und angewandte Mathematik}, 329, 75-86, 1981.
	
	\bibitem[Ser09]{serre2009minkowski}
	Jean-Pierre Serre.
	\newblock A Minkowski-style bound for the orders of the finite subgroups of the
	Cremona group of rank 2 over an arbitrary field.
	\newblock {\em Moscow Mathematical Journal}, 9(1):183--198, 2009.
	
	\bibitem[Ser10]{serre2010bounds}
	Jean-Pierre Serre.
	\newblock Bounds for the orders of the finite subgroups of $G(k)$.
	\newblock {\em arXiv preprint arXiv:1011.0346}, 2010.
	
	\bibitem[SY79]{schoen1979compact}
	Richard Schoen and Shing~Tung Yau.
	\newblock Compact group actions and the topology of manifolds with non-positive
	curvature.
	\newblock {\em Topology}, 18(4):361--380, 1979.
	
	\bibitem[Sza19]{szabo2019special}
	D{\'a}vid~R Szab{\'o}.
	\newblock Special $ p $-groups acting on compact manifolds.
	\newblock {\em arXiv preprint arXiv:1901.07319}, 2019.
	
	\bibitem[Sza23]{szabo2023constructing}
	D{\'a}vid~R Szab{\'o}.
	\newblock Constructing highly symmetric compact manifolds and algebraic
	varieties.
	\newblock {\em arXiv preprint arXiv:2304.10366}, 2023.
	
	\bibitem[vL18]{van2018towers}
	Wouter van~Limbeek.
	\newblock Towers of regular self-covers and linear endomorphisms of tori.
	\newblock {\em Geometry \& Topology}, 22(4):2427--2464, 2018.
	
	\bibitem[vL21]{van2021structure}
	Wouter van Limbeek.
	\newblock Structure of normally and finitely non-co-Hopfian groups.
	\newblock {\em Groups, Geometry, and Dynamics}, 15(2):465--489, 2021.
	
	\bibitem[Weh94]{Wehrfritz1994}
	Bert~A.~F. Wehrfritz.
	\newblock {Two Remarks on Polycyclic Groups}.
	\newblock {\em Bulletin of the London Mathematical Society}, 26(6):543--548, 11
	1994.
	
	\bibitem[WW83]{washiyama1983degree}
	Ryo Washiyama and Tsuyoshi Watabe.
	\newblock On the degree of symmetry of a certain manifold.
	\newblock {\em Journal of the Mathematical Society of Japan}, 35(1):53--58,
	1983.
	
	\bibitem[Zar14]{zarhin2014theta}
	Yuri~G Zarhin.
	\newblock Theta groups and products of abelian and rational varieties.
	\newblock {\em Proceedings of the Edinburgh Mathematical Society},
	57(1):299--304, 2014.
	
\end{thebibliography}

\end{document}